\documentclass[11pt,a4paper]{amsart}
\usepackage{amsmath,amssymb, amsbsy}
\usepackage{subfigure}
\usepackage{graphpap,latexsym,epsf}
\usepackage{color,psfrag}
\usepackage[dvips]{graphicx}

\usepackage[english]{babel}
\usepackage{times}
\usepackage{graphicx}
\usepackage{amscd}
\usepackage{amsmath}
\usepackage{amsfonts}
\usepackage{amssymb}
\usepackage{amsthm}
\usepackage{latexsym}
\newtheorem{theorem}{Theorem}[section]
\newtheorem{proposition}[theorem]{Proposition}
\newtheorem{lemma}[theorem]{Lemma}
\newtheorem{corollary}{Corollary}[section]
\newtheorem{open}{Open Problem}[section]
\newtheorem{definition}[theorem]{Definition}

\newtheorem{remark}{Remark}

\def\neweq#1{\begin{equation}\label{#1}}
\def\endeq{\end{equation}}

\newcommand{\R}{\mathbb{R}}
\newcommand{\N}{\mathbb{N}}
\newcommand{\eps}{\varepsilon}

\newcommand{\EC}{\overset{E}{\longrightarrow}}
\newcommand{\EEC}{\overset{EE}{\longrightarrow}}
\newcommand{\CC}{\overset{C}{\longrightarrow}}
\newcommand{\ECDUAL}{\overset{E^*}{\longrightarrow}}
\newcommand{\EECDUAL}{\overset{EE^*}{\longrightarrow}}
\newcommand{\CCDUAL}{\overset{C^*}{\longrightarrow}}

\begin{document}

\title{Spectral stability for a class of fourth order Steklov problems under domain perturbations}

\author{Alberto Ferrero, Pier Domenico Lamberti}

\address{\hbox{\parbox{5.7in}{\medskip\noindent{Alberto Ferrero, \\
Universit\`a del Piemonte Orientale, \\
        Dipartimento di Scienze e Innovazione Tecnologica, \\
        Viale Teresa Michel 11, 15121 Alessandria, Italy. \\[5pt]
Pier Domenico Lamberti, \\
Universit\`a di Padova, \\
Dipartimento di Matematica `Tullio Levi-Civita', \\
Via Trieste 63, 35121 Padova, Italy. \\[3pt]
        \em{E-mail addresses: }{\tt alberto.ferrero@uniupo.it, lamberti@math.unipd.it}}}}}

\date{\today}

\maketitle

\begin{abstract} We study the spectral stability of two
fourth order Steklov problems upon domain perturbation.
One of the two problems is  the classical DBS - Dirichlet Biharmonic Steklov - problem, the other one is  a
variant.  Under a comparatively weak condition on the convergence of the domains, we prove the stability of
the resolvent operators  for both problems, which implies the stability of
eigenvalues and    eigenfunctions. The stability estimates for the eigenfunctions are expressed in terms
of the strong $H^2$-norms. The analysis is carried out without assuming that the domains are star-shaped.
Our condition turns out to be sharp at least for the variant of the DBS problem. In the case of the   DBS problem,
we prove stability of a suitable Dirichlet-to-Neumann type map under very weak conditions on the convergence of the domains and we formulate an open problem. As bypass product of our analysis, we provide some stability and instability results for Navier and Navier-type boundary value problems for the
biharmonic operator.
\end{abstract}

\vspace{11pt}


\vspace{6pt}

\section{Introduction} \label{s:introduction}
In this paper we consider a class of {\it Steklov problems} and
related spectral properties. By Steklov problems we mean a class
of elliptic problems with parameters in the boundary conditions.
This denomination is motivated by \cite{Steklov} where the classical Steklov problem
\begin{equation} \label{eq:classicalSteklov}
\begin{cases}
\Delta u=0, & \qquad \text{in } \Omega \, , \\
u_{\nu }=\lambda u, & \qquad \text{on } \partial\Omega \, , \\
\end{cases}
\end{equation}
was introduced.  We refer to the recent paper \cite{girpol} for a survey on this subject. Here we are
interested in the study of a class of elliptic equations with the
biharmonic operator subject to Steklov type boundary conditions. Such
kind of conditions for the biharmonic operator can be already
found in the following works \cite{Fichera, Kuttler1,Kuttler2,KuSi,Payne}.
In those papers the attention is mainly focused on the
first eigenvalue of the fourth order Steklov problem and its
dependence on the domain. To be more clear, let us introduce the
classical DBS - Dirichlet Biharmonic Steklov - problem.
Let $\Omega\subset \R^N$ ($N\ge 2$) be a bounded domain (i.e., connected open set) with sufficiently regular boundary. Consider the fourth order problem
\begin{equation} \label{eq:Steklov}
\begin{cases}
\Delta^2 u=0, & \qquad \text{in } \Omega \, , \\
u=0, & \qquad \text{on } \partial\Omega \, , \\
\Delta u-d u_\nu=0, & \qquad \text{on } \partial\Omega \, ,
\end{cases}
\end{equation}
where $u_\nu$ denotes the normal derivative on the boundary and $d$ is a real parameter.
By a solution of \eqref{eq:Steklov} we mean a function $u\in H^2(\Omega)\cap H^1_0(\Omega)$ such that
\begin{equation} \label{eq:Steklov-var}
\int_\Omega \Delta u \Delta v \, dx=d\int_{\partial\Omega} u_\nu v_\nu \, dS \, ,
\qquad \text{for any } v\in H^2(\Omega)\cap H^1_0(\Omega) \, .
\end{equation}
Clearly \eqref{eq:Steklov} is a linear homogeneous problem and
hence it always admits the trivial solution. Therefore, it appears
reasonable to interpret \eqref{eq:Steklov} as  an eigenvalue
problem with respect to the parameter $d$: we say that $d$ is an
eigenvalue for \eqref{eq:Steklov} if \eqref{eq:Steklov} admits a
nontrivial solution in the sense given in \eqref{eq:Steklov-var}.
Such a nontrivial solution is called eigenfunction for $d$. In the
sequel, we will refer to \eqref{eq:Steklov} as the {\it Steklov
problem}. It is well-known that the eigenvalues of
\eqref{eq:Steklov} are strictly positive and that the first
eigenvalue $d_1$ admits the following variational characterization
\begin{equation} \label{eq:d1}
d_1=\inf_{v\in (H^2(\Omega)\cap H^1_0(\Omega))\setminus H^2_0(\Omega )} \frac{\int_\Omega |\Delta v|^2 \, dx}{\int_{\partial\Omega}
v_\nu^2 \, dS} \, .
\end{equation}
In \cite{KuSi,Payne} the authors study the isoperimetric
properties of $d_1$ and in \cite{Fichera}  it is proved   that $d_1$ is the sharp constant for $L^2$ a priori
estimates for solutions of the (second order) Laplace equation
under nonhomogeneous Dirichlet boundary conditions, see also \cite{Kuttler1,Kuttler2}.

From a physical point of view, it is known that in the two
dimensional case $N=2$, Steklov boundary conditions for the
biharmonic operator have a natural interpretation in modeling the
vibrations of a hinged plate. For more details see for example
\cite{BuGa, BuPr, FeGaWe,GaGrSw,GaSw, LaPr}.

In the last years, the Steklov boundary conditions for the
biharmonic operator were studied in \cite{BeGaMi,GaSw} from the
point of view of the positivity preserving property for
nonhomogeneous biharmonic equations. Other recent results were
focused on isoperimetric properties of $d_1$ and on related shape
optimization problems, see \cite{AnGa,BuFeGa,BuGa,FeGaWe}. More
precisely from these papers, among many other results, we know
that, among all convex domains of fixed volume, there exists an
optimal domain which minimizes $d_1$ and that this domain is not
the ball: for example in dimension $N=2$ the square is better than
the disk. It is still an open question, even in dimension $N=2$,
which is the optimal shape among convex domains of fixed volume.
We also quote the recent paper \cite{Au} where the eigenvalues of \eqref{eq:Steklov}
are used to provide the singular value decomposition for the Poisson kernel of the Laplace operator on bounded domains.

For its physical interpretation and for its relevance in the study
of positivity preserving properties, in our paper we also consider
the following Steklov type problem
\begin{equation} \label{eq:Steklov-modificato}
\begin{cases}
\Delta^2 u=0, & \qquad \text{in } \Omega \, , \\
u=0,  & \qquad \text{on } \partial\Omega \, , \\
\Delta u-K(x)u_\nu-\delta u_\nu=0,  & \qquad \text{on } \partial\Omega \, ,
\end{cases}
\end{equation}
where $K$ denotes the mean curvature of the boundary, i.e. the sum
of the principal curvatures and $\delta$ is a real parameter. As
one can see for example from \cite{GaSw}, when $N=2$, the
curvature of the boundary naturally comes out when we consider the
model of a hinged plate. Problem \eqref{eq:Steklov-modificato}
admits the following variational formulation: we say that $u\in H^2(\Omega)\cap H^1_0(\Omega)$ is a weak solution
of \eqref{eq:Steklov-modificato} if
\begin{equation} \label{eq:Steklov-modificato-var}
\int_\Omega D^2 u: D^2 v\, dx=\delta \int_{\partial\Omega} u_\nu
v_\nu \, dS \qquad \text{for any } v\in H^2(\Omega)\cap
H^1_0(\Omega)
\end{equation}
where $D^2 u$, $D^2 v$ denote the Hessian matrices
of $u,v$ respectively and $D^2 u\!:\!D^2v\!:=\!\sum_{i,j=1}^N \!\frac{\partial^2
u}{\partial x_i \partial x_j} \frac{\partial^2 v}{\partial x_i
\partial x_j}$. In the sequel we also write $|D^2 u|^2:=\sum_{i,j=1}^N \left(\frac{\partial^2
u}{\partial x_i \partial x_j}\right)^2$. For more details on the
correlation between \eqref{eq:Steklov-modificato} and
\eqref{eq:Steklov-modificato-var} see the proof of
\cite[Proposition 5]{Chasman} and combine it with the identity
$\Delta u=u_{\nu\nu}+K(x)u_\nu$ on $\partial\Omega$ (see \cite[(4.68), p. 62]{Sperb})
valid for any smooth enough function $u$ vanishing on
$\partial\Omega$. Here $u_{\nu\nu}$ denotes  the second order normal
derivative, i.e. $u_{\nu\nu}=\sum_{i,j=1}^N \tfrac{\partial^2
u}{\partial x_i
\partial x_j}\nu_i \nu_j$ with $\nu=(\nu_1,\dots,\nu_N)$.

Similarly to problem \eqref{eq:Steklov}, also problem
\eqref{eq:Steklov-modificato} may be interpreted as an eigenvalue
problem with respect to the parameter $\delta$: we say that
$\delta$ is an eigenvalue of \eqref{eq:Steklov-modificato} if
\eqref{eq:Steklov-modificato} admits a nontrivial solution in the
sense given in \eqref{eq:Steklov-modificato-var}. In the sequel,
we refer to \eqref{eq:Steklov-modificato} as the {\it modified
Steklov problem}.

It is known that the set of the eigenvalues of \eqref{eq:Steklov}
is a countable set of isolated positive numbers which may be ordered in an
increasing sequence diverging to $+\infty$. We choose the usual notation of
repeating each eigenvalue as many times as its multiplicity:
\begin{equation} \label{eq:sequence}
0<d_1<d_2\le \dots \le d_n \le \dots
\end{equation}
As one can see from \eqref{eq:sequence} we wrote $d_1<d_2$ in
order to emphasize the fact that the first eigenvalue $d_1$ is
simple. We quote the papers \cite{BuFeGa,FeGaWe} for more details
on these and further properties of the eigenvalues of
\eqref{eq:Steklov}.

In a completely similar way, it can be proved that also the set of
the eigenvalues of \eqref{eq:Steklov-modificato} is a countable
set of isolated numbers which may be ordered in an increasing
sequence diverging to $+\infty$ and moreover the least eigenvalue
is strictly positive. For the eigenvalues of
\eqref{eq:Steklov-modificato} we use a notation analogous to
\eqref{eq:sequence}:
\begin{equation} \label{eq:sequence-modificato}
0<\delta_1\le \delta_2\le \dots \le \delta_n \le \dots
\end{equation}

Here we are interested in the study of the stability of the
eigenvalues and eigenfunctions of \eqref{eq:Steklov} and
\eqref{eq:Steklov-modificato} with respect to domain perturbation.
We note that the stability of the eigenvalues of  the classical second order problem \eqref{eq:classicalSteklov} has been discussed in
\cite{Bo} as  well as in \cite{terelst} which is actually concerned with   the stability of the so-called Dirichlet-to-Neumann maps.

More precisely if $\{\Omega_\eps\}_{0<\eps\le \eps_0}$ is a family
of domains which converges to a domain $\Omega$ in a suitable
sense, denoting  by $\{d_n^\eps\}_{n=1}^\infty$,
$\{d_n\}_{n=1}^\infty$ the eigenvalues of the Steklov problem in
$\Omega_\eps$ and $\Omega$ respectively, and by
$\{\delta_n^\eps\}_{n=1}^\infty$ $\{\delta_n\}_{n=1}^\infty$ the
eigenvalues of the modified Steklov problem in $\Omega_\eps$ and
$\Omega$ respectively, we say that we have stability for the eigenvalues of each of the two
problems, if for any $n\ge 1$, $d_n^\eps\to d_n$ and
$\delta_n^\eps \to \delta_n$ as $\eps\to 0$. We are also
interested in what we will call spectral convergence of
$S_{\Delta,\eps}$ to $S_\Delta$ and of $S_{D^2,\eps}$ to
$S_{D^2}$, where $S_{\Delta,\eps}$, $S_\Delta$ denote the linear
operators associated to \eqref{eq:Steklov} respectively in
$\Omega_\eps$ and $\Omega$, and $S_{D^2,\eps}$, $S_{D^2}$ denote
the linear operators associated to \eqref{eq:Steklov-modificato}
respectively in $\Omega_\eps$ and $\Omega$.
See Section \ref{s:functional-setting-B} for the precise
definitions of these operators.

Basically, these operators are the resolvent operators associated with problems  \eqref{eq:Steklov}, \eqref{eq:Steklov-modificato} and their spectral convergence is obtained
by proving their compact convergence which is a well-known
notion of convergence  implying the convergence
of eigenvalues and eigenfunctions. Unfortunately, since the underlying function spaces depend on $\eps $, we cannot directly use
the standard notion of compact convergence, but we need to use suitable `connecting systems'
which allow to pass from the varying Hilbert spaces defined on $\Omega_{\eps}$ to the limiting fixed Hilbert space defined on $\Omega$.  Namely, we use  a number of notions and results which go back to the works of F. Stummel~\cite{Stu} and G. Vainniko~\cite{Vainikko} and which have been further implemented in \cite{ArCaLo, CarPisk}. See also the recent paper \cite{Bogli}. In particular we use the notion of $E$-compact convergence. These notions are discussed in detail in Section~ \ref{ss:convergenceint}.

In order to prove the $E$-compact convergence of the operators under consideration, we consider domains $\Omega_{\eps}$, $\Omega $ belonging to
 uniform classes of domains with $C^{1,1}$ boundaries - see Definition~\ref{d:atlas-B} - and we  require that the boundaries of $\Omega_{\eps}$  converge to the boundary of $\Omega$
in the sense of  \eqref{eq:assumptions}. Condition \eqref{eq:assumptions} has been introduced in \cite{ArLa2} and turns out to be a sharp condition in the analysis of domain perturbation problems for fourth order operators. Note that if the boundaries of $\Omega_{\eps}$ have uniformly bounded $C^{1,1}$-norms and  converge to the boundary  of $\Omega$
in $C^1$ then condition \eqref{eq:assumptions}  is satisfied. However, it is important to note that condition \eqref{eq:assumptions} allows to deal also with boundaries whose $C^{1,1}$-norms are not uniformly bounded with respect to $\eps$.

Condition \eqref{eq:assumptions} allows to  construct in Section~\ref{subsecoperatorsE} a linear continuous  operator $E_{\eps}:H^2(\Omega)\cap H^{1}_0(\Omega )\to H^2(\Omega_{\eps})\cap H^{1}_0(\Omega_{\eps } )$ and to define the notion of $E$-convergence for a sequence $u_{\eps }\in H^2(\Omega_{\eps})\cap H^{1}(\Omega_{\eps} )$ to a function $u\in H^2 (\Omega )\cap H^{1}_0(\Omega )$, that is
$$
\|  u_{\eps }-E_{\eps}u\|_{H^2(\Omega _{\eps })}\to 0, \ \ {\rm as}\ \eps \to 0.
$$

The operators $E_{\eps}$  are  constructed by pasting together suitable pull-back operators defined by means of  appropriate local diffeomorphisms. Importantly,
the operators $E_{\eps}$ enjoy the following property: for any
$\eps >0$ there exists an open set  $K_{\eps }\subset \Omega\cap
\Omega _{\eps }$ such that $(E_{\eps}u)(x)=u(x)$ for all $x\in
K_{\eps}$ and such that $|(\Omega_{\eps}\cup \Omega ) \setminus
K_\eps |\to 0$ as $\eps \to 0$. Here and in the sequel $|A|$ denotes the Lebesgue measure of any measurable set $A\subset \R^N$.
This implies the more familiar strong convergence
\begin{equation}
\label{intersection}
\|  u_{\eps }-u\|_{H^2(\Omega _{\eps } \cap \Omega )  }\to 0, \ \ {\rm as}\ \eps \to 0 .
\end{equation}
Moreover, it also implies
\begin{equation}
\label{intesection1}
\|  \nabla u_{\eps }-E_{\eps}\nabla u\|_{L^2(\partial \Omega _{\eps })}\to 0, \ \ {\rm as}\ \eps \to 0,
\end{equation}
and we note that $E_{\eps }\nabla u_{|\partial \Omega_{\eps} }$ is nothing but a natural transplantation of $u$ from $\partial \Omega $ to
$\partial \Omega_{\eps }$
(which differs from an exact pull-back by a minor deformation at the intersections of different local charts describing the boundaries, a deformation which vanishes as $\eps \to 0$). For more details on \eqref{intersection}-\eqref{intesection1} see Proposition \ref{intersec}.

One of the main results of the paper is Theorem~\ref{t:main-1} where it is proved that condition \eqref{eq:assumptions} implies the $E$-compact convergence of  $S_{\Delta,\eps}$ to $S_\Delta$ and of $S_{D^2,\eps}$ to
$S_{D^2}$ as $\eps\to 0$. On the base of the general Theorem~\ref{vaithm}, this implies the spectral convergence of   $S_{\Delta,\eps}$ to $S_\Delta$ and of $S_{D^2,\eps}$ to
$S_{D^2}$, hence the convergence of the eigenvalues and the $E$-convergence of the eigenfunctions in the sense of Theorem~\ref{vaithm}. In particular the eigenfunctions converge in the sense of \eqref{intersection} and \eqref{intesection1}.

We note that our domains are not required to be star-shaped as in
\cite{terelst} (which is concerned with  second order operators)
and that removing such an assumption is not straightforward. In
fact, here we only assume that the domains under consideration
belong to a uniform class of  domains defined by means of the same
atlas and  this leads to a number of technical difficulties  that
are overcome in Section 3. We believe that Section 3 gives a
partial answer (in the case of fourth order operators)  to the
question raised in the introduction of \cite{terelst}  where the
authors suggest the problem of studying the convergence of the
resolvent operators of second order Steklov problems on arbitrary
Lipschitz domains.

In Section~\ref{s:optimality} we analyze the optimality of condition \eqref{eq:assumptions}. In the spirit  of \cite{ArLa2}, we consider the case of a domain
of the form $\Omega=W\times (-1,0)$ where $W$ is a cuboid or a bounded domain in $\R^{N-1}$ of class $C^{1,1}$ and we assume that the perturbed domain
$\Omega_\eps$ is given by
\begin{equation*}
\Omega_\eps=\{(x',x_N):x'\in W\, ,
-1<x_N<g_\eps(x')\}
\end{equation*}
where $g_\eps(x')=\eps^\alpha b(x'/\eps)$ for any $x'\in W$ and
$b:\R^{N-1}\to [0,+\infty)$ is a nonconstant $Y$-periodic
function where $Y=\left(-\frac 12,\frac 12\right)^{N-1}$ is the unit cell in $\R^{N-1}$. We denote by $\Gamma_\eps$ and $\Gamma$ the sets
\begin{equation*}
\Gamma_\eps:=\{(x',g_\eps(x')):x'\in W\} \quad \text{and} \quad
\Gamma:=W\times \{0\} \, .
\end{equation*}

It appears that for $\alpha >3/2$ condition \eqref{eq:assumptions} is satisfied hence we have spectral convergence  of   $S_{\Delta,\eps}$ to $S_\Delta$ and of $S_{D^2,\eps}$ to
$S_{D^2}$, see Corollary~\ref{t:main-2ante}.  For $\alpha <3/2 $  condition  \eqref{eq:assumptions} is not satisfied and we have a degeneration phenomenon. Namely, for $1\le \alpha <3/2$ we have spectral convergence of the operator $S_{D^2,\eps}$  to  $S_{D^2,\Gamma }$ where $S_{D^2,\Gamma }$ is an operator encoding   Dirichlet boundary conditions on $\Gamma$, that is $u=u_{\nu}=0$ on $\partial \Gamma$, see Theorem~\ref{t:main-2}.  For $\alpha =3/2$, we have spectral convergence to another operator  which encodes the appearance of a strange term in the boundary conditions, see Theorem~\ref{thmcritstek}. We note that in order to simplify our analysis in the case $\alpha =3/2$,  we impose Dirichlet boundary conditions on $\partial \Omega_{\eps}\setminus \Gamma_{\eps}$ and $\partial\Omega\setminus \Gamma$.

We also note that imposing Dirichlet boundary conditions on $\partial \Omega_{\eps}\setminus \Gamma_{\eps}$ and $\partial\Omega\setminus \Gamma$ allows to view the  trichotomy discussed above in a rather transparent way. Namely, under Dirichlet boundary conditions on $\partial \Omega_{\eps}\setminus \Gamma_{\eps}$ and $\partial\Omega\setminus \Gamma$ we have that:
\begin{itemize}
\item[(i)] If $\alpha >3/2$ then  $\delta_n(\eps )\to \delta _n(0
)$, as $\eps \to 0$.
\item[(ii)] If $\alpha =3/2$ then
$\delta_n(\eps )\to  \delta_n(0 )+\gamma $ as $\eps \to 0$,
where $\gamma >0$ is the strange term defined by  \eqref{strangecurvature}.
 \item[(iii)] If
$1\le \alpha < 3/2$ then $\delta_n(\eps)\to \infty $, as $\eps
\to 0$.
\end{itemize}

The strange term $\gamma$ is the same term  found in \cite{ArLa2} in the analysis of the Babuska Paradox and is referred to as `the strange curvature' in \cite{ArLa2}.

The analysis of the cases $\alpha \le 3/2$ is carried out in Section~\ref{s:optimality} only for problem \eqref{eq:Steklov-modificato} since
problem \eqref{eq:Steklov} presents severe difficulties in these cases.  In order to overcome these difficulties, in Section~\ref{navtoneusection} we change our point of view
and we recast the eigenvalue problem \eqref{eq:Steklov} as an eigenvalue problem for a suitable Dirichlet-to-Neumann-type map. Considering
our boundary conditions, we find it natural to call such a map Neumann-to-Navier map, and  actually we find it simpler to consider its inverse which we call Navier-to-Neumann map and we denote by ${\mathcal{N}}_{\Omega }$. Since we aim at relaxing the boundary regularity assumptions, we are forced to work with the space $H(\Delta, \Omega)\cap H^1_0(\Omega ) $ and we find it convenient  to  consider
the normal derivatives of the functions in such a space as elements of $H^{-1/2}(\partial\Omega)$  although they belong to $L^2(\partial \Omega)$ see \cite[Thm.~3.2]{Au} (recall that the space $H(\Delta, \Omega) $ is the space of functions in $L^2(\Omega )$ with distributional Laplacian in $L^2(\Omega)$).  In this way, our   Navier-to-Neumann map ${\mathcal{N}}_{\Omega }$
can be  considered as a map from $H^1(\Omega)/H^1_0(\Omega )$ to its dual and problem \eqref{eq:Steklov} can be recast in the form
${\mathcal{N}}_{\Omega }f =d^{-1} J_0   f$ where $J_0$ is defined by $
\langle J_0 f,v\rangle:=\int_{\partial\Omega} fv \, dS$  for all $f,v\in H^1(\Omega)/H^1_0(\Omega )$. It is interesting to note that this weak  formulation of problem  \eqref{eq:Steklov} makes sense under the sole assumption that $\Omega$ is a bounded domain of class $C^{0,1}$.
We note that, in order to avoid some technicalities which would not add much to the results of Section~\ref{navtoneusection}, we shall actually avoid using the quotient spaces and we shall  consider the Navier-to-Neumann  ${\mathcal{N}}_{\Omega }$ directly as map from  $H^1(\Omega)$ to its dual, which has also its own interest,  see Remark~\ref{veryweakrem}.

Thus we study the behaviour of the maps ${\mathcal{N}}_{\Omega_{\eps} }$ as $\eps \to 0$ and
in Theorem~\ref{l:EE*} we prove their $EE^*$-convergence to ${\mathcal{N}}_{\Omega  }$ under weak conditions on the convergence of domains. Here the  $EE^*$-convergence  is a natural type of convergence for operators acting from a family of spaces to the corresponding duals and is based on two connecting systems $E_{\eps}$ and $E_{\eps}^*$ ($E_{\eps}$ is now defined  by means of a standard extension operator for the space $H^1(\Omega)$), see Definition~\ref{connectingdual}.  In Theorem~\ref{t:dual} we also prove a compactness result for the convergence of  ${\mathcal{N}}_{\Omega_{\eps} }$ under an additional assumption on the convergence of $\Omega_{\eps}$.

Unfortunately, the compact convergence result obtained in Theorem
\ref{t:dual} is not sufficient to prove a stability result for
the spectrum of the Steklov problem \eqref{eq:Steklov} under such very weak assumptions on the convergence of $\Omega_{\eps}$.
Indeed, in order to apply \cite[Theorem 6.1]{Vainikko}, the regular
convergence condition 1) in \cite[Section 6]{Vainikko} has to be
satisfied by the family of operators $\{\mathcal
N_{\Omega_\eps}-\lambda (J_0)_\eps\}_{\eps}$ for any $\lambda\in
\Lambda$, where $\Lambda$ is a suitable bounded set in the complex
plane, and this does not seem the case, not even if we consider these operators on  the quotient spaces and adjust the corresponding connecting systems.
Taking into account this issue we suggest the Open Problem \ref{o:1}.

As a bypass product of the analysis carried out in Section~~\ref{navtoneusection}, we prove
stability and instability results for Navier and Navier-type problems, see Theorems~\ref{stabnav} and \ref{trico}.

The paper is organized as follows. In Section~\ref{preliminaries} we discuss some preliminary results, in particular we discuss the notion of $E$-convergence and  our functional setting. In  Section~\ref{strongsec} we prove the main stability result of the paper under the assumption that condition \eqref{eq:assumptions} holds. In Section~\ref{s:optimality} we discuss the optimality of condition \eqref{eq:assumptions}. In Section~\ref{navtoneusection} we relax the assumptions on the convergence of the domains and we prove a stability result for the above mentioned Navier-to-Neumann map; moreover, we discuss the stability and instability of Navier and Navier-type problems.



\section{Preliminaries and notation} \label{preliminaries}

\subsection{A general approach to spectral stabilty} \label{ss:convergenceint}

The study of all spectral problems considered in this paper will
be reduced to the study of convenient non-negative  compact
self-adjoint operators in Hilbert spaces.

A natural way to do so is to consider appropriate modified versions
of the classical Dirichlet-to-Neumann map  densely defined in $L^2(\partial \Omega)$ and to consider their resolvents which are defined on the whole of  $L^2(\partial \Omega)$. For our fourth order Steklov problems, it is natural to talk about Neumann-to-Navier map and its inverse will be  called here  Navier-to-Neuman map.  Then one note that the  eigenvalues of those Navier-to-Neumann maps are in fact the reciprocals of the eigenvalues of our differential problems. One drawback of using Navier-to-Neumann maps consists in the fact that eventually they provide information concerning the convergence of the
(normal derivatives of the) eigenfunctions considered as elements of $L^2(\partial \Omega)$. In order to provide global information
about the behaviour of the eigenfunctions in $\Omega$,  in the regular cases  we shall actually consider appropriate realisations
of the resolvents in $H^2(\Omega)$ which will allow  to get stronger stability  results  involving  the norm of
$H^2(\Omega)$.

Thus,  our spectral problems will be considered from two different
points of view depending on the regularity of the domain
perturbations. In any case, both methods lead to the study of
suitable families  $\{B_\eps\}_{0<\eps\le \eps_0}$ of non-negative
compact self-adjoint operators defined in  Hilbert spaces
${\mathcal{H}}_{\eps}$ associated with the domains $\Omega_{\eps
}$. It is then important to study the convergence of those
operators $B_{\eps}$ as $\eps \to 0$.  To do so, we shall use the
notion of $E$-convergence  which we now recall.  We follow the
approach of \cite{Vainikko}, further developed in  \cite{ArCaLo}, \cite{CarPisk}.

In the spirit of  \cite{ArCaLo} we denote by
${\mathcal{H}}_\eps$ a family of Hilbert spaces for $\eps\in [0,\eps_0]$ and we
assume that there exists a family of linear operators
$E_\eps:{\mathcal{H}}_0\to {\mathcal{H}}_\eps$ such that
\begin{equation} \label{eq:norm-convergence}
\|E_\eps u\|_{{\mathcal{H}}_\eps} \overset{\eps\to 0}{\longrightarrow}
\|u\|_{{\mathcal{H}}_0} \, , \qquad \text{for all } u\in {\mathcal{H}}_0 \, .
\end{equation}

\begin{definition} \label{d:E-convergence-ArCaLo} We say that a family
$\{u_\eps\}_{0<\eps\le \eps_0}$, with $u_\eps\in {\mathcal{H}}_\eps$,
$E$-converges to $u\in {\mathcal{H}}_0$ if $\|u_\eps-E_\eps u\|_{{\mathcal{H}}_\eps}\to 0$
as $\eps\to 0$. We write this as $u_\eps \EC u$.
\end{definition}

\begin{definition} \label{d:EE-convergence-ArCaLo}
Let $\{B_\eps\in\mathcal L({\mathcal{H}}_\eps):\eps \in (0,\eps_0]\}$ be a
family of linear and continuous operators. We say that
$\{B_\eps\}_{0<\eps\le \eps_0}$ converges to $B_0\in \mathcal L({\mathcal{H}}_0)$ as $\eps\to 0$
if $B_\eps u_\eps \EC B_0 u$ whenever $u_\eps \EC u$. We write
this as $B_\eps \EEC B_0$.
\end{definition}

\begin{definition} \label{d:precompact-ArCaLo}
Let $\{u_\eps\}_{0<\eps\le \eps_0}$ be a family such that
$u_\eps\in {\mathcal{H}}_\eps$. We say that $\{u_\eps\}_{0<\eps\le
\eps_0}$ is precompact if for any sequence $\eps_n\to 0$ there
exist a subsequence $\{\eps_{n_k}\}$ and $u\in {\mathcal{H}}_0$
such that $u_{\eps_{n_k}} \EC u$.
\end{definition}

\begin{definition} \label{d:CC-convergence-ArCaLo}
We say that $\{B_\eps\}_{0<\eps\le \eps_0}$ with $B_\eps\in
\mathcal L({\mathcal{H}}_\eps)$ and $B_\eps$ compact, converges
compactly to a compact operator $B_0\in \mathcal
L({\mathcal{H}}_0)$ if $B_\eps \EEC B_0$ and for any family
$\{u_\eps\}_{0<\eps\le \eps_0}$ such that $u_\eps\in
{\mathcal{H}}_\eps$, $\|u_\eps\|_{\mathcal H_\eps}=1$, we have
that $\{B_\eps u_\eps\}_{0<\eps\le \eps_0}$ is precompact in the
sense of Definition \ref{d:precompact-ArCaLo}. We write this as
$B_\eps \CC B_0$.
\end{definition}

Recall that given a non-negative compact self-adjoint operator $B$ in a infinite dimensional Hilbert space ${\mathcal{H}}$, the spectrum is a finite or countably infinite  set and all non-zero elements of the spectrum are positive eigenvalues of finite multiplicity. In case the spectrum is a countably infinite set, the eigenvalues can be represented by a non-increasing  sequence $\mu_n$, $n\in \N$, such that $\lim_{n\to \infty }\mu_n=0$. As customary, we agree to repeat each eigenvalue $\mu_n$ in the sequence as many times as its multiplicity. Given a finite set of $m$ eigenvalues $\mu_n,  \dots , \mu_{n+m-1}$ with
$\mu_n\ne \mu_{n-1}$ and $\mu_{n+m-1}\ne \mu_{n+m}$, we call generalized eigenfunction (associated with $\mu_n,  \dots , \mu_{n+m-1}$) any linear combination of eigenfunctions associated with the eigenvalues $\mu_n,  \dots , \mu_{n+m-1}$.
Then we have the following theorem  which is a simplified rephrased version of \cite[Theorem~6.3]{Vainikko}, see also  \cite[Theorem~4.10]{ArCaLo}, \cite[Theorem~5.1]{ArrZua}, \cite[Theorem~3.3]{CarPisk}.

 \begin{theorem}
 \label{vaithm}
 Let  $\{B_\eps\}_{0<\eps\le \eps_0}$ and $B_0$ be non-negative compact self-adjoint operators in the Hilbert
 spaces ${\mathcal{H}}_{\eps }$, ${\mathcal{H}}_0$ respectively. Assume that their  eigenvalues are given
 by the sequences $\mu_{n}(\eps )$ and $\mu_{n}(0)$, $n\in \N$, respectively.  If  $B_\eps \CC B_0$ then we have
 spectral convergence of $B_{\eps}$ to $B_0$ as $\eps \to 0$ in the sense that the following statements hold:
 \begin{itemize}
 \item[(i)] For every $n\in \N$ we have $\mu_n(\eps )\to \mu_n(0)$ as $\eps \to 0$.
 \item[(ii)] If $u_n(\eps)$, $n\in \N$, is an orthonormal sequence of eigenfunctions associated with the eigenvalues  $\mu_n(\eps )$ then
 there exists a sequence $\eps_k$, $k\in \N$,  converging to zero and orthonormal sequence of eigenfunctions  $u_n(0)$, $n\in \N$ associated with  $\mu_n(0 )$, $n\in \N$ such that $u_n(\eps_k)\EC u_n(0)$.
 \item[(iii)]  Given  $m$ eigenvalues $\mu_n(0),  \dots , \mu_{n+m-1}(0)$ with
$\mu_n(0)\ne \mu_{n-1}(0)$ and $\mu_{n+m-1}(0)\ne \mu_{n+m}(0)$
 and corresponding orthonormal eigenfunctions $u_n(0),  \dots , u_{n+m-1}(0)$,
 there exist $m$ orthonormal   generalized eigenfunctions   $v_n(\eps ),  \dots , v_{n+m-1}(\eps )$  associated with  $\mu_n(\eps ),  \dots ,
  \mu_{n+m-1}(\eps )$   such that $v_{n+i}(\eps )\EC u_{n+i}(0)$  for all $i=0, 1,\dots , m-1$.
 \end{itemize}
\end{theorem}

\begin{remark} We mention that an alternative approach to the spectral analysis of operators defined on variable Hilbert spaces has been recently proposed in the book by O. Post \cite{Post}. In particular, that method allows to obtain stability estimates which can be expressed in terms of suitable estimates for the variation of the resolvent operators. It would be interesting to implement that approach in our setting in particular in the cases where we do not have degeneration phenomena but this would require heavier computations going beyond the scopes of the present paper.
\end{remark}

\subsection{Classes of domains} \label{opensets}

In order study the  spectral convergence for
 Steklov eigenvalue problems on varying domains, we shall consider uniform families of domains
for which a number of parameters are prescribed.  This is done by using the notion of atlas
from
\cite{BuLa}, see
also \cite[Section 5]{ArLa2}. According to \cite{ArLa2,BuLa},
given
a set $V\subset \R^N$ and a number $\delta>0$ we write
\begin{equation} \label{eq:def-Vdelta-B}
V_\delta:=\{x\in V:d(x,\partial V)>\delta\} \, .
\end{equation}

\begin{definition} \label{d:atlas-B}
\ \cite[Definition 2.4]{BuLa} Let $\rho>0$, $s,s'\in \N$ with
$s'<s$. Let $\{V_j\}_{j=1}^s$ be a family of open cuboids (i.e.
rotations of rectangle parallelepipeds in $\R^N$) and
$\{r_j\}_{j=1}^s$ be a family of rotations in $\R^N$. We say that
$\mathcal A=(\rho,s,s',\{V_j\}_{j=1}^s,\{r_j\}_{j=1}^s)$ is an
atlas in $\R^N$ with parameters
$\rho,s,s',\{V_j\}_{j=1}^s,\{r_j\}_{j=1}^s$ briefly an atlas in
$\R^N$. Moreover, we say that a bounded domain $\Omega$ in $\R^N$ belongs to the class
$C^{k,\gamma}(\mathcal A)$ with $k\in \N$ and $\gamma\in [0,1]$ if the following conditions are satisfied:
\begin{itemize}
\item[(i)] $\Omega\subset \cup_{j=1}^s (V_j)_\rho$ \ and \
$(V_j)_\rho \cap \Omega\neq \emptyset$ where $(V_j)_\rho$ is meant
in the sense given in \eqref{eq:def-Vdelta-B} ;

\item[(ii)] $V_j\cap \partial\Omega \neq \emptyset$ \ for \
$j=1,\dots,s'$ \ and \ $V_j\cap \partial\Omega=\emptyset$ \ for \
$s'+1\le j\le s$;

\item[(iii)] for $j=1,\dots,s$ \ we have
\begin{align*}
& r_j(V_j)=\{x\in \R^N:a_{ij}<x_i<b_{ij}\, , i=1,\dots,N\}, \quad
j=1,\dots,s ; \\
& r_j(V_j\cap \Omega)=\{x=(x',x_N)\in\R^N:x'\in W_j,
a_{Nj}<x_N<g_j(x')\}, \quad j=1,\dots, s'
\end{align*}
where $x'=(x_1,\dots,x_{N-1})$, $W_j=\{x'\in
\R^{N-1}:a_{ij}<x_i<b_{ij}, i=1,\dots,N-1\}$ and the functions
$g_j\in C^{k,\gamma}(\overline{W_j})$ for any $j\in 1,\dots,s'$ with $k\in
\N\cup \{0\}$ and $0\le \gamma\le 1$. Moreover, for $j=1,\dots,s'$
we assume that $a_{Nj}+\rho\le g_j(x')\le b_{Nj}-\rho$, for all
$x'\in \overline{W_j}$.
\end{itemize}
\noindent Finally we say that $\Omega$ if of class $C^{k,\gamma}$
if it is of class $C^{k,\gamma}(\mathcal A)$ for some atlas
$\mathcal A$. In the sequel $C^{k,0}$ will be simply denoted by $C^k$.
\end{definition}

\subsection{Function spaces}

Let $\Omega$ be a domain  in ${\mathbb{R}}^N$.   In this paper,
we shall often use the standard Sobolev spaces $H^k(\Omega)$ and
$H^k_0(\Omega)$ for $k=1,2$. Recall that $H^k(\Omega)$ is the
space of those functions in $L^2(\Omega)$ with weak derivatives up
to order $k$ in $L^2(\Omega)$, and that $H^k_0(\Omega)$ is the
closure in $H^k(\Omega)$ of $C^{\infty }_c(\Omega)$.  Note that in
this paper we assume that $L^2(\Omega)$ is made of real-valued
functions.

Besides $H^2(\Omega)$ we shall also consider the  larger space $H(\Delta, \Omega )$ defined by
\begin{equation}
 H(\Delta, \Omega)=\{u\in L^2(\Omega):\ \Delta u\in L^2(\Omega )
\},
\end{equation}
where $\Delta u $ denotes  the distributional Laplacian of $u$.
Thus, by definition  $\int_{\Omega }u\Delta \varphi
dx=\int_{\Omega }\Delta u \varphi dx$ for all $\varphi \in
C^{\infty }_c(\Omega)$.  The space $ H(\Delta ,\Omega )$ is a
natural space associated with the Dirichlet Laplacian which we
temporarily denote  here by $\Delta_{D}$. Recall that the operator
$\Delta_{D}$ is the  non-negative self-adjoint operator in
$L^2(\Omega)$ canonically associated with the quadratic form
$\int_{\Omega}\nabla u\nabla \varphi dx$ defined for all
$u,\varphi \in H^1_0(\Omega)$. This means that a function $u\in
L^2(\Omega)$ belongs to the domain ${\rm Dom }(\Delta_{D})$ of
$\Delta_{D}$ if and only if $u\in H^1_0(\Omega)$ and there exists
$f\in L^2(\Omega)$ such that $\int_{\Omega}\nabla u\nabla \varphi
dx=\int_{\Omega }f\varphi dx$ for all $\varphi \in H^1_0(\Omega)
$, in which case $-\Delta_{D}u=f$. It follows by these definitions
that ${\rm Dom }(\Delta_{D})= H(\Delta, \Omega )\cap H^1_0(\Omega
). $

We have the following lemma where by Poincar\'{e} inequality we
mean the classical inequality $\| u\|_{L^2(\Omega )}\le c \|\nabla
u\|_{L^2(\Omega )}$ for all $u\in H^1_0(\Omega )$ and by
Poincar\'{e} constant we mean the best constant $c$ in that
inequality. Recall  that the Poincar\'{e} inequality holds under
general assumptions on $\Omega$, for example if $\Omega$  has
finite measure it holds with $c= c_N |\Omega |^{\frac{1}{N}}$,  or
if $\Omega $ is bounded in one direction it holds with $c= c_N d$
where $d$ is the diameter in that direction (here the symbol $c_N$
is used to denote constants depending only on $N$). Moreover, the definition
of uniform exterior ball
condition  is standard, see e.g., \cite{Adolf}.

\begin{lemma}
\label{complete}Let $\Omega $ be a domain  in  ${\mathbb{R}}^N$
such that the Poincar\'{e} inequality holds. Then the following
statements hold:
\begin{itemize}
\item[(i)]  For all $u\in  H(\Delta, \Omega )\cap H^1_0(\Omega )$
we have
\begin{equation}
\label{poin} \|  u\|_{L^2(\Omega )}\le  c_{{\mathcal{P}}}^2
\|\Delta u\|_{L^2(\Omega)}, \ \ {\rm and}\ \ \| \nabla
u\|_{L^2(\Omega )}\le   c_{{\mathcal{P}}}  \|\Delta
u\|_{L^2(\Omega)} ,
\end{equation}
where $ c_{{\mathcal{P}}}$ denotes the Poincar\'{e} constant.
\item[(ii)] The space $ H(\Delta, \Omega )\cap H^1_0(\Omega ) $
endowed with the scalar product defined by $\int_{\Omega }\Delta u
\Delta vdx$ for all $u, v\in H(\Delta, \Omega )\cap H^1_0(\Omega )
$  is a Hilbert space. \item[(iii)] If $\Omega $ is a bounded open
set of class $C^{0,1}$ which satisfies the uniform exterior ball
condition then $ H(\Delta, \Omega )\cap H^1_0(\Omega ) =
H^2(\Omega)\cap H^1_0(\Omega ) $.
\end{itemize}
\end{lemma}

\begin{proof} The second inequality in  (i) easily follows by the Poincar\'{e} inequality because
 if $u\in H(\Delta, \Omega )\cap H^1_0(\Omega )$ then
\[
\int_{\Omega }|\nabla u|^2dx=-\int_{\Omega }u\Delta u dx\le
\|u\|_{L^2(\Omega)} \|\Delta u\|_{L^2(\Omega)}\le
c_{{\mathcal{P}}} \|\nabla u \|_{L^2(\Omega )} \|\Delta
u\|_{L^2(\Omega)}.
\]
The first inequality in (i) follows by the second inequality
combined again with the Poincar\'{e} inequality. We now prove
statement (ii). It clearly suffices to prove the completeness of
the space. Let $u_n$, $n\in \N$ be a sequence in  $
H(\Delta , \Omega)\cap H^1_0(\Omega ) $ such that $f_n:=\Delta u_n  $ is
a Cauchy sequence in $L^2(\Omega )$. By $(i)$, also $u_n$, $n\in
\N$ is a Cauchy sequence in $H^1_0(\Omega)$. Thus, there exists
$f\in L^2(\Omega )$ and $u\in H^1_0(\Omega )$ such that $\Delta
u_n\to f $ in $L^2(\Omega )$ and $u_n\to u$ in $H^1_0(\Omega )$ as
$n\to \infty$. Passing to the limit in the equality
$\int_{\Omega}\nabla u_n\nabla \varphi dx=-\int_{\Omega
}f_n\varphi dx$ as $n\to \infty $, we get $\int_{\Omega}\nabla
u\nabla \varphi dx=-\int_{\Omega }f\varphi dx$ for all $\varphi\in
H^1_0(\Omega )$ which means that $-\Delta u =f$ hence $u\in
H(\Delta, \Omega)$. This completes the proof of (ii). Statement (iii) is
a consequence of regularity results for the solutions of the
Dirichlet problem, see \cite{Adolf}.
\end{proof}

\begin{remark}\label{remcom} It follows by Lemma~\ref{complete} that if $\Omega$ is a bounded domain of class $C^{1,1}$ then
$H(\Delta, \Omega )\cap H^1_0(\Omega )=H^2(\Omega)\cap H^1_0(\Omega)$ and the two norms are equivalent. Indeed, it is well known that $C^{1,1}$ regularity
implies the validity of the uniform outer ball condition.
\end{remark}

\section{Strong spectral convergence for Steklov and Steklov-type problems}
\label{strongsec}

\subsection{The functional setting} \label{s:functional-setting-B}

Assume that  $\Omega\subset \R^N$ ($N\ge 2$) is a bounded domain of class $C^{0,1}$ satisfying the uniform exterior ball condition. To shorten our notation we set
$$
V(\Omega):=H^2(\Omega)\cap H^1_0(\Omega).
$$

Proceeding as in \cite{BuFeGa}, we relate the spectrum of the
Steklov problem to the spectrum of a suitable self-adjoint compact
operator acting on the space $V(\Omega)$. The space $V(\Omega)$
inherits the scalar product of $H^2(\Omega)$ and being closed in
$H^2(\Omega)$ it becomes a Hilbert space. The space $V(\Omega)$
may be endowed with the two alternative scalar products
 \begin{align} \label{eq:s-p}
& (u,v)\mapsto \int_\Omega \Delta u \Delta v \, dx \qquad
\text{for any } u,v \in V(\Omega) \, , \\[5pt]
\label{eq:s-pbis}
 & (u,v)\mapsto \int_\Omega D^2 u: D^2 v \, dx \qquad
\text{for any } u,v \in V(\Omega) \, .
 \end{align}
Recall  that by Lemma~\ref{complete} (iii) the two scalar products defined in \eqref{eq:s-p}  and \eqref{eq:s-pbis} are equivalent to the one induced by the $H^2(\Omega)$-norm.
 In the sequel we denote by $V'(\Omega)$ the dual space of $V(\Omega)$.

 Consider now the eigenvalue problem \eqref{eq:Steklov} and let us
 introduce a suitable characterization of its spectrum.

We introduce the operator $T_\Delta:V(\Omega)\mapsto V'(\Omega)$
defined by
 $$
_{V'(\Omega)} \langle T_\Delta u,v\rangle_{V(\Omega)}:=\int_\Omega \Delta u \Delta v \, dx \qquad \text{for any } u,v\in V(\Omega) \, .
 $$
The operator $T_\Delta$ is clearly well-defined and continuous.
Moreover by Riesz Theorem we also deduce that $T_\Delta$
is invertible and $T_\Delta^{-1}$ is continuous. Then we introduce
the operator $J:V(\Omega)\mapsto V'(\Omega)$ defined by
\begin{equation*}
_{V'(\Omega)} \langle Ju,v \rangle_{V(\Omega)}:= \int_{\partial\Omega} u_\nu v_\nu \, dS \qquad \text{for any } u,v\in V(\Omega) \, .
\end{equation*}
Since the linear map
\begin{align} \label{eq:trace}
  V(\Omega) & \mapsto (L^2(\partial\Omega))^N \\
 \notag  u & \mapsto \nabla u_{|\partial\Omega}
\end{align}
is well-defined and compact being $\partial\Omega$ Lipschitzian  (see \cite[Theorem 6.2, Chap. 2]{Necas} for more details), then the operator $J$ is also
well-defined and compact.

Next we introduce the operator $S_\Delta:V(\Omega)\mapsto
V(\Omega)$ defined by $S_\Delta:=T^{-1}_\Delta\circ J$. Clearly
$S_\Delta$ is a linear compact operator and moreover it is easy to
see that it is also self-adjoint. Moreover one can show that
$\mu\neq 0$ is an eigenvalue of $S_\Delta$ if and only if
$d:=\mu^{-1}$ is an eigenvalue of \eqref{eq:Steklov}. See the
proof of Lemma 2 in \cite{BuFeGa} for more details.

With a completely similar procedure we can treat problem
\eqref{eq:Steklov-modificato} by introducing the operator
$T_{D^2}:V(\Omega)\to V'(\Omega)$ defined as
\begin{equation*}
 _{V'(\Omega)}\langle T_{D^2}\, u,v\rangle_{V(\Omega)}:=\int_\Omega D^2 u:D^2 v \, dx \qquad \text{for any } u,v\in
 V(\Omega)\, .
\end{equation*}
Applying Riesz Theorem, we see that also the operator
$T_{D^2}$ is invertible and its inverse function is continuous.
Then one can define $S_{D^2}:V(\Omega)\to V(\Omega)$ as
$S_{D^2}:=T_{D^2}^{-1}\circ J$ and prove that $S_{D^2}$ is a
compact self-adjoint operator. Moreover any $\mu\neq 0$ is an
eigenvalue of $S_{D^2}$ if and only if $\delta=\mu^{-1}$ is an
eigenvalue of \eqref{eq:Steklov-modificato}. This also proves the
validity of \eqref{eq:sequence-modificato} and show that the first
eigenvalue $\delta_1$ admits the following characterization
\begin{equation} \label{eq:char-delta-1}
\delta_1=\inf_{v\in V(\Omega)\setminus H^2_0(\Omega )} \frac{\int_\Omega
|D^2 v|^2 \, dx}{\int_{\partial\Omega} v_\nu^2 \, dS} \, .
\end{equation}

\subsection{Domain perturbations and definition of the operators $E_{\eps}$}
\label{subsecoperatorsE}

Let $\mathcal A$ be an atlas. We shall consider a family
 $\{\Omega_\eps\}_{0<\eps\le \eps_0}$ of domains of
class $C^{1,1}(\mathcal A)$ converging in a suitable sense to a fixed domain $\Omega$ of
class $C^{1,1}(\mathcal A)$.
For any $0<\eps\le \eps_0$ we introduce the Steklov operators
\begin{equation*}
S_{\Delta,\eps}:V(\Omega_\eps) \to V(\Omega_\eps) \quad \text{and}
\quad S_{D^2,\eps}:V(\Omega_\eps) \to V(\Omega_\eps)
\end{equation*}
according to  the definition given in Section
\ref{s:functional-setting-B},
and we  study the convergence  of $S_{\Delta,\eps}$ to $S_{\Delta}$ and of  $S_{D^2,\eps}$ to  $S_{D^2}$.
To do so, we shall  apply the general method described in
Subsection \ref{ss:convergenceint} where the Hilbert  spaces
${\mathcal{H}}_{\eps}$ and ${\mathcal{H}}$ under consideration
here are exactly $V(\Omega_{\eps})$ and $V(\Omega)$ endowed either
with the scalar product \eqref{eq:s-p} or with the scalar product
\eqref{eq:s-pbis}.

In particular, we need to define a family of operators $E_{\eps}$
as required in Subsection~ \ref{ss:convergenceint}, which satisfy
condition \eqref{eq:norm-convergence}. In the specific case under
consideration, this means that we have to introduce
 linear operators
$E_\eps:V(\Omega)\to V(\Omega_\eps)$ such that
\begin{equation} \label{eq:conv-norm-Delta}
\int_{\Omega_\eps} |\Delta(E_\eps u)|^2 dx\to \int_\Omega |\Delta
u|^2 dx \qquad \text{as } \eps\to 0\, , \quad \text{for any } u\in
V(\Omega)
\end{equation}
and
\begin{equation}  \label{eq:conv-norm-D^2}
\int_{\Omega_\eps} |D^2(E_\eps u)|^2 dx\to \int_\Omega |D^2 u|^2
dx \qquad \text{as } \eps\to 0\, , \quad \text{for any } u\in
V(\Omega) \, .
\end{equation}

To do so, we proceed as in \cite{ArLa2}. Suppose that the following assumptions hold true: for any $\eps\in (0,\eps_0]$ there exists $\kappa_\eps>0$ such that for any $j=1,\dots ,s'$
\begin{align} \label{eq:assumptions}
\|g_{\eps,j}-
g_j\|_{L^\infty(W_j)}<\kappa_\eps\, , \quad \lim_{\eps \to 0} \kappa_\eps=0\, ,
\qquad \lim_{\eps\to 0} \frac{\|D^\beta(g_{\eps,j}-g_j)\|_{L^\infty(W_j)}}{\kappa_\eps^{3/2-|\beta|}}=0 \quad \forall \beta\in \N^N \ \text{with}
\ |\beta|\le 2 \, .
\end{align}
Here we denoted by $g_j, g_{\eps,j} \in C^{1,1}(\overline{W_j})$ the functions
corresponding respectively to $\Omega$ and $\Omega_\eps$ according
to  Definition \ref{d:atlas-B}.

\begin{remark}\label{remweakconv} One can prove that if
\begin{equation}\label{remweakconv0}
D^{\alpha }g_{\eps,j}(x')\to D^{\alpha }g_{j}(x'),\ \ {\rm uniformly\ on}\ W_j
\end{equation}
 for all $|\alpha |\le 1$ and $j=1,\dots , s'$ and if there exists $M>0$ such that
\begin{equation}\label{remweakconv1}
| D^{\alpha }g_{\eps,j}(x')|\le M,\ \ {\rm a.e.\ on }\ W_j
\end{equation}
 for all $|\alpha |=2$ and $j=1,\dots , s'$, then condition \eqref{eq:assumptions} is satisfied, see \cite[Prop.~6.17]{ArLa2}. We note that condition
  \eqref{eq:assumptions} is much less restrictive than conditions \eqref{remweakconv0}, \eqref{remweakconv1}, since it does not require the $L^{\infty}$-norms
  of the second order weak derivatives of functions $g_{\eps,j}$ to be uniformly bounded. For example,  in Corollary~\ref{t:main-2ante} if  $3/2 < \alpha <2$ then condition  \eqref{eq:assumptions} is satisfied but the $L^{\infty}$-norms of the second order weak derivatives of functions $g_{\eps}$ are not uniformly bounded.
\end{remark}

We are going to define the operator $E_\eps$ by using a partition of unity and pasting together suitable pull-back operators associated with local diffeomorphisms defined on each cuboid of the atlas $\mathcal A$. Note that in the simplified setting of one single cuboid, partition of unity would not be required and the operator $E_\eps$ would be simply defined as in Remark \ref{r:cuboid} below.

We denote by $\hat k$ a fixed constant satisfying $\hat k>6$ whose meaning will be clear below. We also put
$k_\eps:= \hat k \kappa_\eps$, $\tilde g_{\eps,j}:=g_{\eps,j}-k_\eps$ and $K_{\eps,j}:=\{(x',x_N)\in W_j\times (a_{Nj},b_{Nj}):a_{Nj}<x_N<\tilde g_{\eps,j}(x')\}$
for any $j=1,\dots,s'$.

For any $1\le j\le s'$ we define the map $h_{\eps,j}:r_j(\overline{\Omega_\eps\cap V_j}) \to
\R$
\begin{equation} \label{eq:def-h}
h_{\eps,j}(x',x_N):=
\begin{cases} 0, & \qquad \text{if } a_{jN} \le x_N\le \tilde g_{\eps,j}(x'), \\
\left(g_{\eps,j}(x')-g_j(x')\right)\left(\frac{x_N-\tilde g_{\eps,j}(x')}{g_{\eps,j}(x')-\tilde g_{\eps,j}(x')}\right)^3,
& \qquad \text{if } \tilde g_{\eps,j}(x')< x_N\le g_{\eps,j}(x') \, .
\end{cases}
\end{equation}
We observe that $h_{\eps,j}\in C^{1,1}(r_j(\overline{\Omega_\eps
\cap V_j}))$ and that the map
$\Phi_{\eps,j}:r_j(\overline{\Omega_\eps \cap V_j}) \to
r_j(\overline{\Omega \cap V_j})$ defined by
$\Phi_{\eps,j}(x',x_N):=(x',x_N-h_{\eps,j}(x',x_N))$ is a
diffeomorphism of class $C^{1,1}$. When $s'+1\le j\le s$ we define
$\Phi_{\eps,j}:r_j(\overline{V_j})\to r_j(\overline{V_j})$ as the
identity map.

Exploiting the definition of $\Phi_{\eps,j}$ we may define the
above mentioned linear map $E_\eps$. Consider a partition of unity
$\{\psi_j\}_{1\le j\le s}$ subordinate to the open cover
$\{V_j\}_{1\le j\le s}$ of the compact set $\overline{\Omega\cup
\bigcup_{\eps\in (0,\eps_0]} \Omega_\eps}$, i.e.,
a set of functions  $\{\psi_j\}_{1\le j \le s}$ such that
$\text{supp}(\psi_j)\subset V_j$ for any $j\in \{1,\dots,s\}$ and
$\sum_{j=1}^s \psi_j\equiv 1$ in $\overline{\Omega\cup
\bigcup_{\eps\in (0,\eps_0]}\Omega_\eps}$. Indeed one can apply
for example \cite[Lemma IX.3]{Brezis} with
$\Gamma=\overline{\Omega\cup \bigcup_{\eps\in (0,\eps_0]}
\Omega_\eps}$. Note that this partition of unity is independent of $\eps$. We also define the deformation
$\Psi_{\eps,j}:\overline{\Omega_\eps \cap V_j}\to \overline{\Omega
\cap V_j}$ by $\Psi_{\eps,j}:=r_j^{-1}\circ \Phi_{\eps,j}\circ
r_j$. In this way $\Psi_{\eps,j}$ becomes a $C^{1,1}$
diffeomorphism from $\overline{\Omega_\eps \cap V_j}$ onto
$\overline{\Omega \cap V_j}$ for any $j\in \{1,\dots,s\}$.

From the definition of $h_{\eps,j}$ and the restriction $\hat k>6$, we deduce that
\begin{equation} \label{eq:stima-Jacob}
\frac 12 \le {\rm det}(D\Psi_{\eps,j}(x))\le \frac 32 \qquad \text{for any } x\in \Omega_\eps\cap V_j \, .
\end{equation}
In order to show this, we observe that ${\rm det}(D\Psi_{\eps,j})={\rm det}(D\Phi_{\eps,j})=1-\frac{\partial h_{\eps,j}}{\partial x_N}(x',x_N)$.

Given $u\in V(\Omega)$ we put $u_j=\psi_j u$ for any $j\in \{1,\dots,s\}$ in such a way that $u=\sum_{j=1}^s u_j$. Then we define
\begin{equation} \label{eq:E-eps}
E_\eps u:=\sum_{j=1}^{s'} \tilde u_{\eps,j}+\sum_{j=s'+1}^s u_j \in V(\Omega_\eps) \,
\end{equation}
where
\begin{equation} \label{eq:uj-tilde}
\tilde u_{\eps,j}(x)=
\begin{cases} u_j(\Psi_{\eps,j}(x)) & \qquad \text{if } x\in \Omega_\eps\cap V_j \, , \\[8pt]
 0 & \qquad \text{if } x\in \Omega_\eps\setminus V_j \, .
\end{cases}
\end{equation}
for any $j\in \{1,\dots,s'\}$.

\begin{remark} \label{r:cuboid} If $\Omega$ and $\Omega_\eps$ are in the form
\begin{align*}
& \Omega=\{(x',x_N)\in \R^N:x'\in W \ \text{and} \ a_N<x_N<g(x')\} \, , \\
& \Omega_\eps=\{(x',x_N)\in \R^N:x'\in W \ \text{and} \ a_N<x_N<g_\eps(x')\} \, ,
\end{align*}
where $W$ is a cuboid or a bounded domain in $\R^{N-1}$ of class $C^{1,1}$ then the operator $E_\eps$ can be defined in the following simple way
\begin{equation*}
E_\eps u(x)=u(\Phi_\eps(x)) \qquad \text{for any } x\in \Omega_\eps
\end{equation*}
where $\Phi_\eps(x',x_N)=(x',x_N-h_\eps(x',x_N))$ and $h_\eps$ is defined by \eqref{eq:def-h} with $g_\eps$ and $g$ in place of $g_{\eps,j}$ and $g_j$ respectively.
\end{remark}

\begin{lemma} \label{l:1} Let $\mathcal A$ be an atlas. Let
$\{\Omega_\eps\}_{0<\eps\le\eps_0}$ be a family of domains of
class $C^{1,1}(\mathcal A)$, $\Omega$ a domain of class
$C^{1,1}(\mathcal A)$ and for any $\eps\in (0,\eps_0]$ let $E_\eps$ be the map defined in \eqref{eq:E-eps}. Assume  the validity of  condition  \eqref{eq:assumptions}.

Then the following assertions hold true:
\begin{itemize}
\item [(i)] the map $E_\eps:V(\Omega)\to V(\Omega_\eps)$ is continuous for any $\eps\in (0,\eps_0]$;

\item [(ii)] the family of operators $\{E_\eps\}_{0<\eps\le
\eps_0}$ satisfies \eqref{eq:conv-norm-Delta} and
\eqref{eq:conv-norm-D^2}.
\end{itemize}
\end{lemma}

\begin{proof} Since the proof of (i) easily follows from the definition of $E_\eps$, it is left to the reader.

We now prove \eqref{eq:conv-norm-Delta}.
Let $u\in V(\Omega)$ and according to \eqref{eq:uj-tilde}
introduce the functions $\tilde u_{\eps,j}$. Throughout the proof
of this lemma we denote by $C$ a positive constant independent of
$\eps$ which may vary from line to line.
We observe that in order to prove  \eqref{eq:conv-norm-Delta}, it is sufficient to
show that
\begin{align} \label{eq:pas-lim}
\lim_{\eps\to 0} \int_{\Omega_\eps} \frac{\partial^2\tilde u_{\eps,i}}{\partial x_k^2}
\frac{\partial^2\tilde u_{\eps,j}}{\partial x_l^2}\, dx=\int_\Omega \frac{\partial^2 u_i}{\partial x_k^2}
\frac{\partial^2 u_j}{\partial x_l^2} \, dx
\end{align}
for any $i,j\in \{1,\dots,s'\}$ and $k,l\in \{1,\dots,N\}$.
We first prove that \eqref{eq:pas-lim} holds when $i=j$ and $k=l$. One can proceed similarly to the proof of \cite[Lemma 6.2]{ArLa2} and
hence we only give here a sketch of it. By \eqref{eq:assumptions} and (6.7) in \cite{ArLa2} we have that
\begin{equation} \label{eq:estimate}
\|D^\alpha h_{\eps,j}\|_{L^\infty(W_j)}\le C \sum_{0\le \gamma \le \alpha} \frac{\|D^{\gamma}(g_{\eps,j}-g_j)\|_{L^\infty(W_j)}}{\kappa_\eps^{|\alpha|-|\gamma|}}
\end{equation}
for any $\alpha\in \N^N$ with $|\alpha|\le 2$. By $\gamma\le
\alpha$ we mean that the inequality holds component by component.
For any $x\in \Omega_\eps\cap V_i$ we have
\begin{equation} \label{eq:compute}
\frac{\partial^2 \tilde u_{\eps,i}}{\partial x_k^2}(x)=\sum_{j=1}^N \left\{ \left[
\sum_{l=1}^N \frac{\partial^2 u_i}{\partial x_l \partial x_j}(\Psi_{\eps,i}(x)) \frac{\partial[(\Psi_{\eps,i}(x))_l]}{\partial x_k} \right]
 \frac{\partial[(\Psi_{\eps,i}(x))_j]}{\partial x_k}+\frac{\partial u_i}{\partial x_j}(\Psi_{\eps,i}(x)) \frac{\partial^2[(\Psi_{\eps,i}(x))_j]}
 {\partial x_k^2} \right\} \, .
\end{equation}
By \eqref{eq:assumptions}, \eqref{eq:estimate} and the definitions of $\Phi_{\eps,i},\Psi_{\eps,i}$, we deduce that there exists $C>0$ independent of $\eps$
such that
\begin{equation} \label{eq:estimate-2}
\left\|\frac{\partial[(\Psi_{\eps,i}(x))_l]}{\partial x_k}  \frac{\partial[(\Psi_{\eps,i}(x))_j]}{\partial x_k}\right\|_{L^\infty(V_i\cap \Omega_\eps)}\le C \quad
\text{and} \quad \left\|\frac{\partial^2[(\Psi_{\eps,i}(x))_j]}{\partial x_k^2}\right\|_{L^\infty(V_i\cap \Omega_\eps)}=o(\kappa_\eps^{-1/2}) \ \ \text{as } \eps\to 0 \, .
\end{equation}
Exploiting \eqref{eq:compute}-\eqref{eq:estimate-2} and proceeding as in the proof of \cite[Lemma 6.2]{ArLa2} we deduce that
\begin{equation} \label{eq:pezzo-1}
\lim_{\eps\to 0} \int_{(\Omega_\eps \cap V_i)\setminus r_i^{-1}(K_{\eps,i})} \left( \frac{\partial^2\tilde u_{\eps,i}}{\partial x_k^2}\right)^2 dx=0 \, .
\end{equation}
We briefly give a proof of \eqref{eq:pezzo-1}. First of all by \eqref{eq:stima-Jacob} we have
\begin{align} \label{eq:o(1)-1}
\int_{(\Omega_\eps \cap V_i)\setminus r_i^{-1}(K_{\eps,i})}
\left(\frac{\partial^2 u_i}{\partial x_l \partial x_j}(\Psi_{\eps,i}(x))\right)^2 dx \le
C \int_{(\Omega\cap V_i)\setminus r_i^{-1}(K_{\eps,i})}
\left(\frac{\partial^2 u_i}{\partial x_l \partial x_j} \right)^2 dx\to 0
\end{align}
as $\eps\to 0$ since $|(\Omega\cap V_i)\setminus r_i^{-1}(K_{\eps,i})|\to 0$ as $\eps\to 0$.

On the other hand, putting $v_i(z)=u_i(r_i^{-1}(z))$ for any $z\in r_i(\Omega\cap V_i)$, by \eqref{eq:assumptions}, \eqref{eq:estimate-2} and changes of variables, we have
\begin{align} \label{eq:o(1)-2}
& \int_{(\Omega_\eps \cap V_i)\setminus r_i^{-1}(K_{\eps,i})} \left(\frac{\partial u_i}{\partial x_j}(\Psi_{\eps,i}(x)) \frac{\partial^2[(\Psi_{\eps,i}(x))_j]}
 {\partial x_k^2}\right)^2 dx\le o(\kappa_\eps^{-1}) \sum_{l=1}^N \int_{r_i(\Omega\cap V_i)\setminus K_{\eps,i}}
 \left(\frac{\partial v_i}{\partial z_l}\right)^2 dz \\
 \notag & =o(\kappa_\eps^{-1}) \sum_{l=1}^N \int_{W_i} \left( \int_{\tilde g_{\eps,i}(z')}^{g_i(z')} \left(\frac{\partial v_i}{\partial z_l}(z',z_N)\right)^2
 dz_N\right)dz' \\
 \notag & \le o(\kappa_\eps^{-1}) \sum_{l=1}^N \int_{W_i} |g_i(z')-\tilde g_{\eps,i}(z')| \left\|\frac{\partial v_i}{\partial z_l}(z',\cdot)\right\|
 _{L^\infty(a_{Ni},g_i(z'))}^2 dz' \\
 \notag & \le o(1) \sum_{l=1}^N \int_{W_i} \left\|\frac{\partial
v_i}{\partial z_l}(z',\cdot)\right\|
 _{H^1(a_{Ni},g_i(z'))}^2 dz'\\
 \notag &  \le o(1) \|v_i\|_{H^2(r_i(\Omega\cap V_i))}=o(1) \, .
\end{align}
Note that here we have used the classical one dimensional estimate
\begin{equation}\label{uniinf}
\|f\|_{L^{\infty }(a,b)}\le C \| f\|_{H^1(a,b)}
\end{equation} for all $f\in H^1(a,b) $ with a constant $C=C(d)$ uniformly bounded for $b-a>d$.
Combining \eqref{eq:o(1)-1}, \eqref{eq:o(1)-2} with \eqref{eq:compute} and \eqref{eq:estimate-2}, \eqref{eq:pezzo-1} follows.

On the other hand
\begin{equation} \label{eq:pezzo-2}
\lim_{\eps\to 0}\int_{r_i^{-1}(K_{\eps,i})}  \left( \frac{\partial^2\tilde u_{\eps,i}}{\partial x_k^2}\right)^2 dx
=\lim_{\eps\to 0} \int_{r_i^{-1}(K_{\eps,i})}  \left( \frac{\partial^2 u_{i}}{\partial x_k^2}\right)^2 dx
=\int_{\Omega\cap V_i}  \left( \frac{\partial^2 u_{i}}{\partial x_k^2}\right)^2 dx=\int_{\Omega}  \left( \frac{\partial^2 u_{i}}{\partial x_k^2}\right)^2 dx .
\end{equation}
Combining \eqref{eq:pezzo-1}-\eqref{eq:pezzo-2} we conclude the proof of \eqref{eq:pas-lim} in the case $i=j$ and $k=l$.

It remains to prove \eqref{eq:pas-lim} in the general case.

We observe that
\begin{align} \label{eq:pass-1}
 \int_{\Omega_\eps} & \frac{\partial^2\tilde u_{\eps,i}}{\partial x_k^2}
\frac{\partial^2\tilde u_{\eps,j}}{\partial x_l^2}\, dx
 =\int_{\Omega_\eps\cap V_i\cap V_j} \frac{\partial^2\tilde u_{\eps,i}}{\partial x_k^2}
\frac{\partial^2\tilde u_{\eps,j}}{\partial x_l^2}\, dx \\[8pt]
\notag & =\int_{r_i^{-1}(K_{\eps,i})\cap r_j^{-1}(K_{\eps,j})}
\frac{\partial^2\tilde u_{\eps,i}}{\partial x_k^2}
\frac{\partial^2\tilde u_{\eps,j}}{\partial x_l^2}\,
dx+\int_{(\Omega_\eps\cap V_i\cap V_j)\setminus
(r_i^{-1}(K_{\eps,i})\cap r_j^{-1}(K_{\eps,j}))}
\frac{\partial^2\tilde u_{\eps,i}}{\partial x_k^2}
\frac{\partial^2\tilde u_{\eps,j}}{\partial x_l^2}\, dx \, .
\end{align}
We have that $\Psi_{\eps,i}$ and $\Psi_{\eps,j}$ are the identity maps if restricted to $r_i^{-1}(K_{\eps,i})\cap r_j^{-1}(K_{\eps,j})$ so that by \eqref{eq:assumptions} we obtain
\begin{equation} \label{eq:pass-2}
\lim_{\eps\to 0} \int_{r_i^{-1}(K_{\eps,i})\cap r_j^{-1}(K_{\eps,j})} \frac{\partial^2\tilde u_{\eps,i}}{\partial x_k^2}
\frac{\partial^2\tilde u_{\eps,j}}{\partial x_l^2}\, dx=\lim_{\eps\to 0}
\int_{r_i^{-1}(K_{\eps,i})\cap r_j^{-1}(K_{\eps,j})} \frac{\partial^2 u_i}{\partial x_k^2}
\frac{\partial^2 u_j}{\partial x_l^2}\, dx=\int_\Omega \frac{\partial^2 u_i}{\partial x_k^2}
\frac{\partial^2 u_j}{\partial x_l^2}\, dx  \, .
\end{equation}
In order to estimate the other term in the right hand side of
\eqref{eq:pass-1}, we split the domain of integration in the
following way
\begin{equation} \label{eq:splitting}
(\Omega_\eps\cap V_i\cap V_j)\setminus
(r_i^{-1}(K_{\eps,i})\cap r_j^{-1}(K_{\eps,j}))=[(\Omega_\eps\cap V_i\cap V_j)\setminus r_i^{-1}(K_{\eps,i})]\cup [(r_i^{-1}(K_{\eps,i})\setminus r_j^{-1}(K_{\eps,j}))\cap V_j] \, .
\end{equation}
Then
\begin{align}  \label{eq:pass-3}
& \left|\int_{(\Omega_\eps\cap V_i\cap V_j)\setminus
r_i^{-1}(K_{\eps,i})} \frac{\partial^2\tilde u_{\eps,i}}{\partial x_k^2}
\frac{\partial^2\tilde u_{\eps,j}}{\partial x_l^2}\, dx\right| \\
\notag & \le \left(\int_{(\Omega_\eps\cap V_i)\setminus
r_i^{-1}(K_{\eps,i})} \left(\frac{\partial^2\tilde u_{\eps,i}}{\partial x_k^2}\right)^2 dx \right)^{1/2}
\left(\int_{\Omega_\eps\cap V_j} \left(\frac{\partial^2\tilde u_{\eps,j}}{\partial x_l^2}\right)^2 dx \right)^{1/2}\to 0 \qquad \text{as } \eps\to 0
\end{align}
in view of \eqref{eq:pezzo-1} and \eqref{eq:pas-lim} in the case
$i=j$ and $k=l$. On the other hand
\begin{align} \label{eq:pass-4}
 &  \left|\int_{(r_i^{-1}(K_{\eps,i})\setminus r_j^{-1}(K_{\eps,j}))\cap V_j} \frac{\partial^2\tilde u_{\eps,i}}{\partial x_k^2}
\frac{\partial^2\tilde u_{\eps,j}}{\partial x_l^2}\, dx\right| \\
\notag & \le
\left(\int_{r_i^{-1}(K_{\eps,i})} \left(\frac{\partial^2\tilde u_{\eps,i}}{\partial x_k^2}\right)^2 dx \right)^{1/2}
\left(\int_{(\Omega_\eps \cap V_j)\setminus r_j^{-1}(K_{\eps,j})} \left(\frac{\partial^2\tilde u_{\eps,j}}{\partial x_l^2}\right)^2 dx \right)^{1/2}\to 0 \qquad \text{as } \eps \to 0
\end{align}
in view of \eqref{eq:pezzo-1} and \eqref{eq:pezzo-2}.

Combining \eqref{eq:pass-2}-\eqref{eq:pass-4} with
\eqref{eq:pass-1}, the proof of \eqref{eq:pas-lim} follows also in
the general case.

We now prove  the validity of
\eqref{eq:conv-norm-D^2}. We observe that it is sufficient to show
that
\begin{align} \label{eq:pas-lim-bis}
\lim_{\eps\to 0} \int_{\Omega_\eps} \frac{\partial^2\tilde
u_{\eps,i}}{\partial x_k \partial x_l} \frac{\partial^2\tilde
u_{\eps,j}}{\partial x_k \partial x_l}\, dx=\int_\Omega
\frac{\partial^2 u_i}{\partial x_k \partial x_l} \frac{\partial^2
u_j}{\partial x_k \partial x_l} \, dx \, .
\end{align}
The procedure that one has to follow to prove
\eqref{eq:pas-lim-bis} is essentially the same adopted in Step 2:
first one has to prove \eqref{eq:pas-lim-bis} in the case $i=j$
with $k,l\in \{1,\dots,N\}$ not necessarily equal and then one has
to pass to the general case by arguing as above.
\end{proof}

\subsection{The  stability  result} \label{ss:s-r}

In this subsection we prove the following  result on spectral
convergence for problems \eqref{eq:Steklov} and
\eqref{eq:Steklov-modificato}. One of the main assumptions is
condition \eqref{eq:assumptions} which, as we can see from Lemma
\ref{l:1}, guarantees the validity of the properties
\eqref{eq:conv-norm-Delta} and \eqref{eq:conv-norm-D^2}.

\begin{theorem} \label{t:main-1}
Let $\mathcal A$ be an atlas. Let
$\{\Omega_\eps\}_{0<\eps\le \eps_0}$ be a family of domains of
class $C^{1,1}(\mathcal A)$ and let $\Omega$ be a domain of
class $C^{1,1}(\mathcal A)$. Assume the validity of condition
\eqref{eq:assumptions}. Then $S_{\Delta,\eps} \CC  S_{\Delta}$ and  $S_{D^2,\eps} \CC  S_{D^2}$ with respect to the operators $E_{\eps}$ defined in \eqref{eq:E-eps}. In particular, the spectra of \eqref{eq:Steklov} and
\eqref{eq:Steklov-modificato} behave continuously at $\eps=0$ in
the sense of Theorem \ref{vaithm}.
\end{theorem}

\begin{proof}
The proof of Theorem~\ref{t:main-1} follows directly by Theorem
\ref{vaithm} and Lemmas~\ref{l:E-conv}-\ref{l:compact} below.
\end{proof}

\begin{remark} \label{r:simplified} In the case of a single cuboid we can state a simplified version of Theorem \ref{t:main-1} the proof of which would follow the same lines of the proof of Theorem \ref{t:main-1} with obvious modifications. Namely, if $\Omega$, $\Omega_\eps$ are as in Remark \ref{r:cuboid} with $g_\eps$ and $g$ satisfying \eqref{eq:assumptions} with $g_\eps$ and $g$ in place of $g_{\eps,j}$ and $g_j$ respectively, then $S_{\Delta,\eps} \CC  S_{\Delta}$ and $S_{D^2,\eps} \CC  S_{D^2}$ with respect to the operators $E_{\eps}$ defined in Remark \ref{r:cuboid}. In particular, the spectra of \eqref{eq:Steklov} and \eqref{eq:Steklov-modificato} behave continuously at $\eps=0$ in the sense of Theorem \ref{vaithm}.

Note that although in this specific case $\Omega_\eps$ and $\Omega$ are not of class $C^{1,1}$ but only piecewise $C^{1,1}$, each of them is of class $C^{0,1}$ and satisfies the uniform outer ball condition.
\end{remark}

In order to prove Lemma~\ref{l:E-conv} and Lemma~\ref{l:compact},  we need a number of preliminary  technical results.

To begin with, we give the  definition of a map which acts
between the spaces $V(\Omega)$, $V(\Omega_\eps)$ in a reversed way
with respect to $E_\eps$.
For any $w\in V(\Omega_\eps)$ put
\begin{equation} \label{eq:hat-w}
\widehat w_{\eps,j}(x)=
\begin{cases} w_j(\Psi_{\eps,j}^{-1}(x)), & \qquad \text{if } x\in \Omega\cap V_j \\[8pt]
 0, & \qquad \text{if } x\in \Omega\setminus V_j \, .
\end{cases}
\end{equation}
for any $j\in \{1,\dots,s'\}$ and $w_j:= \psi_j w$ for any $j\in \{1,\dots,s\}$. We define
\begin{equation}\label{inversodiE}
B_\eps w:=\sum_{j=1}^{s'} \widehat w_{\eps,j}+\sum_{j=s'+1}^s w_j.
\end{equation}
 In this way we have constructed a map
$B_\eps:V(\Omega_\eps)\to V(\Omega)$.

\begin{lemma} \label{l:nirenberg} Let $\mathcal A$ be an atlas. Let
$\{\Omega_\eps\}_{0<\eps\le\eps_0}$ be a family of domains of
class $C^{1,1}(\mathcal A)$ and $\Omega$ a domain of class
$C^{1,1}(\mathcal A)$. Assume the validity of  condition \eqref{eq:assumptions}.
Then, up to shrink $\eps_0$ if necessary,
\begin{equation*}
\|B_\eps w\|_{H^2(\Omega)} \le C\|\Delta w\|_{L^2(\Omega_\eps)}
\qquad \text{for any } w\in V(\Omega_\eps) \ \text{and} \ \eps\in
(0,\eps_0]
\end{equation*}
where $C>0$ is a constant independent of $\eps$.
\end{lemma}

\begin{proof} Throughout the proof we denote by $C$ a general constant independent of $\eps$ which may vary from line to line.
Let $w\in V(\Omega_\eps)$. Since by \eqref{eq:assumptions} $|\Omega_\eps|\to |\Omega|$ as $\eps\to 0$ and the Poincar\'e constant depends only on the dimension of the space and the volume of the domain (see the comment before Lemma \ref{complete}), we deduce by \eqref{poin}  that
possibly  shrinking $\eps_0$
\begin{align} \label{eq:Faber-Krahn}
 \int_{\Omega_\eps} |\nabla w|^2 dx \le C \|\Delta w\|_{L^2(\Omega_\eps)}^2\  \  \text{and} \ \
\int_{\Omega_\eps} w^2\, dx \le C\|\Delta w\|_{L^2(\Omega_\eps)}^2
\end{align}
for all $\eps \in (0,\eps_0]$.
Define $w_{j}:=\psi_j w$. Then by \eqref{eq:Faber-Krahn} and the fact that
$\Delta w_{j}=w \Delta \psi_j+2\nabla \psi_j\cdot \nabla w+\psi_j \Delta w$, it follows
\begin{equation} \label{eq:stima-Delta}
\|\Delta w_{j}\|_{L^2(\Omega_\eps)}\le C\|\Delta w\|_{L^2(\Omega_\eps)} \, .
\end{equation}
In the rest of the proof we fix $j\in \{1,\dots,s'\}$.
By direct computation one can show that the function $\widehat w_{\eps,j}$ defined in \eqref{eq:hat-w} satisfies
\begin{equation*}
{\rm div}\left(A_\eps(y) \nabla \widehat w_{\eps,j}(y)\right)=[{\rm det}(D\Psi_{\eps,j}(\Psi_{\eps,j}^{-1}(y)))]^{-1}\ \Delta w_{j}(\Psi_{\eps,j}^{-1}(y))
\qquad \text{for any } y\in \Omega\cap V_j
\end{equation*}
where $A_\eps(y):=[{\rm det}(D\Psi_{\eps,j}(\Psi_{\eps,j}^{-1}(y)))]^{-1} [D\Psi_{\eps,j}(\Psi_{\eps,j}^{-1}(y))][D\Psi_{\eps,j}(\Psi_{\eps,j}^{-1}(y))]^T$.
We may think to the function $\Phi_{\eps,j}$ as defined over all the set $\{(x',x_N)\in \R^N:x'\in W_j\, , x_N\le g_{\eps,j}(x')\}$ since we may
extend trivially the function $h_{\eps,j}$ over the same set.
Let us introduce the operator $L_\eps v:={\rm div}\left(A_\eps(y) \nabla v\right)$ for any function $v\in V(\Omega)$.
Then we have
\begin{align} \label{eq:L-eps}
\int_{\Omega\cap V_j} |L_\eps \widehat w_{\eps,j}|^2 dy=\int_{\Omega_\eps\cap V_j} |\Delta w_{j}|^2 \ [{\rm det}(D\Psi_{\eps,j})]^{-1} dx
\end{align}
Combining \eqref{eq:stima-Delta}, \eqref{eq:L-eps} and \eqref{eq:stima-Jacob}, we deduce that
\begin{align} \label{eq:bound}
\int_{\Omega\cap V_j} |L_\eps \widehat w_{\eps,j}|^2 dy\le C \|\Delta w_{j}\|_{L^2(\Omega_\eps)}^2
 \le C\|\Delta w\|_{L^2(\Omega_\eps)}^2 \qquad \text{for any } \eps\in (0,\eps_0] \, .
\end{align}

In order to complete the proof of the lemma we prove the following claim: there exists a constant $C$ independent of $\eps$ such that
\begin{equation} \label{eq:stima-regolarita}
\|u\|_{H^2(\Omega)} \le C \|L_\eps u\|_{L^2(\Omega\cap V_j)} \qquad \text{for any } u\in V(\Omega)\, , \, \text{supp}(u) \subset V_j \, ,  \text{and} \ \eps\in (0,\eps_0] \, .
\end{equation}
In order to do that, we proceed by using a standard method from the theory of regularity.


For any $u\in V(\Omega)$ with $\text{supp}(u) \subset V_j$, we define the function $f_\eps\in L^2(r_j^{-1}(\Sigma_j))$ by setting
$f_\eps(y)=-L_\eps u(y)$ if $y\in \Omega\cap V_j$ and
$f_\eps(y)=0$ if $y\in r_j^{-1}(\Sigma_j) \setminus (\Omega \cap V_j) $ where
$$
\Sigma_j:=\{(x',x_N)\in \R^N: x'\in W_j\, ,
x_N\le g_j(x')\} \, .
$$
We define the deformation
\begin{align*}
& \widetilde\Phi_j:\Sigma_j \to W_j \times (-\infty,0] \\
& \widetilde\Phi_j(x',x_N):=(x',x_N-g_j(x'))
\end{align*}
and the function $\tilde u :  W_j \times (-\infty,0]\to \R$ which is defined by $\tilde
u(y)=u(r_j^{-1}(\widetilde\Phi_j^{-1}(y)))$ if $y\in
\widetilde\Phi_j(r_j(\Omega\cap V_j))$ and $\tilde u(y)=0$ if
$y\in (W_j\times (-\infty,0))\setminus
\widetilde\Phi_j(r_j(\Omega\cap V_j))$.

We observe that
$\tilde u\in H^2(W_j\times (-\infty,0))\cap H^1_0(W_j \times (-\infty,0))$ and its support is a compact set contained in $W_j\times (-\infty,0]$ thus allowing
to extend it trivially over all $\R^{N-1}\times (-\infty,0)$; we still denote by $\tilde u$ this trivial extension.
Moreover, we set  $\widetilde\Psi_j(z):=\widetilde\Phi_j(r_j(z))$ for all
$z\in r_j^{-1}(\Sigma_j)$, and we note that $\tilde u$ solves the equation
\begin{equation} \label{eq:equation}
-{\rm div}\left(\tilde A_\eps(y) \nabla \tilde u\right)=\tilde f_\eps \qquad \text{in } W_j\times (-\infty,0)
\end{equation}
where
$$
\tilde A_\eps(y):=[{\rm det}(D\widetilde\Psi_j(\widetilde\Psi_j^{-1}(y)))]^{-1} [D\widetilde\Psi_j(\widetilde\Psi_j^{-1}(y))]
A_\eps(\widetilde\Psi_j^{-1}(y))[D\widetilde\Psi_j(\widetilde\Psi_j^{-1}(y))]^T
$$
and $\tilde f_\eps(y):=[{\rm
det}(D\widetilde\Psi_j(\widetilde\Psi_j^{-1}(y)))]^{-1}\
f_\eps(\widetilde\Psi_j^{-1}(y))$ for any $y\in W_j\times
(-\infty,0)$.

We point out that, in the definition of $\tilde A_\eps$, the function
$A_\eps$ has been  extended by zero  over the set
$r_j^{-1}(\Sigma_j)$.

Then, we perform the Nirenberg translation method. In order to estimate second derivatives of $\tilde u$, we fix
arbitrarily $i\in \{1,\dots,N-1\}$ and we define $\tilde u_h(x):=\tilde u(x+he_i)$ for any $h\in \R$ where we  have omitted for simplicity the index $i$.
We use the same notation for any function $v\in H^1_0(\R^{N-1}\times (-\infty,0))$.
Similarly, we denote by $(\tilde A_\eps)_h$ the translation of $\tilde A_\eps$ where $\tilde A_\eps$ is extended by zero to the whole of
$\R^{N-1}\times (-\infty,0)$ (the fact that this extension may not be Lipschitz continuous in $\R^{N-1}\times (-\infty,0)$ is not relevant in what follows).

Fix $v\in H^1_0(W_j\times (-\infty,0))$ such that ${\rm supp}(v)$ is a compact set contained in $W_j\times (-\infty,0]$. Then we have for $|h|$
small enough  $v_{-h}\in H^1_0(W_j\times (-\infty,0))$ and
\begin{align} \label{eq:nirenberg-1}
  & \int_{W_j\times (-\infty,0)} \left(\tilde A_\eps \nabla (\tilde u_h-\tilde u)\right)
\cdot \nabla v \, dy\\
\notag & =\int_{W_j\times (-\infty,0)} \tilde f_\eps (v_{-h}-v) \, dy
-\int_{W_j\times (-\infty,0)} \left[\left((\tilde A_\eps)_h-\tilde A_\eps\right) \nabla \tilde u_h\right]
\cdot \nabla v \, dy
\end{align}
Now we choose $v=\tilde u_h-\tilde u$ in \eqref{eq:nirenberg-1} and we proceed by estimating the term on the right hand side of it.

In the rest of proof we denote by $C$ positive constants independent of $h$ and $\eps$ which may vary from line to line.

Let $\omega\subset \R^{N-1}$ and $\delta>0$ be such that $\omega$ is open,
$$
\text{supp}(\tilde u)+he_i \subset \omega\times (-\infty,0] \quad \text{and} \quad (\overline\omega\times (-\infty,0])+he_i \subset W_j\times (-\infty,0]
$$
for any $h\in (-\delta,\delta)$.  In this way we can write
\begin{equation} \label{eq:lip-1}
|(\tilde A_\eps)_h(y)-\tilde A_\eps(y)|\le C |h| \max_{n,m,l\in \{1,\dots,N\}} \left\|\frac{\partial(\tilde A_\eps)_{nm}}{\partial y_l}\right\|
_{L^\infty(W_j\times (-\infty,0))} \quad \forall y\in \omega\times (-\infty,0)\, , \forall |h|<\delta
\end{equation}
where $|A|:=\sup_{x\in \R^N\setminus\{0\}} |Ax|/|x|$ for any $N\times N$ matrix.
With this choice of $\omega$ and $\delta$, exploiting \eqref{eq:lip-1}, \eqref{eq:assumptions}, the definitions
of $\tilde A_\eps$, $\Psi_{\eps,j}$, $\Phi_{\eps,j}$, the fact that
$\tilde A_\eps(y',y_N)=I_N$ if $y_N<-(\hat k+1)\kappa_\eps<\tilde g_{\eps,j}(y')-g_j(y')$ (here $I_N$ denotes the $N\times N$ identity matrix),
and \eqref{uniinf} we have
\begin{align} \label{eq:Ah}
& \left|\int_{W_j\times (-\infty,0)} \left[\left((\tilde A_\eps)_h-\tilde A_\eps\right) \nabla \tilde u_h\right]
\cdot \nabla (\tilde u_h-\tilde u) \, dy\right|=\left|\int_{\omega\times (-\infty,0)} \left[\left((\tilde A_\eps)_h-\tilde A_\eps\right) \nabla \tilde u_h\right]
\cdot \nabla (\tilde u_h-\tilde u) \, dy\right| \\
\notag &
=\left|\int_{\omega} \left(\int_{-(\hat k+1)\kappa_\eps}^0 \left[\left((\tilde A_\eps)_h-\tilde A_\eps\right) \nabla \tilde u_h\right]
\cdot \nabla (\tilde u_h-\tilde u)\, dy_N \right) \, dy'\right| \\
\notag  & \le C|h|\  o(\kappa_\eps^{-1/2})  \left[\int_\omega\left(\int_{-(\hat k+1)\kappa_\eps}^0
\left|\nabla\Big(\tilde u_h-\tilde u\Big) \right|^2
dy_N \right) dy'
\right]^{1/2}   \times \\
\notag  & \qquad \qquad \times \left[\int_\omega (\hat k+1)\kappa_\eps \ \left\| |
\nabla \tilde u_h(y',\cdot)|
\right\|^2_{L^\infty(-\infty,0)} dy' \right]^{1/2}\\
\notag  & \le C|h|\ o(1) \ \|\tilde u_h-\tilde u\|_{H^1(W_j\times (-\infty,0))}
\left[\int_\omega \left\| \left|\nabla \tilde u_h(y',\cdot) \right|
\right\|^2_{H^1(-\infty,0)} dy' \right]^{1/2} \\
\notag  & \le C|h|\ o(1) \cdot \|\tilde u_h-\tilde u\|_{H^1(W_j\times (-\infty,0))} \|\tilde u\|_{H^2(W_j\times (-\infty,0))}
\end{align}
where here and in the rest of the proof, $o(\kappa_\eps^{-1/2})$ and $o(1)$ denote functions depending only on $\eps$ but independent of $h$
which have the prescribed asymptotic behavior as $\eps\to 0$.
We point out that the estimate $\left\| |
\nabla \tilde u_h(y',\cdot)|
\right\|^2_{L^\infty(-\infty,0)}\le C \left\| |
\nabla \tilde u_h(y',\cdot)|
\right\|^2_{H^1(-\infty,0)}$ has been obtained by exploiting the fact that there exists $M>0$ such that
$\text{supp}(\tilde u_h)\subset W_j\times (-M,0]$ for any $h\in (-\delta,\delta)$ and applying \eqref{uniinf} on the interval $(-M,0)$.
On the other hand we also have that
\begin{equation} \label{eq:f-eps}
\left|\int_{W_j\times (-\infty,0)} \tilde f_\eps [(\tilde u_h-\tilde u)_{-h}-(\tilde u_h-\tilde u)] \, dy \right|\le C|h|
\ \|\tilde f_\eps\|_{L^2(W_j\times (-\infty,0))}
\ \|\tilde u_h-\tilde u\|
_{H^1(W_j\times (-\infty,0))}
\end{equation}
Let us consider the left hand side of \eqref{eq:nirenberg-1}. Since by \eqref{eq:assumptions}, $\tilde A_\eps\to I_N$ in $L^\infty(W_j\times (-\infty,0))$,
then up to shrink $\eps_0$ if necessary we have
\begin{equation} \label{eq:ellipticity}
\sum_{n,m=1}^N (\tilde A_\eps(y))_{nm} \ \xi_n \xi_m \ge \frac 12 |\xi|^2  \qquad \text{for any } y\in W_j\times (-\infty,0)\, , \xi\in \R^N \ \text{and}
\ \eps\in (0,\eps_0] \, .
\end{equation}
Therefore we have that
\begin{align} \label{eq:left-hand-side}
& \int_{W_j\times (-\infty,0)} \left(\tilde A_\eps(y)\nabla (\tilde u_h-\tilde u)\right)
\cdot \nabla(\tilde u_h-\tilde u) \, dy \\
\notag & \qquad \ge \frac 12 \int_{W_j\times (-\infty,0)} |\nabla(\tilde u_h-\tilde u)|^2 dy \ge
C\|\tilde u_h-\tilde u\|_{H^1(W_j\times (-\infty,0))}^2 \, .
\end{align}
Combining \eqref{eq:Ah}-\eqref{eq:left-hand-side} with \eqref{eq:nirenberg-1} we obtain
\begin{align*}
\left\|\frac{\tilde u_h-\tilde u}h\right\|_{H^1(W_j\times (-\infty,0))}\!\! \le o(1)\cdot  \|\tilde u\|_{H^2(W_j\times (-\infty,0))}
\!+\! C \|\tilde f_\eps\|_{L^2(W_j\times (-\infty,0))} \ \ \  \forall 0<|h|<\delta \, , \forall \eps\in (0,\eps_0] \, .
\end{align*}
Passing to the limit as $h\to 0$ we obtain
\begin{equation} \label{eq:1-N-1}
\left\|\frac{\partial \tilde u}{\partial x_i}\right\|_{H^1(W_j\times (-\infty,0))} \le
 o(1)\cdot  \|\tilde u\|_{H^2(W_j\times (-\infty,0))}
+C \|\tilde f_\eps\|_{L^2(W_j\times (-\infty,0))} \quad  \forall \eps\in (0,\eps_0]
\end{equation}
for any $i\in \{1,\dots,N-1\}$.

Since $\tilde u$ solves equation \eqref{eq:equation}, using a standard argument (see e.g. page 319 in \cite{Evans}) one can show that the validity of \eqref{eq:1-N-1} for any $i\in \{1,\dots,N-1\}$ implies
the validity of \eqref{eq:1-N-1} also for $i=N$. Therefore we may conclude that up to shrink $\eps_0$ if necessary, for any $\eps\in (0,\eps_0]$ we have
\begin{equation*} 
\left\|D^\alpha \tilde u\right\|_{L^2(W_j\times (-\infty,0))}\!\! \le
 o(1)\cdot  \|\tilde u\|_{H^2(W_j\times (-\infty,0))}
+C \|\tilde f_\eps\|_{L^2(W_j\times (-\infty,0))} \ \forall \alpha\in \N^N, |\alpha|=2, \ \forall \eps\in (0,\eps_0]
\end{equation*}
and, in turn,
\begin{equation}  \label{eq:second-derivatives}
\|\tilde u\|_{H^2(W_j\times (-\infty,0))}\le C\left(\|\tilde u\|_{H^1(W_j\times (-\infty,0))}+\|\tilde f_\eps\|_{L^2(W_j\times (-\infty,0))}\right) \, .
\end{equation}
Testing \eqref{eq:equation} with $\tilde u$, using
\eqref{eq:ellipticity} and the Poincar\'e inequality in $W_j\times
(-\infty,0)$, we deduce that
\begin{equation*} 
\|\tilde u\|_{H^1(W_j\times (-\infty,0))}\le C
 \|\tilde f_\eps\|_{L^2(W_j\times (-\infty,0))} \quad  \forall \eps\in (0,\eps_0]
\end{equation*}
which combined with \eqref{eq:second-derivatives} gives
\begin{equation*} 
\|\tilde u\|_{H^2(W_j\times (-\infty,0))}\le C
 \|\tilde f_\eps\|_{L^2(W_j\times (-\infty,0))} \quad  \forall \eps\in (0,\eps_0] \, .
\end{equation*}
Coming back to $V_j$ the last estimate proves \eqref{eq:stima-regolarita}.

Applying \eqref{eq:stima-regolarita} to $\widehat w_{\eps,j}$ and using \eqref{eq:bound} it follows that
\begin{equation} \label{eq:norma-H2}
\|\widehat w_{\eps,j}\|_{H^2(\Omega)}\le C\|L_\eps \widehat
w_{\eps,j}\|_{L^2(\Omega\cap V_j)}\le C\|\Delta
w\|_{L^2(\Omega_\eps)}
\end{equation}
for any $\eps\in (0,\eps_0]$.

The proof of the lemma then follows from \eqref{eq:hat-w}.
\end{proof}

As a consequence of Lemma \ref{l:nirenberg} we prove the following

\begin{lemma} \label{l:unif-equiv} Let $\mathcal A$ be an atlas. Let
$\{\Omega_\eps\}_{0<\eps\le\eps_0}$ be a family of domains of
class $C^{1,1}(\mathcal A)$ and $\Omega$ a domain of class
$C^{1,1}(\mathcal A)$. Assume the validity of condition
\eqref{eq:assumptions}. Then, up to shrink $\eps_0$ if
necessary,
\begin{equation*}
\|w\|_{H^2(\Omega_\eps)}\le C\|\Delta w\|_{L^2(\Omega_\eps)}
\qquad \text{for any } w\in V(\Omega_\eps) \ \text{and} \ \eps\in
(0,\eps_0]
\end{equation*}
where $C>0$ is a constant independent of $w$ and $\eps$.
\end{lemma}

\begin{proof} Let $w\in V(\Omega_\eps)$. In this proof we use the
same notations introduced in the proof of Lemma \ref{l:nirenberg}.
By $C$ we denote positive constants independent of $w$ and $\eps$ which
may vary from line to line.

For any $j\in \{1,\dots,s'\}$ we have
\begin{align} \label{eq:compute-bis}
\frac{\partial^2 w_j}{\partial x_n \partial x_m}(x)&=\sum_{k=1}^N \left\{
 \left[ \sum_{l=1}^N \frac{\partial^2 \widehat
w_{\eps,j}}{\partial x_l
\partial x_k}(\Psi_{\eps,j}(x))
\frac{\partial[(\Psi_{\eps,j}(x))_l]}{\partial x_m} \right]
 \frac{\partial[(\Psi_{\eps,j}(x))_k]}{\partial x_n} \right. \\[7pt]
\notag  & \left. +\frac{\partial \widehat w_{\eps,j}}{\partial
x_k}(\Psi_{\eps,j}(x))
 \frac{\partial^2[(\Psi_{\eps,j}(x))_k]}
 {\partial x_n\partial x_m} \right \}  \, .
\end{align}
By \eqref{eq:assumptions}, \eqref{eq:compute-bis}, with an
argument similar to the one exploited in the proof of Lemma
\ref{l:1}, we infer
\begin{equation*}
\|w_j\|_{H^2(\Omega_\eps)}\le C\|\widehat
w_{\eps,j}\|_{H^2(\Omega)} \, .
\end{equation*}
The proof of the lemma now follows combining this estimate with
\eqref{eq:norma-H2}.
\end{proof}

Thanks to Lemma \ref{l:unif-equiv} we deduce that, under the
assumption \eqref{eq:assumptions}, the two norms
\begin{equation*}
u\mapsto \left(\int_{\Omega_\eps}|\Delta u|^2 dx \right)^{1/2} \,
, \qquad u\mapsto \left(\int_{\Omega_\eps}|D^2 u|^2 dx
\right)^{1/2}
\end{equation*}
defined on the space $V(\Omega_\eps)$, are \textit{uniformly equivalent} in
the sense that there exists a constant $C>0$ independent of $u$ and $\eps$
such that
\begin{equation*}
 \left(\int_{\Omega_\eps}|\Delta u|^2 dx \right)^{1/2}\le C\left(\int_{\Omega_\eps}|D^2 u|^2 dx
\right)^{1/2} \qquad \text{and} \qquad
\left(\int_{\Omega_\eps}|D^2 u|^2 dx \right)^{1/2} \le C
\left(\int_{\Omega_\eps}|\Delta u|^2 dx \right)^{1/2}
\end{equation*}
for any $u\in V(\Omega_\eps)$.

For this reason, for any $u\in V(\Omega_\eps)$, we put
\begin{equation} \label{eq:norm-V-eps}
\|u\|_{V(\Omega_\eps)}:=
\begin{cases}
\left(\int_{\Omega_\eps}|\Delta u|^2 dx \right)^{1/2}, & \qquad
\text{for problem } \eqref{eq:Steklov} \, , \\
\left(\int_{\Omega_\eps}|D^2 u|^2 dx \right)^{1/2}, & \qquad
\text{for problem } \eqref{eq:Steklov-modificato} \, .
\end{cases}
\end{equation}

For the same reason, given a set $A$, we denote by
$Q_A(\cdot,\cdot)$ and $Q_A(\cdot)$ the bilinear form and the
quadratic form respectively defined by
\begin{equation} \label{eq:QA-Laplacian}
Q_A(u,v):=\int_A \Delta u \Delta v \, dx  \qquad \text{and} \qquad
Q_A(u)=Q_A(u,u)
\end{equation}
or by
\begin{equation} \label{eq:QA-Hessian}
Q_A(u,v):=\int_A D^2 u : D^2 v \, dx  \qquad \text{and} \qquad
Q_A(u)=Q_A(u,u)
\end{equation}
depending again on the problem we are treating among
\eqref{eq:Steklov} or \eqref{eq:Steklov-modificato}.

The next lemma deals with a uniform estimate with respect to
domain perturbation of the $H^{1/2}$-norm on the boundary in terms
of the $H^1$-norm of the domain. For simplicity, here and in the sequel we denote the trace of a function and the function itself with the same symbol.

\begin{lemma} \label{l:H1/2} Let $\mathcal A$ be an atlas. Let
$\{\Omega_\eps\}_{0<\eps\le\eps_0}$ be a family of domains of
class $C^{1,1}(\mathcal A)$ and $\Omega$ a domain of class
$C^{1,1}(\mathcal A)$. Assume the validity of condition
\eqref{eq:assumptions}. Then, up to shrink $\eps_0$ if
necessary,
\begin{equation*}
\|w\|_{H^{1/2}(\partial\Omega_\eps)}\le C\|w\|_{H^1(\Omega_\eps)}
\qquad \text{for any } w\in H^1(\Omega_\eps) \ \text{and} \
\eps\in (0,\eps_0]
\end{equation*}
where $C>0$ is a constant independent of $\eps$.
\end{lemma}

\begin{proof} In this proof $C$ will denote a positive constant independent of $w$ and $\eps$
which may vary from line to line. The norm in
$H^{1/2}(\partial\Omega_\eps)$ can be defined locally using the
partition of unity $\{\psi_j\}_{1\le j\le s}$ introduced in
Subsection \ref{subsecoperatorsE}. In other words if $v\in
H^{1/2}(\partial\Omega_\eps)$, for any $j=1,\dots,s'$, one can
define $v_j=\psi_j v$ on $\partial\Omega_\eps\cap V_j$ and the
function
\begin{equation} \label{eq:v}
\bar v_{\eps,j}(y'):=v_j(r_j^{-1}(y',g_{\eps,j}(y'))) \qquad
\text{for any } y'\in W_j\subset \R^{N-1} \, .
\end{equation}
Let now $w\in H^1(\Omega_\eps)$ and still denote by $w$ its trace
on $\partial\Omega_\eps$ so that $w$ can be also seen as an
element of $H^{1/2}(\partial\Omega_\eps)$. Let $\overline{
w}_{\eps,j}$ be the corresponding function defined on $W_j$
according to \eqref{eq:v}.

Exploiting in an appropriate way the maps $\Phi_{\eps,j}$,
$\Psi_{\eps,j}$ defined in Subsection \ref{subsecoperatorsE} and
the maps $\widetilde \Phi_{\eps,j}$, $\widetilde \Psi_{\eps,j}$
defined in the proof of Lemma \ref{l:nirenberg}, since $w\in
H^1(\Omega_\eps)$, the function $\overline{ w}_{\eps,j}$ can be also
interpreted as a function defined on the whole $W_j\times
(-\infty,0)$ which belongs to $H^1(W_j\times (-\infty,0))$ and its
support is a compact set contained in $W_j\times (-\infty,0]$.

Now it is clear that the extension-by-zero of $\overline{
w}_{\eps,j}$ to the whole $\R^{N-1}\times (-\infty , 0)$, that we still denote by
$\overline{ w}_{\eps,j}$, belongs to $H^1(\R^{N-1}\times (-\infty , 0))$ and its trace
belongs to $H^{1/2}(\R^{N-1})$. Since $\overline{ w}_{\eps,j}$ is a
compactly supported $H^{1/2}(\R^{N-1})$-function, its norm may be
defined in the usual way by mean of the Fourier transform.

Hence, as possible choice (see \cite[Page 83]{Necas}) for
$H^{1/2}(\partial\Omega_\eps)$-norm of $w$, we put
\begin{equation} \label{eq:H1/2-norm}
\|w\|_{H^{1/2}(\partial\Omega_\eps)}=\left(\sum_{j=1}^{s'}
\|\overline w_{\eps,j}\|_{H^{1/2}(\R^{N-1})}^2\right)^{1/2} \, .
\end{equation}
Applying a well-known trace theorem for functions in $H^1(\R^{N-1}\times (-\infty,0))$
(see Theorem 5.1, Chap. 2 in \cite{Necas}) we obtain
\begin{equation*}
\|\overline{ w}_{\eps,j}\|_{H^{1/2}(\R^{N-1})}\le C \|\overline{
w}_{\eps,j}\|_{H^1(\R^{N-1}\times (-\infty , 0))} \, .
\end{equation*}
By \eqref{eq:assumptions} and direct computation, one can verify
that
\begin{equation*}
\|\overline{ w}_{\eps,j}\|_{H^1( \R^{N-1}\times (-\infty , 0))}\le
C\|w_j\|_{H^1(\Omega_\eps)}
\end{equation*}
The proof of the lemma now follows from \eqref{eq:H1/2-norm} and
the definition of $w_j$.
\end{proof}

\begin{lemma} \label{l:d-1-eps} Let $\mathcal A$ be an atlas. Let
$\{\Omega_\eps\}_{0<\eps\le\eps_0}$ be a family of domains of
class $C^{1,1}(\mathcal A)$ and $\Omega$ a domain of class
$C^{1,1}(\mathcal A)$. Assume the validity of condition
\eqref{eq:assumptions}. Then
\begin{equation*}
\liminf_{\eps\to 0} d_1^\eps>0
\end{equation*}
where $d_1^\eps$ is given by \eqref{eq:d1} with $\Omega_\eps$ in place of $\Omega$.
\end{lemma}

\begin{proof} In this proof we denote by $C$ a general constant independent of $\eps$ which may vary
from estimate to estimate. For any $0<\eps\le \eps_0$, let
$w_\eps$ an eigenfunction of \eqref{eq:Steklov} corresponding to
$d_1^\eps$ satisfying $\int_{\partial\Omega_\eps} (w_\eps)_\nu^2
dS=1$ and let $u_\eps:=B_\eps w_\eps$ with $B_\eps$ as in
Subsection \ref{ss:s-r}. Suppose by contradiction that
$\liminf_{\eps\to 0} d_1^\eps=0$ so that along a sequence of
values of $\eps$ converging to zero we may assume that
$d_1^\eps\to 0$. For simplicity we write this as $d_1^\eps\to 0$
as $\eps\to 0$. Then we have that along the same sequence
$\|w_\eps\|_{V(\Omega_\eps)}\to 0$ as $\eps\to 0$ since
$\int_{\Omega_\eps} |\Delta w_\eps|^2 dx=d_1^\eps
\int_{\partial\Omega_\eps} (w_\eps)_\nu^2 dS$. Applying Lemma
\ref{l:nirenberg} we deduce that $u_\eps\to 0$ in $V(\Omega)$.

According to \eqref{eq:hat-w}, for any $1\le j \le s'$, we can define $\widehat w_{\eps,j}$.
Moreover for any $1\le j\le s$ we put $w_{\eps,j}:=\psi_j w_\eps$. In this way we have that
$u_\eps=\sum_{j=1}^{s'} \widehat w_{\eps,j}+\sum_{j=s'+1}^s w_{\eps,j}$.
By direct computation one sees that
\begin{equation} \label{eq:boundary-est-1}
\frac{\partial \widehat w_{\eps,j}}{\partial x_k}(x)=\sum_{n=1}^N \frac{\partial w_{\eps,j}}{\partial x_n}(\Psi_{\eps,j}^{-1}(x))
\frac{\partial[(\Psi_{\eps,j}^{-1}(x))_n]}{\partial x_k} \qquad \text{for any } x\in \Omega\cap V_j \, .
\end{equation}
We observe that from the definition of $\Psi_{\eps,j}$ and
\eqref{eq:assumptions} we have
\begin{equation} \label{eq:PSI-1}
\frac{\partial[(\Psi_{\eps,j}^{-1}(x))_n]}{\partial x_k}\to \delta_{nk} \qquad \text{uniformly in } V_j\cap \Omega \ \ \text{as } \eps\to 0
\end{equation}
for any $n,k\in \{1,\dots,N\}$.

On the other hand, since $\int_{\partial\Omega_\eps} |\nabla w_\eps|^2 \, dS=\int_{\partial\Omega_\eps} (w_\eps)_\nu^2 \, dS=1$ being $w_\eps\equiv 0$ on
$\partial\Omega_\eps$, for any $j\in \{1,\dots,s'\}$ and $k\in \{1,\dots,N\}$,
$\int_{\partial\Omega_\eps} \left(\frac{\partial w_{\eps,j}}{\partial x_k}\right)^2 dS$ remains bounded as $\eps\to 0$. Moreover, using the local parametrizations
of $\partial\Omega_\eps$ and $\partial\Omega$ combined with \eqref{eq:assumptions}, we also have
\begin{equation} \label{eq:confronto-1}
\int_{\partial\Omega\cap V_j} \left(\frac{\partial w_{\eps,j}}{\partial x_k}(\Psi_{\eps,j}^{-1}(x))\right)^2 dS
=\int_{\partial\Omega_\eps\cap V_j} \left(\frac{\partial w_{\eps,j}}{\partial x_k}\right)^2 dS+o(1)=O(1) \qquad \text{as } \eps\to 0 \, .
\end{equation}
Combining \eqref{eq:confronto-1} with \eqref{eq:PSI-1}, for any $j\in \{1,\dots,s'\}$ and $n,k\in \{1,\dots,N\}$, we obtain
\begin{align} \label{eq:delta-1}
& \int_{\partial\Omega\cap V_j} \left(\frac{\partial w_{\eps,j}}{\partial x_n}(\Psi_{\eps,j}^{-1}(x))\right)^2
\left(\frac{\partial[(\Psi_{\eps,j}^{-1}(x))_n]}{\partial x_k}\right)^2 \, dS\to 0 \quad \text{as } \eps\to 0
\end{align}
if $n\neq k$ and
\begin{align} \label{eq:delta-2}
& \int_{\partial\Omega\cap V_j} \left(\frac{\partial w_{\eps,j}}{\partial x_k}(\Psi_{\eps,j}^{-1}(x))\right)^2
\left(\frac{\partial[(\Psi_{\eps,i}^{-1}(x))_k]}{\partial x_k}\right)^2 \, dS
=\int_{\partial\Omega\cap V_j} \left(\frac{\partial w_{\eps,j}}{\partial x_k}(\Psi_{\eps,j}^{-1}(x))\right)^2 dS+o(1)
\quad \text{as } \eps\to 0 \, .
\end{align}
From the proof of Lemma \ref{l:nirenberg} one can deduce that not only $u_\eps\to 0$ in $V(\Omega)$ but also $\widehat w_{\eps,j}\to 0$ in $V(\Omega)$,
for any $j\in \{1,\dots,s'\}$. Hence from the continuity of the trace map \eqref{eq:trace} we also have that $\frac{\partial\widehat w_{\eps,j}}{\partial x_k}\to 0$
in $L^2(\partial\Omega)$ as $\eps\to 0$. From this and \eqref{eq:boundary-est-1}, \eqref{eq:delta-1}, it follows that
$\frac{\partial w_{\eps,j}}{\partial x_k}(\Psi_{\eps,j}^{-1}(x))\frac{\partial[(\Psi_{\eps,j}^{-1}(x))_k]}{\partial x_k}\to 0$ in $L^2(\partial\Omega)$
as $\eps\to 0$. This, together with \eqref{eq:confronto-1} and \eqref{eq:delta-2} yields
\begin{equation*}
\int_{\partial\Omega_\eps} \left(\frac{\partial w_{\eps,j}}{\partial x_k}\right)^2 dS\to 0 \qquad \text{as } \eps \to 0  \quad
\text{for any } j\in \{1,\dots,s'\} \ \text{and} \ k\in \{1,\dots,N\}
\end{equation*}
and, in turn, $\int_{\partial\Omega_\eps} |\nabla w_\eps|^2 dS=o(1)$ as $\eps\to 0$, a contradiction.
\end{proof}

Now, we introduce some notations which will be fundamental in the
proofs of the next lemmas. For any $y\in \Psi_{\eps,j}(\partial
\Omega_\eps\cap V_i \cap V_j)$ we put
$\Theta_{\eps,i,j}(y):=\Psi_{\eps,i}(\Psi_{\eps,j}^{-1}(y))$ in
order to define
\begin{equation}  \label{eq:Thetaij}
\Theta_{\eps,i,j}:\Psi_{\eps,j}(\partial\Omega_\eps\cap V_i\cap
V_j)\to \Theta_{\eps,i,j}(\Psi_{\eps,j}(\partial\Omega_\eps\cap
V_i\cap V_j))
\end{equation}
as a diffeomorphism between two open subsets of the manifold
$\partial\Omega$.

For any $k\in \{1,\dots,s'\}$ we define the local charts
$\Gamma_k:\partial\Omega\cap V_k\to W_k\subset \R^{N-1}$  where
$\Gamma_k(y):=P(r_k(y))$ and $P:\R^{N}\to \R^{N-1}$ is the
projection $(x',x_N)\mapsto x'$. We observe that
$\Gamma_k^{-1}:W_k \to
\partial\Omega\cap V_k$ satisfies
$\Gamma_k^{-1}(z')=r_k^{-1}(z',g_k(z'))$ for any $z'\in W_k$. Next
we introduce the map
\begin{equation} \label{eq:Upsilon}
\Upsilon_{\eps,i,j}:\Gamma_j(\Psi_{\eps,j}(\partial\Omega_\eps\cap
V_i\cap V_j))\to \Upsilon_{\eps,i,j}
(\Gamma_j(\Psi_{\eps,j}(\partial\Omega_\eps\cap V_i\cap V_j)))
\end{equation}
where $\Upsilon_{\eps,i,j}(z'):=\Gamma_i(\Theta_{\eps,i,j}(\Gamma_j^{-1}(z')))$ for any $z'\in \Gamma_j(\Psi_{\eps,j}(\partial\Omega_\eps\cap V_i\cap V_j))$.

\begin{lemma} \label{p:7} Let $\mathcal A$ be an atlas. Let
$\{\Omega_\eps\}_{0<\eps\le\eps_0}$ be a family of domains of
class $C^{1,1}(\mathcal A)$ and $\Omega$ a domain of class
$C^{1,1}(\mathcal A)$. Assume the validity of condition \eqref{eq:assumptions}.

\begin{itemize}
\item[(i)]  Let $i,j\in \{1,\dots,s'\}$.  Let $\{f_\eps\}_{0<\eps\le\eps_0}\subset L^2(\R^{N-1})$ be such
that ${\rm supp}(f_\eps)\subset \Gamma_i(\partial\Omega\cap V_i)$
for any $\eps\in (0,\eps_0]$ and let $f\in L^2(\R^{N-1})$ be such
that $f_\eps\to f$ in $L^2(\R^{N-1})$ as $\eps\to 0$. Let
$\Upsilon_{\eps,i,j}$ be as in \eqref{eq:Upsilon}.
  Define
\begin{equation*}
\tilde f_\eps(z'):=
\begin{cases}
f_\eps(\Upsilon_{\eps,i,j}(z')) & \qquad \text{if } z'\in \Gamma_j(\Psi_{\eps,j}(\partial\Omega_\eps\cap V_i\cap V_j)) \, ,\\[8pt]
0 & \qquad \text{if } z'\in \R^{N-1}\setminus
\Gamma_j(\Psi_{\eps,j}(\partial\Omega_\eps\cap V_i\cap V_j)) \, ,
\end{cases}
\end{equation*}
for any $\eps\in (0,\eps_0]$ and
\begin{equation*}
\tilde f(z'):=
\begin{cases}
f(\Gamma_i(\Gamma_j^{-1}(z'))) & \qquad \text{if } z'\in \Gamma_j(\partial\Omega\cap V_i\cap V_j) \, , \\[8pt]
0 & \qquad \text{if } z' \in \R^{N-1}\setminus
\Gamma_j(\partial\Omega\cap V_i\cap V_j) \, .
\end{cases}
\end{equation*}
Then $\tilde f_\eps\to \tilde f$ in $L^2(\R^{N-1})$ as $\eps\to
0$.

\item[(ii)] Let $\{\omega_\eps\}_{0<\eps\le \eps_0}\subset
L^2(\partial\Omega)$ be such that ${\rm supp}(\omega_\eps)\subset
\partial\Omega \cap V_i$ for any $\eps\in (0,\eps_0]$, for some $i\in
\{1,\dots,s'\}$. Suppose that there exists $\omega\in
L^2(\partial\Omega)$ such that $\omega_\eps\to \omega$ in
$L^2(\partial\Omega)$ as $\eps\to 0$. For $j\in \{1,\dots,s'\}$
let $\Theta_{\eps,i,j}$ be as in \eqref{eq:Thetaij}.
For any $\eps\in (0,\eps_0]$ define the function
\begin{equation*}
\tilde\omega_\eps(y):=
\begin{cases}
\omega_\eps(\Theta_{\eps,i,j}(y)) & \qquad \text{if } y\in \Psi_{\eps,j}(\partial\Omega_\eps\cap V_i\cap V_j)\, , \\
0 & \qquad \text{if } y\in \partial\Omega \setminus
\Psi_{\eps,j}(\partial\Omega_\eps\cap V_i\cap V_j)  \, .
\end{cases}
\end{equation*}
Then $\tilde\omega_\eps\to \omega\chi_{\partial\Omega\cap V_i\cap
V_j}$ in $L^2(\partial\Omega)$ as $\eps\to 0$.
\end{itemize}
\end{lemma}

\begin{proof} We divide the proof of the lemma into three steps.
Throughout this proof, $C$ denotes a constant independent of
$\eps$ which may vary from line to line.

{\bf Step 1.} In this step we prove (i) when $f_\eps\equiv f$ for
any $\eps\in (0,\eps_0]$.

For any $\sigma>0$ let $\varphi_\sigma\in C^0(\R^{N-1})$ be such
that $\|f-\varphi_\sigma\|_{L^2(\R^{N-1})}<\sigma$ and ${\rm
supp}(\varphi_\sigma)\subset \Gamma_i(\partial\Omega\cap V_i)$.

Similarly to the statement of the lemma we define
\begin{equation*}
\tilde\varphi_{\sigma,\eps}(z'):=
\begin{cases}
\varphi_\sigma(\Upsilon_{\eps,i,j}(z')) & \qquad \text{if } z'\in \Gamma_j(\Psi_{\eps,j}(\partial\Omega_\eps\cap V_i\cap V_j)) \, ,\\[8pt]
0 & \qquad \text{if } z'\in \R^{N-1}\setminus
\Gamma_j(\Psi_{\eps,j}(\partial\Omega_\eps\cap V_i\cap V_j)) \, .
\end{cases}
\end{equation*}
and
\begin{equation*}
\tilde \varphi_\sigma(z'):=
\begin{cases}
\varphi_\sigma(\Gamma_i(\Gamma_j^{-1}(z'))) & \qquad \text{if } z'\in \Gamma_j(\partial\Omega\cap V_i\cap V_j) \, , \\[8pt]
0 & \qquad \text{if } z'\in \R^{N-1}\setminus
\Gamma_j(\partial\Omega\cap V_i\cap V_j) \, .
\end{cases}
\end{equation*}
By \eqref{eq:assumptions} we deduce that
$\tilde\varphi_{\eps,\sigma}\to \tilde\varphi_\sigma$ almost
everywhere in $\R^{N-1}$ and that up to shrink $\eps_0$ if
necessary, there exists a compact set in $\R^{N-1}$ which contains
the support of $\tilde\varphi_{\sigma,\eps}$ for any $\eps\in
(0,\eps_0]$. Therefore by the Dominated Convergence Theorem we
have that
\begin{equation} \label{eq:conv-varphi}
\tilde\varphi_{\sigma,\eps}\to \tilde \varphi_\sigma\qquad
\text{in } L^2(\R^{N-1}) \ \ \text{as } \eps\to 0 \, .
\end{equation}

On the other hand with a change of variable, by
\eqref{eq:assumptions} we obtain
\begin{align} \label{eq:conv-f-varphi-1}
\|\tilde f_\eps-\tilde \varphi_{\sigma,\eps}\|_{L^2(\R^{N-1})}^2 &
=\int_{\Gamma_j(\Psi_{\eps,j}(\partial\Omega_\eps\cap V_i\cap V_j))} |f(\Upsilon_{\eps,i,j}(z'))-\varphi_\sigma(\Upsilon_{\eps,i,j}(z'))|^2 dz' \\
\notag & \le C \int_{\Gamma_i(\partial\Omega\cap V_i)}
|f(x')-\varphi_\sigma(x')|^2 dx' \le C\sigma^2
\end{align}
and
\begin{align} \label{eq:conv-f-varphi-2}
\|\tilde f-\tilde\varphi_\sigma\|_{L^2(\R^{N-1})}^2 & =
\int_{\Gamma_j(\partial\Omega\cap V_i\cap V_j)} |f(\Gamma_i(\Gamma_j^{-1}(z')))-\varphi_\sigma(\Gamma_i(\Gamma_j^{-1}(z')))|^2 dz'\\
\notag & \le C\int_{\Gamma_i(\partial\Omega\cap V_i)}
|f(x')-\varphi_\sigma(x')|^2 dx'\le C\sigma^2 \, .
\end{align}
Combining \eqref{eq:conv-f-varphi-1}, \eqref{eq:conv-f-varphi-2}
we have
\begin{align*}
& \|\tilde f_\eps-\tilde f\|_{L^2(\R^{N-1})} \\
& \le \|\tilde f_\eps-\tilde
\varphi_{\sigma,\eps}\|_{L^2(\R^{N-1})}+\|\tilde
\varphi_{\sigma,\eps}-\tilde\varphi_\sigma\|_{L^2(\R^{N-1})}
+\|\tilde\varphi_\sigma-\tilde f\|_{L^2(\R^{N-1})} \le 2\sqrt C
\sigma+\|\tilde\varphi_{\sigma,\eps}-\tilde\varphi_\sigma\|_{L^2(\R^{N-1})}
\end{align*}
and the proof of Step 1 follows by \eqref{eq:conv-varphi} and the
arbitrariness of $\sigma$.

{\bf Step 2.} In this step we prove (i) in the general case. To
this purpose we define
\begin{equation*}
\tilde g_\eps(z'):=
\begin{cases}
f(\Upsilon_{\eps,i,j}(z')) & \qquad \text{if } z'\in \Gamma_j(\Psi_{\eps,j}(\partial\Omega_\eps\cap V_i\cap V_j)) \, ,\\[8pt]
0 & \qquad \text{if } z'\in \R^{N-1}\setminus
\Gamma_j(\Psi_{\eps,j}(\partial\Omega_\eps\cap V_i\cap V_j))
\end{cases}
\end{equation*}
for any $\eps\in (0,\eps_0]$. By Step 1 we have
\begin{equation} \label{eq:g-eps}
\tilde g_\eps\to \tilde f \qquad \text{in } L^2(\R^{N-1}) \ \
\text{as } \eps \to 0 \, .
\end{equation}
On the other hand by \eqref{eq:assumptions},
\begin{align} \label{eq:g-f-eps}
\|\tilde f_\eps-\tilde g_\eps\|_{L^2(\R^{N-1})}^2
&=\int_{\Gamma_j(\Psi_{\eps,j}(\partial\Omega_\eps\cap V_i\cap
V_j))}
|f_\eps(\Upsilon_{\eps,i,j}(z'))-f(\Upsilon_{\eps,i,j}(z'))|^2 dz' \\
\notag & \le C \int_{\Gamma_i(\partial\Omega\cap V_i)}
|f_\eps(x')-f(x')|^2 dx'=C\|f_\eps-f\|_{L^2(\R^{N-1})} \, .
\end{align}
Combining \eqref{eq:g-eps} and \eqref{eq:g-f-eps} we obtain
$$
\|\tilde f_\eps-\tilde f\|_{L^2(\R^{N-1})}\le \|\tilde
f_\eps-\tilde g_\eps\|_{L^2(\R^{N-1})}+\|\tilde g_\eps-\tilde
f\|_{L^2(\R^{N-1})}
$$
the proof of Step 2 follows.

{\bf Step 3.} In this step we prove (ii).
We set
\begin{equation*}
f_\eps(z')=
\begin{cases}
\omega_\eps(\Gamma_i^{-1}(z')) & \qquad \text{if }  z'\in
\Gamma_i(\partial\Omega\cap V_i) \, , \\
0 & \qquad \text{if } z'\in \R^{N-1}\setminus
\Gamma_i(\partial\Omega\cap V_i) \, ,
\end{cases}
\end{equation*}
\begin{equation*}
f(z')=
\begin{cases}
\omega(\Gamma_i^{-1}(z')) & \qquad \text{if }  z'\in
\Gamma_i(\partial\Omega\cap V_i) \, , \\
0 & \qquad \text{if } z'\in \R^{N-1}\setminus
\Gamma_i(\partial\Omega\cap V_i) \, ,
\end{cases}
\end{equation*}
and
\begin{equation*}
\phi_\eps(y)=
\begin{cases}
f_\eps(\Gamma_i(\Theta_{\eps,i,j}(y)))  & \qquad \text{if } y\in \Psi_{\eps,j}(\partial\Omega_\eps\cap V_i\cap V_j) \, ,\\[8pt]
0 & \qquad \text{if } y\in \partial\Omega \setminus
\Psi_{\eps,j}(\partial\Omega_\eps\cap V_i\cap V_j),
\end{cases}
\end{equation*}
\begin{equation*}
\phi(y)=
\begin{cases}
 f(\Gamma_i(y)) & \qquad \text{if } y\in \partial\Omega\cap V_i\cap V_j \, , \\[8pt]
0 & \qquad \text{if } y\in \partial\Omega\setminus (V_i\cap V_j)
\, .
\end{cases}
\end{equation*}

We observe that, using the notation of statement (i), $\tilde
f_\eps(z')=\phi_\eps(\Gamma_j^{-1}(z'))$
and $\tilde f(z')=\phi(\Gamma_j^{-1}(z'))$ for any
$z'\in \Gamma_j(\partial\Omega\cap V_j)$. Moreover,
\begin{align*}
& \int_{\partial\Omega} |\phi_\eps-\phi|^2 dS=\int_{\partial\Omega\cap V_j} |\phi_\eps-\phi|^2 dS=
\int_{\Gamma_j(\partial\Omega\cap V_j)} |\phi_\eps(\Gamma_j^{-1}(z'))-\phi(\Gamma_j^{-1}(z'))|^2 \sqrt{1+|\nabla_{z'} g_j(z')|^2}  \, dz' \\
& \qquad \le C\int_{\R^{N-1}} |\tilde f_\eps(z')-\tilde f(z')|^2
\, dz'=C \|\tilde f_\eps-\tilde f\|_{L^2(\R^{N-1})} \, .
\end{align*}

Then by (i) we have that  $\phi_\eps\to \phi$ in $L^2(\partial\Omega)$ as $\eps\to 0$. In order to conclude it suffices to observe that $\tilde\omega_\eps$ coincides with $\phi_\eps$ and
$\omega\chi_{\partial\Omega\cap V_i\cap v_j}$ coincides with
$\phi$ .
\end{proof}

Next we prove

\begin{lemma} \label{l:weak-convergence} Let $\mathcal A$ be an atlas. Let
$\{\Omega_\eps\}_{0<\eps\le\eps_0}$ be a family of domains of
class $C^{1,1}(\mathcal A)$ and $\Omega$ a domain of class
$C^{1,1}(\mathcal A)$. Assume the validity of condition \eqref{eq:assumptions}. Let $w_\eps \in V(\Omega_{\eps})$ with $0<\eps\le \eps_0$, and $w\in V(\Omega)$ be
such that $w_\eps \EC w$. If we put $u_\eps:=S_{\Delta,\eps}
w_\eps$ and $u:=S_{\Delta} w$ or $u_\eps:=S_{D^2,\eps} w_\eps$ and
$u:=S_{D^2} w$, then $B_\eps u_\eps\rightharpoonup u$ in
$V(\Omega)$ as $\eps\to 0$.
\end{lemma}

\begin{proof} We prove the lemma by using an argument which works
for both problems \eqref{eq:Steklov} and
\eqref{eq:Steklov-modificato}. In doing this, we use the notation
introduced in \eqref{eq:norm-V-eps}, \eqref{eq:QA-Laplacian},
\eqref{eq:QA-Hessian}. We divide the proof of the lemma in several
steps.

{\bf Step 1.} In this step we prove that the $V(\Omega_\eps)$-norm
of $u_\eps$ is uniformly bounded with respect to $\eps\in
(0,\eps_0]$.

First of all we observe that
\begin{align}\label{duno}
\|u_\eps\|_{V(\Omega_\eps)}^2 & =\int_{\partial\Omega_\eps}
(w_\eps)_\nu (u_\eps)_\nu \, dS \le
\left(\int_{\partial\Omega_\eps} (w_\eps)_\nu^2 \, dS\right)^{1/2}
\left(\int_{\partial\Omega_\eps} (u_\eps)_\nu^2 \, dS\right)^{1/2} \\
\notag & \le (d_1^\eps)^{-1} \left(\int_{\Omega_\eps} |\Delta w_\eps|^2
dx\right)^{1/2} \left(\int_{\Omega_\eps} |\Delta u_\eps|^2
dx\right)^{1/2}\le N (d_1^\eps)^{-1} \|w_\eps\|_{V(\Omega_\eps)}
\|u_\eps\|_{V(\Omega_\eps)} \, .
\end{align}
from which it follows that $\|u_\eps\|_{V(\Omega_\eps)}\le
N(d_1^\eps)^{-1} \|w_\eps\|_{V(\Omega_\eps)}$.

Therefore, by Lemma \ref{l:d-1-eps} and the fact that $\eps\mapsto \|w_\eps\|_{V(\Omega_\eps)}$ is bounded in view of Lemma \ref{l:1} (ii),
it follows that there exists a constant $C>0$ independent of $\eps$ such that
\begin{equation} \label{eq:bound-u-eps}
\|u_\eps\|_{V(\Omega_\eps)}\le C \qquad \text{for any } \eps\in (0,\eps_0] \, .
\end{equation}

{\bf Step 2.} We note  that $\{B_\eps u_\eps\}_{0<\eps\le \eps_0}$
is bounded in $V(\Omega)$ and, in particular, that it is weakly
convergent in $V(\Omega)$ along a sequence $\eps_n\downarrow 0$ as
$n\to +\infty$.
Indeed,
applying Lemma \ref{l:nirenberg} to $\{u_\eps\}_{0<\eps\le
\eps_0}$, it follows that $\{B_\eps u_\eps\}_{0<\eps\le \eps_0}$
is bounded in $V(\Omega)$ up to shrink $\eps_0$ if necessary.
Thus, for  a sequence $\eps_n \downarrow 0$ we have that there exists
$\tilde u\in V(\Omega)$ such that $B_{\eps_n} u_{\eps_n}
\rightharpoonup \tilde u$ in $V(\Omega)$. For simplicity in the
sequel we only write $B_\eps u_{\eps} \rightharpoonup \tilde u$ as
$\eps\to 0$ for denoting this convergence along the sequence
$\{\eps_n\}$. The purpose of the next two steps is to prove that
$\tilde u=u$. This will be done passing to the limit in the
following identity
\begin{equation} \label{eq:u_eps}
Q_{\Omega_\eps}(u_\eps,E_\eps \varphi)=\int_{\partial\Omega_\eps} (w_\eps)_\nu (E_\eps \varphi)_\nu \, dS
\qquad \text{for any } \varphi\in V(\Omega)\, .
\end{equation}

{\bf Step 3.} In this step we pass to the limit in the left hand
side of \eqref{eq:u_eps}.

 We define $K_\eps:=\left(\bigcup_{j=1}^{s'} r_j^{-1}(K_{\eps,j})\right)\cup
\left(\bigcup_{j=s'+1}^s V_j\right)$ and we split the left hand
side of \eqref{eq:u_eps} into two parts
\begin{align} \label{eq:splitting-2}
Q_{\Omega_\eps}(u_\eps,E_\eps \varphi)=Q_{K_\eps}(u_\eps,E_\eps \varphi)+Q_{\Omega_\eps\setminus K_\eps}(u_\eps,E_\eps \varphi) \, .
\end{align}
Here $K_{\eps,j}$ denotes the set defined in Subsection
\ref{subsecoperatorsE}.

We recall that  $\varphi_j=\psi_j\varphi $  and that  $E_\eps \varphi=\sum_{j=1}^{s'} \tilde \varphi_{\eps,j}+\sum_{j=s'+1}^s \varphi_j$. Thus
\begin{align*}
& Q_{K_\eps}(u_\eps,E_\eps \varphi)=\sum_{j=1}^{s'} Q_{K_\eps}(u_\eps,\tilde \varphi_{\eps,j})
+\sum_{j=s'+1}^s Q_{K_\eps}(u_\eps,\varphi_j)
=\sum_{j=1}^{s'} Q_{K_\eps\cap V_j}(u_\eps,\tilde \varphi_{\eps,j})+\sum_{j=s'+1}^s Q_{K_\eps}(u_\eps,\varphi_j)\\
& =\sum_{j=1}^{s'} Q_{(K_\eps\cap V_j)\setminus r_j^{-1}(K_{\eps,j})}(u_\eps,\tilde \varphi_{\eps,j})
+\sum_{j=1}^{s'} Q_{r_j^{-1}(K_{\eps,j})}(u_\eps,\tilde \varphi_{\eps,j})
+\sum_{j=s'+1}^s Q_{K_\eps}(u_\eps,\varphi_j) \, .
\end{align*}
But
\begin{equation*}
|Q_{(K_\eps\cap V_j)\setminus r_j^{-1}(K_{\eps,j})}(u_\eps,\tilde \varphi_{\eps,j})|\le
\left(Q_{(K_\eps\cap V_j)\setminus r_j^{-1}(K_{\eps,j})}(u_\eps)\right)^{1/2}
\left(Q_{(K_\eps\cap V_j)\setminus r_j^{-1}(K_{\eps,j})}(\tilde \varphi_{\eps,j})\right)^{1/2}=o(1)
\end{equation*}
since $Q_{(K_\eps\cap V_j)\setminus r_j^{-1}(K_{\eps,j})}(u_\eps)=O(1)$ thanks to \eqref{eq:bound-u-eps}
and $Q_{(K_\eps\cap V_j)\setminus r_j^{-1}(K_{\eps,j})}(\tilde \varphi_{\eps,j})=o(1)$
as one can see by proceeding as for \eqref{eq:pezzo-1} .
Similarly we have $Q_{(K_\eps\cap V_j)\setminus r_j^{-1}(K_{\eps,j})}(u_\eps,\varphi_j)=o(1)$.
Hence
\begin{align} \label{eq:u-E-phi}
& Q_{K_\eps}(u_\eps,E_\eps \varphi)=
\sum_{j=1}^{s'} Q_{r_j^{-1}(K_{\eps,j})}(u_\eps,\tilde \varphi_{\eps,j})
+\sum_{j=s'+1}^s Q_{K_\eps}(u_\eps,\varphi_j)+o(1) \\
\notag & =\sum_{j=1}^{s'} Q_{r_j^{-1}(K_{\eps,j})}(u_\eps,\varphi_j)+\sum_{j=s'+1}^s Q_{K_\eps}(u_\eps,\varphi_j)+o(1) \\
\notag & =\sum_{j=1}^{s'} Q_{K_\eps}(u_\eps,\varphi_j)+\sum_{j=s'+1}^s Q_{K_\eps}(u_\eps,\varphi_j)+o(1)
=Q_{K_\eps}(u_\eps,\varphi)+o(1) \, .
\end{align}
In a similar way one can treat $Q_{K_\eps}(B_\eps u_\eps,\varphi)$:
\begin{align} \label{eq:Q-K-B}
 & Q_{K_\eps}(B_\eps u_\eps,\varphi)=\sum_{j=1}^{s'} Q_{K_\eps}(\widehat u_{\eps,j},\varphi)
+\sum_{j=s'+1}^s Q_{K_\eps}(u_{\eps,j},\varphi)=\sum_{j=1}^{s'} Q_{K_\eps\cap V_j}(\widehat u_{\eps,j}, \varphi)+\sum_{j=s'+1}^s Q_{K_\eps}(u_{\eps,j},\varphi)\\
\notag & =\sum_{j=1}^{s'} Q_{(K_\eps\cap V_j)\setminus r_j^{-1}(K_{\eps,j})}(\widehat u_{\eps,j},\varphi)
+\sum_{j=1}^{s'} Q_{r_j^{-1}(K_{\eps,j})}(\widehat u_{\eps,j},\varphi)
+\sum_{j=s'+1}^s Q_{K_\eps}(u_{\eps,j},\varphi)\\
\notag & =\sum_{j=1}^{s'} Q_{(K_\eps\cap V_j)\setminus r_j^{-1}(K_{\eps,j})}(\widehat u_{\eps,j},\varphi)
+\sum_{j=1}^{s'} Q_{r_j^{-1}(K_{\eps,j})}(u_{\eps,j},\varphi)
+\sum_{j=s'+1}^s Q_{K_\eps}(u_{\eps,j},\varphi) \, .
\end{align}
From the end of the proof of Lemma \ref{l:nirenberg} we infer that $\{\widehat u_{\eps,j}\}_{0<\eps\le \eps_0}$ is bounded in $H^2(\Omega)$.

Moreover $Q_{(K_\eps\cap V_j)\setminus r_j^{-1}(K_{\eps,j})}(\varphi)=o(1)$. Therefore, as
$\eps\to 0$, we have
\begin{align*}
 & |Q_{(K_\eps\cap V_j)\setminus r_j^{-1}(K_{\eps,j})}(\widehat u_{\eps,j},\varphi) |      \le \left(Q_\Omega(\widehat u_{\eps,j})\right)^{1/2}
 \left(Q_{(K_\eps\cap V_j)\setminus r_j^{-1}(K_{\eps,j})}(\varphi)\right)^{1/2}=o(1) \, ,  \\[5pt]
 &
| Q_{(K_\eps\cap V_j)\setminus r_j^{-1}(K_{\eps,j})}(u_{\eps,j},\varphi)  |   \le \left(Q_{\Omega_\eps}(u_{\eps})\right)^{1/2}
 \left(Q_{(K_\eps\cap V_j)\setminus r_j^{-1}(K_{\eps,j})}(\varphi)\right)^{1/2}=o(1) \, .
\end{align*}
Combining this with \eqref{eq:Q-K-B}, we obtain
\begin{align} \label{eq:B-u-phi}
  Q_{K_\eps}(B_\eps u_\eps,\varphi)& =\sum_{j=1}^{s'} Q_{(K_\eps\cap V_j)\setminus r_j^{-1}(K_{\eps,j})}(u_{\eps,j},\varphi)
+\sum_{j=1}^{s'} Q_{r_j^{-1}(K_{\eps,j})}(u_{\eps,j},\varphi)
+\sum_{j=s'+1}^s Q_{K_\eps}(u_{\eps,j},\varphi)+o(1) \\
\notag & =Q_{K_\eps}(u_\eps,\varphi)+o(1) \qquad \text{as } \eps\to 0 \, .
\end{align}
By \eqref{eq:u-E-phi} and \eqref{eq:B-u-phi} it follows
\begin{equation} \label{eq:scarico-E-B}
Q_{K_\eps}(u_\eps,E_\eps\varphi)=Q_{K_\eps}(B_\eps u_\eps,\varphi)+o(1) \qquad \text{as } \eps\to 0 \, .
\end{equation}
Now, let us consider the second term on the right hand side of \eqref{eq:splitting-2}.
We have
\begin{equation} \label{eq:resto-1}
|Q_{\Omega_\eps\setminus K_\eps}(u_\eps,E_\eps \varphi) |
\le \left(Q_{\Omega_\eps}(u_\eps)\right)^{1/2} \left(Q_{\Omega_\eps\setminus K_\eps}(E_\eps \varphi)\right)^{1/2} \, .
\end{equation}
But, as $\eps\to 0$, we have
\begin{align} \label{eq:resto-2}
Q_{\Omega_\eps\setminus K_\eps}(E_\eps \varphi)\le s' \sum_{j=1}^{s'} Q_{\Omega_\eps\setminus K_\eps}(\tilde\varphi_{\eps,j})
\le s' \sum_{j=1}^{s'} Q_{(\Omega_\eps\cap V_j)\setminus r_j^{-1}(K_{\eps,j})}(\tilde\varphi_{\eps,j})=o(1) \, .
\end{align}
By \eqref{eq:resto-1} and \eqref{eq:resto-2} we arrive to
$Q_{\Omega_\eps\setminus K_\eps}(u_\eps,E_\eps \varphi)=o(1)$ as
$\eps \to 0$.

Similarly we also have $Q_{\Omega\setminus K_\eps}(B_\eps u_\eps,\varphi)=o(1)$ as $\eps \to 0$.
Inserting these two last asymptotic estimates and \eqref{eq:scarico-E-B} into \eqref{eq:splitting-2} we infer
\begin{equation}\label{eq:passaggio-01}
Q_{\Omega_\eps}(u_\eps,E_\eps\varphi)=Q_{K_\eps}(B_\eps u_\eps,\varphi)+Q_{\Omega\setminus K_\eps}(B_\eps u_\eps,\varphi)+o(1)
=Q_\Omega(B_\eps u_\eps,\varphi)+o(1) \qquad \text{as } \eps\to 0 \, .
\end{equation}
Since $B_\eps u_\eps \rightharpoonup \tilde u$ in $V(\Omega)$ then
\begin{equation} \label{eq:passaggio-1}
Q_{\Omega_\eps}(u_\eps,E_\eps\varphi)\to Q_\Omega(\tilde u,\varphi) \qquad \text{as } \eps\to 0 \, .
\end{equation}

{\bf Step 4.} The next purpose is to pass to the limit in the
right hand side of \eqref{eq:u_eps}.

First of all we observe that thanks to Lemma \ref{l:1}
\begin{align} \label{eq:sost}
& \left|\int_{\partial\Omega_\eps} (w_\eps)_\nu (E_\eps
\varphi)_\nu \, dS-\int_{\partial\Omega_\eps} (E_\eps w)_\nu
(E_\eps
\varphi)_\nu \, dS\right| \\
\notag & \qquad \le 2(d_1^\eps)^{-1} \ \|w_\eps-E_\eps
w\|_{V(\Omega_\eps)}\cdot (\|\varphi\|_{V(\Omega)}+o(1))=o(1)  \,
,
\end{align}
as $\eps\to 0$. We claim that
\begin{equation} \label{eq:claim}
\int_{\partial\Omega_\eps} (E_\eps w)_\nu (E_\eps \varphi)_\nu \,
dS\to \int_{\partial\Omega} w_\nu \varphi_\nu \, dS \qquad
\text{as } \eps\to 0 \, .
\end{equation}
In order to estimate $\int_{\partial\Omega_\eps} (E_\eps w)_\nu
(E_\eps \varphi )_\nu \, dS$ we proceed as follows:
\begin{align} \label{eq:computation-0}
& \int_{\partial\Omega_\eps} (E_\eps w)_\nu (E_\eps \varphi)_\nu
\, dS=\sum_{n,m=1}^N \int_{\partial\Omega_\eps}
\frac{\partial(E_\eps w)}{\partial x_n} \, \nu_{\eps,n}
\frac{\partial(E_\eps \varphi)}{\partial x_m} \, \nu_{\eps,m} \,
dS
\\
\notag & =\sum_{n,m=1}^N \sum_{i,j=1}^{s'}
\int_{\partial\Omega_\eps \cap V_i \cap V_j}
\frac{\partial(w_{i}(\Psi_{\eps,i}(x)))}{\partial x_n}
\frac{\partial (\varphi_{j}(\Psi_{\eps,j}(x)))}{\partial x_m} \,
\nu_{\eps,n}(x) \nu_{\eps,m}(x) \, dS
\\
\notag & =\sum_{n,m,k,l=1}^N \sum_{i,j=1}^{s'}
\int_{\Psi_{\eps,j}(\partial\Omega_\eps \cap V_i \cap V_j)}
\frac{\partial w_{i}   }{\partial
y_k} (\Psi_{\eps,i}(\Psi_{\eps,j}^{-1}(y)))\frac{\partial \varphi_{j}}{\partial y_l}(y)
W_{\eps,n,m,k,l,i,j}(y) \, dS \, .
\end{align}
where we denoted by $\nu_\eps=(\nu_{\eps,1},\dots,\nu_{\eps,N})$
the unit normal to $\partial\Omega_\eps$ and we put
\begin{equation} \label{eq:W-eps-nmj-0}
W_{\eps,n,m,k,l,i,j}(y):= \nu_{\eps,n}(\Psi_{\eps,j}^{-1}(y))
\nu_{\eps,m}(\Psi_{\eps,j}^{-1}(y))
\frac{\partial(\Psi_{\eps,i})_k}{\partial
x_n}(\Psi_{\eps,j}^{-1}(y))
\frac{\partial(\Psi_{\eps,j})_l}{\partial
x_m}(\Psi_{\eps,j}^{-1}(y)) \tfrac{\sqrt{1+|\nabla_{x'}
g_{\eps,j}(\Gamma_j(y))|^2}}{\sqrt{1+|\nabla_{x'}
g_{j}(\Gamma_j(y))|^2}} \, .
\end{equation}
By \eqref{eq:assumptions} we deduce that the trivial extension of
$W_{\eps,n,m,k,l,i,j}$ to the whole $\partial\Omega$ converges
almost everywhere to the function $\nu_n \nu_m \delta_{kn}
\delta_{lm} \ \chi_{\partial\Omega\cap V_i\cap V_j}$ and it
remains uniformly bounded as $\eps\to 0$. Then, applying
Lemma \ref{p:7} (ii) to $w$,
by \eqref{eq:computation-0} we obtain as $\eps\to 0$
\begin{align*}
\int_{\partial\Omega_\eps} (E_\eps w)_\nu (E_\eps \varphi)_\nu \,
dS\to \sum_{n,m=1}^N \sum_{i,j=1}^{s'} \int_{\partial\Omega\cap
V_i\cap V_j} \frac{\partial w_i}{\partial y_n} \frac{\partial
\varphi_j}{\partial y_m} \,  \nu_n \nu_m \,
dS=\int_{\partial\Omega} w_\nu \varphi_\nu \, , dS \, .
 \end{align*}
This completes the proof of \eqref{eq:claim}. In turn, by
\eqref{eq:sost}-\eqref{eq:claim}, we obtain
\begin{equation} \label{eq:passaggio-2}
\int_{\partial\Omega_\eps} (w_\eps)_\nu (E_\eps \varphi)_\nu \, dS
\to \int_{\partial\Omega} w_\nu \varphi_\nu \, dS \qquad \text{as
} \eps\to 0 \, .
\end{equation}

{\bf Step 5.} In this last step we complete the proof of the
lemma.

Inserting \eqref{eq:passaggio-1} and \eqref{eq:passaggio-2} into
\eqref{eq:u_eps} we deduce that
\begin{equation*}
Q_\Omega(\tilde u,\varphi)=\int_{\partial\Omega} w_\nu \varphi_\nu
\, dS \qquad \text{for any } \varphi\in V(\Omega)\, .
\end{equation*}
On the other hand, also the function $u$ is a solution of the same
variational problem and hence, by uniqueness of the solution of
such a problem, we conclude that $\tilde u=u$. In particular this
means that the weak limit $\tilde u$ does not depend on the choice
of the sequence $\eps\downarrow 0$, thus proving that the
convergence $B_\eps u_\eps\rightharpoonup u$ does not occur only
along a special sequence but as $\eps \to 0$ in the usual sense.
This completes the proof of the lemma.
\end{proof}

\begin{lemma} \label{l:E-conv} Let $\mathcal A$ be an atlas. Let
$\{\Omega_\eps\}_{0<\eps\le\eps_0}$ be a family of domains of
class $C^{1,1}(\mathcal A)$ and $\Omega$ a domain of class
$C^{1,1}(\mathcal A)$. Assume the validity of condition \eqref{eq:assumptions}.
 Let $w_\eps \in V(\Omega_{\eps})$ with $0<\eps\le \eps_0$, and $w\in V(\Omega)$ be
such that $w_\eps \EC w$. If we put $u_\eps:=S_{\Delta,\eps}
w_\eps$ and $u:=S_\Delta w$ or $u_\eps:=S_{D^2,\eps} w_\eps$ and
$u:=S_{D^2} w$ then $u_\eps \EC u$. In particular this implies
that $S_{\Delta,\eps} \EEC S_\Delta$ and $S_{D^2,\eps} \EEC
S_{D^2}$ as $\eps\to 0$ in the sense of Definition
\ref{d:EE-convergence-ArCaLo}.
\end{lemma}

\begin{proof} We use the notation of \eqref{eq:norm-V-eps},
\eqref{eq:QA-Laplacian}, \eqref{eq:QA-Hessian} in order to treat
the problems \eqref{eq:Steklov} and \eqref{eq:Steklov-modificato}
simultaneously.

Let us consider
\begin{equation} \label{eq:conv-finale}
\|u_\eps-E_\eps
u\|_{V(\Omega_\eps)}^2=\|u_\eps\|_{V(\Omega_\eps)}^2-2Q_{\Omega_\eps}(u_\eps,E_\eps
u)+\|E_\eps u\|_{V(\Omega_\eps)}^2 \, .
\end{equation}
Proceeding as in \eqref{eq:passaggio-01} and  \eqref{eq:passaggio-1} we
obtain
\begin{equation} \label{eq:conv-1}
Q_{\Omega_\eps}(u_\eps,E_\eps u)=Q_\Omega(B_\eps u_\eps,u)+o(1)\to
Q_\Omega(u,u) \, .
\end{equation}
The fact that $Q_\Omega(B_\eps u_\eps,u)\to Q_\Omega(u,u)$ is a
consequence of the weak convergence $B_\eps u_\eps \rightharpoonup
u$ in $V(\Omega)$ proved in Lemma \ref{l:weak-convergence}. On the
other hand by Lemma \ref{l:1} we have
\begin{equation} \label{eq:conv-2}
\|E_\eps u\|_{V(\Omega_\eps)}^2\to
\|u\|_{V(\Omega)}^2=Q_\Omega(u,u) \, .
\end{equation}
It remains to prove that $\|u_\eps\|_{V(\Omega_\eps)}^2\to
\|u\|_{V(\Omega)}^2$. We proceed as follows
\begin{align} \label{eq:57}
& Q_{\Omega_\eps}(u_\eps,u_\eps)=\int_{\partial\Omega_\eps}
(w_\eps)_\nu (u_\eps)_\nu \, dS=\int_{\partial\Omega_\eps} (E_\eps
w)_\nu (u_\eps)_\nu \, dS+o(1)
\end{align}
where the second identity can be obtained by proceeding exactly as
in \eqref{eq:sost} and exploiting the fact that
$\|u_\eps\|_{V(\Omega_\eps)}=O(1)$ as $\eps\to 0$ as we have shown
in \eqref{eq:bound-u-eps}.

We claim that
\begin{equation} \label{eq:claim-2}
\int_{\partial\Omega_\eps} (E_\eps w)_\nu (u_\eps)_\nu \, dS\to
\int_{\partial\Omega} w_\nu u_\nu \, dS=Q_\Omega(u,u) \, .
\end{equation}
Inserting \eqref{eq:conv-1}-\eqref{eq:claim-2} into
\eqref{eq:conv-finale} we obtain the proof of the lemma. It
remains to prove \eqref{eq:claim-2}.

In order to estimate $\int_{\partial\Omega_\eps} (E_\eps w)_\nu
(u_\eps)_\nu \, dS$ we proceed as follows:
\begin{align} \label{eq:computation}
& \int_{\partial\Omega_\eps} (E_\eps w)_\nu (u_\eps)_\nu \,
dS=\sum_{n,m=1}^N \int_{\partial\Omega_\eps} \frac{\partial(E_\eps
w)}{\partial x_n} \, \nu_{\eps,n} \frac{\partial u_\eps}{\partial
x_m} \, \nu_{\eps,m} \, dS
\\
\notag & =\sum_{n,m=1}^N \sum_{i,j=1}^{s'}
\int_{\partial\Omega_\eps \cap V_i \cap V_j}
\frac{\partial(w_{i}(\Psi_{\eps,i}(x)))}{\partial x_n}
\frac{\partial u_{\eps,j}}{\partial x_m}(x) \, \nu_{\eps,n}(x)
\nu_{\eps,m}(x) \, dS
\\
\notag & =\sum_{n,m,k,l=1}^N \sum_{i,j=1}^{s'}
\int_{\Psi_{\eps,j}(\partial\Omega_\eps \cap V_i \cap V_j)}
\frac{\partial w_{i}     }{\partial
y_k}  (\Psi_{\eps,i}(\Psi_{\eps,j}^{-1}(y))) \frac{\partial \widehat u_{\eps,j}}{\partial y_l}(y)
W_{\eps,n,m,k,l,i,j}(y) \, dS \, .
\end{align}
with $\nu_{\eps,n}$, $\nu_{\eps,m}$ and $W_{\eps,n,m,k,l,i, j}$ as in
\eqref{eq:W-eps-nmj-0}.

Applying Lemma \ref{p:7} (ii) to $w$, applying
\eqref{eq:norma-H2} to $u_\eps$, exploiting the fact that $B_\eps
u_\eps \rightharpoonup u$ in $V(\Omega)$, and proceeding as at the
end of the proof of Lemma \ref{l:weak-convergence}, as $\eps\to
0$, we have
\begin{align*}
& \int_{\partial\Omega_\eps} (E_\eps w)_\nu (u_\eps)_\nu \, dS
=\sum_{n,m=1}^N \sum_{i,j=1}^{s'} \int_{\partial\Omega\cap V_i\cap
V_j} \frac{\partial w_{i}}{\partial y_n} \frac{\partial \widehat
u_{\eps,j}}{\partial y_m} \nu_n \nu_m \,
dS+o(1)\\
& = \sum_{n,m=1}^N \sum_{i=1}^{s'} \int_{\partial\Omega\cap V_i}
\frac{\partial w_{i}}{\partial y_n} \frac{\partial(B_\eps
u_\eps)}{\partial y_m} \nu_n \nu_m \,
dS+o(1)\\
& =\sum_{n,m=1}^N \sum_{i=1}^{s'} \int_{\partial\Omega\cap V_i}
\frac{\partial w_{i}}{\partial y_n} \frac{\partial u}{\partial
y_m} \nu_n \nu_m \, dS+o(1)=\int_{\partial\Omega} w_\nu u_\nu \,  dS+o(1) \,
.
\end{align*}
This completes the proof of \eqref{eq:claim-2} and of the lemma.
\end{proof}

\begin{lemma} \label{l:compact} Let $\mathcal A$ be an atlas. Let
$\{\Omega_\eps\}_{0<\eps\le\eps_0}$ be a family of domains of
class $C^{1,1}(\mathcal A)$ and $\Omega$ a domain of class
$C^{1,1}(\mathcal A)$. Assume the validity of condition  \eqref{eq:assumptions}. Let $w_\eps \in V(\Omega_{\eps})$ with   $0<\eps\le \eps_0$ be such that
$\|w_\eps\|_{V(\Omega_\eps)}=1$. Then $\{S_{\Delta,\eps}
w_\eps\}_{0<\eps\le \eps_0}$ and $\{S_{D^2,\eps}
w_\eps\}_{0<\eps\le \eps_0}$ are precompact in the sense of
Definition \ref{d:precompact-ArCaLo}. In particular, by Definition
\ref{d:CC-convergence-ArCaLo} and Lemma \ref{l:E-conv} we have
that $S_{\Delta,\eps} \CC S_\Delta$ and $S_{D^2,\eps} \CC S_{D^2}$
as $\eps \to 0$.
\end{lemma}

\begin{proof} According to the notations of Lemma
\ref{l:weak-convergence} we put $u_\eps:=S_{\Delta,\eps} w_\eps$
for problem \eqref{eq:Steklov} and $u_\eps:=S_{D^2,\eps} w_\eps$
for problem \eqref{eq:Steklov-modificato}. Exploiting the fact
that $\|w_\eps\|_{V(\Omega_\eps)}$ is bounded we may proceed as in
the proof of Lemma \ref{l:weak-convergence} and prove that $B_\eps
u_\eps \rightharpoonup \widetilde u$ along a sequence, for some
$\widetilde u\in V(\Omega)$. We divide the remaining part of the
proof into four steps. We recall that as in the proof of Lemma
\ref{l:weak-convergence}, for any $\eps\in (0,\eps_0]$ we have
that $u_\eps$ satisfies \eqref{eq:u_eps}.

{\bf Step 1.} In this step we pass to the limit in the right hand
side of \eqref{eq:u_eps}.

We put $h_{\eps,n,i}:=\psi_i \frac{\partial w_\eps}{\partial x_n}$
for $n=1,\dots,N$ and $i=1,\dots,s'$, and we proceed as follows:
\begin{align} \label{eq:computation-0-bis}
& \int_{\partial\Omega_\eps} (w_\eps)_\nu (E_\eps \varphi)_\nu \,
dS=\sum_{n,m=1}^N \int_{\partial\Omega_\eps} \frac{\partial w_\eps
}{\partial x_n} \, \nu_{\eps,n} \frac{\partial(E_\eps
\varphi)}{\partial x_m} \, \nu_{\eps,m} \, dS
\\
\notag & =\sum_{n,m=1}^N \sum_{i,j=1}^{s'}
\int_{\partial\Omega_\eps \cap V_i \cap V_j} h_{\eps,n,i}(x)
\frac{\partial (\varphi_{j}(\Psi_{\eps,j}(x)))}{\partial x_m} \,
\nu_{\eps,n}(x) \nu_{\eps,m}(x) \, dS
\\
\notag & =\sum_{n,m,l=1}^N \sum_{i,j=1}^{s'}
\int_{\Psi_{\eps,j}(\partial\Omega_\eps \cap V_i \cap V_j)}
h_{\eps,n,i}(\Psi_{\eps,j}^{-1}(y)) \frac{\partial
\varphi_{j}}{\partial y_l}(y) W_{\eps,n,m,l,j}(y) \, dS \, .
\end{align}
where we denoted by $\nu_\eps=(\nu_{\eps,1},\dots,\nu_{\eps,N})$
the unit normal to $\partial\Omega_\eps$ and we put
\begin{equation*} 
W_{\eps,n,m,l,j}(y):= \nu_{\eps,n}(\Psi_{\eps,j}^{-1}(y))
\nu_{\eps,m}(\Psi_{\eps,j}^{-1}(y))
\frac{\partial(\Psi_{\eps,j})_l}{\partial
x_m}(\Psi_{\eps,j}^{-1}(y)) \tfrac{\sqrt{1+|\nabla_{x'}
g_{\eps,j}(\Gamma_j(y))|^2}}{\sqrt{1+|\nabla_{x'}
g_{j}(\Gamma_j(y))|^2}} \, .
\end{equation*}
By \eqref{eq:assumptions} we deduce that the trivial extension of
$W_{\eps,n,m,l,j}$ to the whole $\partial\Omega$ converges almost
everywhere to the function $\nu_n \nu_m  \delta_{lm} \
\chi_{\partial\Omega\cap V_i\cap V_j}$ and it remains uniformly
bounded as $\eps\to 0$.

Now we observe that by Lemma \ref{l:unif-equiv} and Lemma
\ref{l:H1/2}, there exists a positive constant $C$ such that
$$
\|h_{\eps,n,i}\|_{H^{1/2}(\partial\Omega_\eps)}<C \qquad \text{for
any } \eps\in (0,\eps_0] \, .
$$
Hence, if we put $\omega_{\eps,n,i}=h_{\eps,n,i}\circ
\Psi_{\eps,i}^{-1}$, by \eqref{eq:H1/2-norm}, we also deduce that
$\{\omega_{\eps,n,i}\}_{0<\eps\le \eps_0}$ is bounded in
$H^{1/2}(\partial\Omega)$ and by the compact embedding
$H^{1/2}(\partial\Omega)\subset L^2(\partial\Omega)$ we also have
that $\{\omega_{\eps,n,i}\}_{0<\eps\le \eps_0}$ is precompact in
$L^2(\partial\Omega)$. Then, along a sequence $\eps_k\downarrow 0$
we may assume that $\omega_{\eps_k,n,i}\to F_{n,i}$ in
$L^2(\partial\Omega)$ as $k\to +\infty$. For simplicity, in the
rest of the proof of the lemma we will omit the subindex $k$ and
we simply write $\omega_{\eps,n,i}\to F_{n,i}$ in
$L^2(\partial\Omega)$ as $\eps\to 0$.

But
$\omega_{\eps,n,i}(\Theta_{\eps,i,j}(y))=h_{\eps,n,i}(\Psi_{\eps,j}^{-1}(y))$
for any $y\in \Psi_{\eps,j}(\partial\Omega_\eps\cap V_i\cap V_j)$.

Then, applying Lemma \ref{p:7} (ii) to $\omega_{\eps,n,i}$,
by \eqref{eq:computation-0-bis} we obtain as $\eps\to 0$
\begin{align} \label{eq:PassLim-1}
\int_{\partial\Omega_\eps} (w_\eps)_\nu (E_\eps \varphi)_\nu \,
dS\to \sum_{n,m=1}^N \sum_{i,j=1}^{s'} \int_{\partial\Omega\cap
V_i\cap V_j} F_{n,i} \frac{\partial \varphi_j}{\partial y_m} \,
\nu_n \nu_m \, dS=\int_{\partial\Omega} (F\cdot \nu) \,
\varphi_\nu \,  dS
 \end{align}
where we put $F=\left(\sum_{i=1}^{s'}
F_{1,i},\dots,\sum_{i=1}^{s'} F_{N,i}\right)$.

{\bf Step 2.} In this step we pass to the limit in the left hand
side of \eqref{eq:u_eps}.

We can proceed exactly as in the proof of Step 3 in Lemma
\ref{l:weak-convergence} since the only fact exploited there, is
that the $V(\Omega_\eps)$-norm of $u_\eps$ can be estimated
uniformly with respect to $\eps$. In this way, one proves that
\eqref{eq:passaggio-1} holds true.

Inserting \eqref{eq:passaggio-1} and \eqref{eq:PassLim-1} into
\eqref{eq:u_eps} we obtain
\begin{equation} \label{eq:F-nu}
Q_\Omega(\widetilde u,\varphi)=\int_{\partial\Omega} (F\cdot \nu)
\, \varphi_\nu \,  dS \qquad \text{for any } \varphi \in V(\Omega)
\, .
\end{equation}

{\bf Step 3.} The purpose of this step is to pass to the limit in
the right hand side of the following identity
\begin{equation} \label{eq:u-eps-bis}
Q_{\Omega_\eps}(u_\eps,u_\eps)=\int_{\partial\Omega_\eps}
(w_\eps)_\nu (u_\eps)_\nu \, dS \, .
\end{equation}
Similarly to \eqref{eq:computation-0-bis} we can write
\begin{align} \label{eq:conv-forte}
& \int_{\partial\Omega_\eps} (w_\eps)_\nu (u_\eps)_\nu \, dS
=\sum_{n,m,l=1}^N \sum_{i,j=1}^{s'}
\int_{\Psi_{\eps,j}(\partial\Omega_\eps \cap V_i \cap V_j)}
h_{\eps,n,i}(\Psi_{\eps,j}^{-1}(y)) \frac{\partial \widehat
u_{\eps,j}}{\partial y_l}(y) W_{\eps,n,m,l,j}(y) \, dS
\end{align}
with $h_{\eps,n,i}$ and $W_{\eps,n,m,l,j}$ as in Step 1, and
$\widehat u_{\eps,j}$ as in \eqref{eq:hat-w}. By
\eqref{eq:norma-H2} we deduce that, passing to the limit along a
subsequence of the sequence $\{\eps_k\}$ introduced in Step 1, for any $j\in \{1,\dots,s'\}$ there exists a function
$U_j\in V(\Omega)$ such that $\widehat u_{\eps,j}\rightharpoonup
U_j$ in $V(\Omega)$. Using the same notations of Step 1, we simply
write $\eps\to 0$ to denote the convergence along a sequence. By
\eqref{eq:conv-forte}, we then have as $\eps\to 0$
\begin{equation} \label{eq:conv-forte-2}
\int_{\partial\Omega_\eps} (w_\eps)_\nu (u_\eps)_\nu \, dS \to
\sum_{n,m,l=1}^N \sum_{i,j=1}^{s'} \int_{\partial\Omega\cap V_i
\cap V_j} F_{n,i} \frac{\partial U_j}{\partial y_l} \nu_n \nu_m \,
dS
\end{equation}
Since $B_\eps u_\eps=\sum_{j=1}^{s'} \widehat u_{\eps,j}$ on
$\partial\Omega$, from the uniqueness of the weak limit in
$V(\Omega)$, one immediately obtains $\widetilde u=\sum_{j=1}^{s'}
U_j$ which inserted into \eqref{eq:conv-forte-2} gives
\begin{equation}  \label{eq:conv-forte-3}
\int_{\partial\Omega_\eps} (w_\eps)_\nu (u_\eps)_\nu \, dS \to
 \int_{\partial\Omega} (F\cdot \nu) \,  \widetilde u_\nu \,
dS \, .
\end{equation}

{\bf Step 4.} In this step we conclude the proof of the lemma.

Choosing $\varphi=\widetilde u$ in \eqref{eq:F-nu} and combining
this with \eqref{eq:u-eps-bis}, \eqref{eq:conv-forte-3}, we obtain
as $\eps\to 0$ along a proper sequence $\eps_k\downarrow 0$,
\begin{equation} \label{eq:conv-norm-1}
\|u_\eps\|_{V(\Omega_\eps)}^2=Q_{\Omega_\eps}(u_\eps,u_\eps)=\int_{\partial\Omega_\eps}
(w_\eps)_\nu (u_\eps)_\nu \, dS\to \int_{\partial\Omega} (F\cdot
\nu)\, \widetilde u_\nu \, dS=Q_\Omega(\widetilde u,\widetilde
u)=\|\widetilde u\|_{V(\Omega)}^2 \, .
\end{equation}
On the other hand, by \eqref{eq:u_eps}, \eqref{eq:PassLim-1},
\eqref{eq:F-nu} with $\varphi=\widetilde u$, we obtain as $\eps\to
0$ along the same sequence converging to zero,
\begin{equation} \label{eq:conv-norm-2}
Q_{\Omega_\eps}(u_\eps,E_\eps \widetilde u)\to Q_\Omega
(\widetilde u,\widetilde u)=\|\widetilde u\|_{V(\Omega)}^2 \, .
\end{equation}
Combining \eqref{eq:conv-norm-1} and \eqref{eq:conv-norm-2} with
Lemma \ref{l:1} (ii), we obtain
\begin{align*}
\|u_\eps-E_\eps \widetilde
u\|_{V(\Omega_\eps)}^2=\|u_\eps\|_{V(\Omega_\eps)}^2-2Q_{\Omega_\eps}(u_\eps,E_\eps
\widetilde u)+\|E_\eps \widetilde u\|_{V(\Omega_\eps)}^2\to 0 \, .
\end{align*}
This proves that along a sequence we have $u_\eps \EC \widetilde
u$ as $\eps\to 0$ or equivalently that
$\{S_{\Delta,\eps}w_\eps\}_{0<\eps\le \eps_0}$ and
$\{S_{D^2,\eps}w_\eps\}_{0<\eps\le \eps_0}$ are precompact in the
sense of Definition \ref{d:precompact-ArCaLo}. The proof of the lemma now follows from Lemma \ref{l:E-conv} and Definition \ref{d:CC-convergence-ArCaLo}.
\end{proof}

As a bypass product of a number of results proved in this section, we have the following proposition
which we believe has its own interest  since it clarifies even more the meaning of $E$-convergence with respect to the operators  $E_{\eps}$ used above.

\begin{proposition}\label{intersec}
 Let $\mathcal A$ be an atlas. Let
$\{\Omega_\eps\}_{0<\eps\le\eps_0}$ be a family of domains of
class $C^{1,1}(\mathcal A)$, $\Omega$ a domain of class
$C^{1,1}(\mathcal A)$ and for any $\eps\in (0,\eps_0]$ let $E_\eps$ be the map defined in \eqref{eq:E-eps}. Assume  the validity of  condition  \eqref{eq:assumptions}.
If  $u_{\eps}\in V(\Omega_{\eps})$, $u\in V(\Omega)$ is  such that $u_{\eps}\EC u$ as $\eps \to 0 $ then
\begin{equation}\label{intersec0}
\| u_{\eps}- u\|_{H^2(\Omega_{\eps}\cap \Omega )}\to 0, \ \ {\rm and }\ \ \| \nabla u_{\eps}     - E_{\eps} \nabla u   \|_{L^2(\partial \Omega_{\eps})}\to 0,
\end{equation}
as $\eps\to 0$.
\end{proposition}

\begin{proof}Let $K_{\eps}\subset \Omega_{\eps}\cap \Omega  $ be as in the proof of  Lemma~\ref{l:weak-convergence}. Recall  that
$E_{\eps}u=u$ on $ K_{\eps}$.
In order to prove the first limit in  \eqref{intersec} it suffices to note  that by Lemma~\ref{l:unif-equiv} and the proof of Lemma~\ref{l:1} we have
\begin{eqnarray*}\label{intersec1}
\| u_{\eps}- u\|_{H^2(\Omega_{\eps}\cap \Omega )}&\le& \| u_{\eps}- E_{\eps}u\|_{H^2(\Omega_{\eps}\cap \Omega )}+\| E_{\eps}u- u\|_{H^2(\Omega_{\eps}\cap \Omega )}\nonumber \\
&\le & \| u_{\eps}- E_{\eps}u\|_{H^2(\Omega_{\eps} )}+\| E_{\eps}u- u\|_{H^2((\Omega_{\eps}\cap \Omega)\setminus K_{\eps} )}\nonumber\\
&\le & C \| u_{\eps}- E_{\eps}u\|_{V(\Omega_{\eps} )}+\| E_{\eps}u- u\|_{H^2((\Omega_{\eps}\cap \Omega)\setminus K_{\eps} )}\nonumber\\
&= & C \| u_{\eps}- E_{\eps}u\|_{V(\Omega_{\eps} )}+ o(1).
\end{eqnarray*}
In order to prove the second  limit in  \eqref{intersec} we note  that by Lemma~\ref{l:H1/2} the trace constant can be taken uniform with the respect to $\eps$ hence $\| \nabla u_{\eps}-\nabla( E_{\eps}u) \|_{L^2(\partial \Omega_{\eps })}   \le C   \| u_{\eps}- E_{\eps}u\|_{H^2(\Omega_{\eps} )}=o(1) $. Finally, by standard computations one can see that  $ \| \nabla u_{\eps}- E_{\eps}\nabla u \|_{L^2(\partial \Omega_{\eps })}    = \| \nabla u_{\eps}-\nabla( E_{\eps}u) \|_{L^2(\partial \Omega_{\eps })}
+o(1)$.
\end{proof}

\section{Optimality of condition \eqref{eq:assumptions}: degenerations and strange terms
}\label{s:optimality}

We plan to discuss the optimality of condition
\eqref{eq:assumptions}. We ask whether spectral stability is still
verified if we remove condition \eqref{eq:assumptions} from the statement of
Theorem \ref{t:main-1}. In order to give an answer to this
question we assume particular conditions on the domain $\Omega$
and on its perturbations $\Omega_\eps$. We follow closely the
approach introduced in \cite{ArLa2}. We assume that $\Omega$ is in
the form $\Omega=W\times (-1,0)$ where $W$ is a cuboid or a bounded domain in $\R^{N-1}$ of class $C^{1,1}$.
We assume that the perturbed domain
$\Omega_\eps$ is given by
\begin{equation*}
\Omega_\eps=\{(x',x_N):x'\in W\, ,
-1<x_N<g_\eps(x')\}
\end{equation*}
where $g_\eps(x')=\eps^\alpha b(x'/\eps)$ for any $x'\in W$ and
$b:\R^{N-1}\to [0,+\infty)$ is a $Y$-periodic function where $Y=\left(-\frac 12,\frac 12\right)^{N-1}$ is the unit cell in $\R^{N-1}$.

We denote by $\Gamma_\eps$ and $\Gamma$ the sets
\begin{equation} \label{eq:Gamma-eps}
\Gamma_\eps:=\{(x',g_\eps(x')):x'\in W\} \quad \text{and} \quad
\Gamma:=W\times \{0\} \, .
\end{equation}

In this case, condition \eqref{eq:assumptions} can be used provided $\alpha >3/2$. In fact, we have the following straightforward application of
Theorem~\ref{t:main-1} in the simplified form discussed in Remark \ref{r:simplified}.

\begin{corollary} \label{t:main-2ante} Suppose that the function $b$ defined above
belongs to $C^{1,1}(\R^{N-1})$. Let $\Omega$ and $\Omega_\eps$ be
as above. Let
$S_{\Delta,\eps}$, $S_{\Delta}$ be the Steklov operators for
\eqref{eq:Steklov} in $\Omega_\eps$ and
$\Omega$ respectively, and let
$S_{D^2,\eps}$, $S_{D^2}$ be the Steklov operators for
\eqref{eq:Steklov-modificato} in $\Omega_\eps$ and
$\Omega$ respectively,
 see Subsection \ref{s:functional-setting-B}. If $\alpha
>3/2$ then $S_{\Delta,\eps}\CC S_{\Delta}$ and $S_{D^2,\eps}\CC S_{D^2}$ as $\eps\to0$.
\end{corollary}

\begin{proof} For fixed $\tilde \alpha\in
] 3/2  ,\alpha  [$, we set $\kappa_{\eps }=\eps
^{2\tilde \alpha /3}$ and we note that condition
\eqref{eq:assumptions} is satisfied  with $g_{\eps,j}$, $g_{j}$
replaced by $g_{\eps}$, $g\equiv 0$ respectively. The proof
follows by Theorem~\ref{t:main-1} (actually, by the simplified
version of it concerning the case of domains defined as
subgraphs of functions defined in  one single cuboid, see Remark \ref{r:simplified}).
\end{proof}

We shall see that  Corollary~\ref{t:main-2ante} no longer holds for problem \eqref{eq:Steklov-modificato}  if
$\alpha \le 3/2$. The cases $\alpha <3/2$ and $\alpha =3/2$ are
considerably different. In the first case we have degeneration to
Dirichlet boundary conditions, in the second case we have the
appearance of a strange term in the limiting boundary conditions.
We discuss the two cases in the following subsections.

\subsection{Case $\alpha <3/2$: degeneration}\label{degsec}

In this case it appears that the energy spaces $V(\Omega_{\eps
})=H^2(\Omega_\eps)\cap H^1_0(\Omega_\eps)$ degenerate as $\eps
\to 0$ to a suitable energy space which encodes Dirichlet boundary
conditions on $\Gamma$.
 Namely,
let us define the space
\begin{equation} \label{eq:H^2-0-Gamma}
H^2_{0,\Gamma}(\Omega):=\{u\in H^2(\Omega)\cap H^1_0(\Omega):
u_\nu=0 \ \text{on} \ \Gamma\} \, .
\end{equation}
According to Section \ref{s:functional-setting-B}, we define
$T_{D^2,\Gamma}:H^2_{0,\Gamma}(\Omega)\to
(H^2_{0,\Gamma}(\Omega))'$ by
$$
_{(H^2_{0,\Gamma}(\Omega))'} \langle T_{D^2,\Gamma}u,v \rangle_{H^2_{0,\Gamma}(\Omega)} :=\int_\Omega D^2u:D^2v\, dx
\qquad \text{for any } u,v\in H^2_{0,\Gamma}(\Omega) \, .
$$
We also define the operator $J_\Gamma:H^2_{0,\Gamma}(\Omega)\to
(H^2_{0,\Gamma}(\Omega))'$ by
$$
_{(H^2_{0,\Gamma}(\Omega))'} \langle J_\Gamma u,v\rangle_{H^2_{0,\Gamma}(\Omega)} :=\int_{\partial\Omega\setminus \Gamma}
u_\nu v_\nu \, dS \qquad \text{for any } u,v\in
H^2_{0,\Gamma}(\Omega) \, .
$$
Finally we define $S_{D^2,\Gamma}:H^2_{0,\Gamma}(\Omega)\to
H^2_{0,\Gamma}(\Omega)$ as
$S_{D^2,\Gamma}:=T_{D^2,\Gamma}^{-1}\circ J_\Gamma$.

 We shall prove that $S_{D^2,\eps}$  converges compactly to $S_{D^2,\Gamma}$ as $\eps \to 0$. To do so we need to clarify which
family of operators $E_{\eps}:  H^2_{0,\Gamma}(\Omega)\to
V(\Omega_{\eps })$ we consider. In the case $\alpha <3/2$ we can
no longer use the  operators $E_{\eps }$ defined in
Subsection~\ref{subsecoperatorsE}. However, here it is possible to
consider a much simpler family of operators given by the
extension-by-zero operators.

 Let us define the extension-by-zero
operator $\mathcal{E}_0:H^2_{0,\Gamma}(\Omega)\to V(\Omega_\eps)$
which maps any function $u\in H^2_{0,\Gamma}(\Omega)$ to
$\mathcal{E}_0u$ where $\mathcal{E}_0 u:W\times (-1,+\infty)\to \R$
coincides with $u$ over $\Omega$ and vanishes in $W\times
[0,+\infty)$. We observe that $\mathcal E_0$ is well-defined as an
operator from $H^2_{0,\Gamma}(\Omega)$ to $V(\Omega_\eps)$ since
for any $u\in H^2_{0,\Gamma}(\Omega)$ we have $\mathcal E_0 u\in
H^2(W\times (-1,+\infty))$ being $u=u_\nu=0$ on $\Gamma$ and hence
its restriction to $\Omega_\eps$ belongs indeed to
$V(\Omega_\eps)$.

Since we are dealing with problem \eqref{eq:Steklov-modificato},
the space $V(\Omega_\eps)$ and the space $H^2_{0,\Gamma}(\Omega)$
will be endowed with the second scalar product in \eqref{eq:s-p},
i.e.
\begin{align*}
& \|u\|_{V(\Omega_\eps)}:=\left(\int_{\Omega_\eps} |D^2 u|^2
dx\right)^{1/2} \qquad \text{for any } u\in V(\Omega_\eps) \, , \\
& \|u\|_{H^2_{0,\Gamma}(\Omega)}:=\left(\int_{\Omega} |D^2 u|^2
dx\right)^{1/2} \qquad \text{for any } u\in H^2_{0,\Gamma}(\Omega)
\, .
\end{align*}

According to Definitions
\ref{d:E-convergence-ArCaLo}-\ref{d:CC-convergence-ArCaLo}, the
following notions of convergence are well defined
\begin{equation*}
u_\eps \EC u \, , \qquad B_\eps \EEC B_0 \, , \qquad B_\eps \CC
B_0
\end{equation*}
once we put $\mathcal H_\eps=V(\Omega_\eps)$, $\mathcal
H_0=H^2_{0,\Gamma}(\Omega)$ and $E_\eps=\mathcal E_0$. We observe
that property \eqref{eq:norm-convergence} is trivially satisfied.

Then the following result holds

\begin{theorem} \label{t:main-2} Suppose that the function $b$ defined above
belongs to $C^{1,1}(\R^{N-1})$ and is not constant.
Let $\Omega$ and $\Omega_\eps$ be as above. Let
$S_{D^2,\eps}$, $S_{D^2}$ be the Steklov operators for
\eqref{eq:Steklov-modificato} in $\Omega_\eps$ and
$\Omega$ respectively, see Section \ref{s:functional-setting-B}. Finally, let
$S_{D^2,\Gamma}:H^2_{0,\Gamma}(\Omega)\to H^2_{0,\Gamma}(\Omega)$
be as above.
 If $\alpha\in \left[1,\frac 32\right)$ then
$S_{D^2,\eps} \CC S_{D^2,\Gamma}$ as $\eps\to 0$ and in particular
the eigenvalues of \eqref{eq:Steklov-modificato} in $\Omega_\eps$
converge, as $\eps\to 0$, to the  eigenvalues of the following
problem
\begin{equation} \label{eq:P-alpha<3/2}
\begin{cases}
\Delta^2 u=0, & \qquad \text{\rm in } \Omega \, , \\
u=0, & \qquad \text{\rm on } \partial\Omega \, , \\
u_\nu=0, & \qquad \text{\rm on } \Gamma \, ,\\
\Delta u-K(x)u_\nu-\delta u_\nu=0, & \qquad \text{\rm on }
\partial\Omega\setminus \Gamma \, .
\end{cases}
\end{equation}
\end{theorem}

%

The proof of Theorem \ref{t:main-2} is based on a number of technical results. First of all,
  we prove that the $V(\Omega_\eps)$-norm is uniformly
equivalent to the $H^2(\Omega_\eps)$-norm.

\begin{lemma} \label{l:unif-equiv-2} Suppose that all the assumptions of Theorem
\ref{t:main-2} are satisfied. Then there exist $\eps_0>0$ and
$C>0$ such that
\begin{itemize}
\item[(i)] for any $u\in V(\Omega_\eps)$ and $\eps\in (0,\eps_0]$
we have
\begin{equation*} \|u\|_{H^2(\Omega_\eps)}\le C
\|u\|_{V(\Omega_\eps)}  \, ;
\end{equation*}

\item[(ii)] for any $v\in H^1(\Omega_\eps)$ and $\eps\in
(0,\eps_0]$ we have
\begin{equation*}
\|v\|_{L^2(\partial\Omega_\eps)} \le C\|v\|_{H^1(\Omega_\eps)} \,
.
\end{equation*}
\end{itemize}
\end{lemma}

\begin{proof} The proof of (i) is a consequence of Lemma~\ref{complete}, see also \eqref{eq:Faber-Krahn}.
The proof of (ii) follows from direct computation exploiting the
fact that under the assumptions of Theorem \ref{t:main-2}, the
first order derivatives of $g_\eps$ remain uniformly bounded as
$\eps\to 0$.
\end{proof}

As an immediate consequence of Lemma \ref{l:unif-equiv-2} we
obtain the following uniform estimate for the first eigenvalue of
\eqref{eq:Steklov-modificato}.

\begin{lemma} \label{l:delta-1} Suppose that all the assumptions of Theorem
\ref{t:main-2} are satisfied. Then
$$
\liminf_{\eps\to 0} \delta_1^\eps>0
$$
where $\delta_1^\eps$ is defined by \eqref{eq:char-delta-1} with $\Omega_\eps$ in place of $\Omega$.
\end{lemma}

\begin{proof} We denote by $C$ a positive
constant independent of $\eps$ and $u$ which may vary from line to
line. Let $u\in V(\Omega_\eps)$ and for $i\in \{1,\dots,N\}$ let
us apply Lemma \ref{l:unif-equiv-2} to $\frac{\partial
u}{\partial x_i}$:
\begin{equation*}
\left\|\frac{\partial u}{\partial
x_i}\right\|_{L^2(\partial\Omega_\eps)}\le C\left\|\frac{\partial
u}{\partial x_i}\right\|_{H^1(\Omega_\eps)}\le
C\|u\|_{H^2(\Omega_\eps)}\le C
\|u\|_{V(\Omega_\eps)} \qquad \text{for any } u\in V(\Omega_\eps)\ \text{and} \ \eps\in (0,\eps_0] \, .
\end{equation*}
Thus
$$
\int_{\partial\Omega_\eps} u_\nu^2 \, dS\le C \int_{\Omega_\eps}
|D^2 u|^2 dx \qquad \text{for any } u\in V(\Omega_\eps) \
\text{and} \ \eps\in (0,\eps_0] \, .
$$
The proof of the lemma then follows from \eqref{eq:char-delta-1}.
\end{proof}

Next we prove that the family of operators $\{S_{D^2,\eps}\}$
$E$-converges to $S_{D^2,\Gamma}$ as $\eps\to 0$.

\begin{lemma} \label{l:E-conv-bis} Suppose that all the assumptions of Theorem
\ref{t:main-2} are satisfied. Then $S_{D^2,\eps}\EEC
S_{D^2,\Gamma}$ as $\eps\to 0$.
\end{lemma}

\begin{proof} Let $w_\eps\in V(\Omega_\eps)$ and $w\in
H^2_{0,\Gamma}(\Omega)$ be such that $w_\eps \EC w$ as $\eps\to
0$. Put
$u_\eps:=S_{D^2,\eps}w_\eps$. 
We claim that $\|u_\eps\|_{V(\Omega_\eps)}$ remains bounded as
$\eps\to 0$. Using the same argument as in \eqref{duno} we have
$
 \|u_\eps\|_{V(\Omega_\eps)}^2 \le (\delta_1^\eps)^{-1} \|w_\eps\|_{V(\Omega_\eps)}
\|u_\eps\|_{V(\Omega_\eps)}
$
and this together with Lemma \ref{l:delta-1} proves the claim
being $\|w_\eps\|_{V(\Omega_\eps)}$ bounded as $\eps\to 0$ as a
consequence of the fact that $w_\eps \EC w$.

In particular by Lemma \ref{l:unif-equiv-2} (i) we also have that
$\{(u_\eps)_{|\Omega}\}_{0<\eps\le \eps_0}$ is bounded in $H^2(\Omega)$ as $\eps\to
0$ and hence along a sequence we have that there exists $u\in
H^2(\Omega)$, possibly depending on the choice of the sequence,
such that $(u_\eps)_{|\Omega}\rightharpoonup u$ in $H^2(\Omega)$
as $\eps\to 0$.

By taking the trivial extension of $u_\eps$ to the whole $\R^N$,
one sees that it is weakly convergent in $H^1(\R^N)$ and pointwise
convergent almost everywhere (up to a subsequence) to a function
belonging to $H^1(\R^N)$, whose restriction to $\Omega$ coincides
with $u$. This shows that $u\in H^1_0(\Omega)$.

We claim that $u_\nu=\frac{\partial u}{\partial x_N}=0$ on
$\Gamma$ thus showing that $u\in H^2_{0,\Gamma}(\Omega)$. In order
to prove the claim we apply \cite[Lemma 4.3]{CaLuSu} or more
precisely its extension to the $N$-dimensional case. Indeed, for
any $i\in \{1,\dots,N-1\}$, proceeding as in proof of Theorem 7.4
in \cite{ArLa2}, we define the family of vector fields depending
on the parameter $\eps$
$$
V_\eps^{(i)}=\left(0,\dots,0,-\frac{\partial u_\eps}{\partial
x_N},0, \dots,0,\frac{\partial u_\eps}{\partial x_i}\right)
$$
where the only two nontrivial components are the $i$-th and $N$-th
ones. Since $u_\eps(x',g_\eps(x'))=0$ for any $x'\in W$ it follows
that $V_\eps^{(i)}\cdot \nu=0$ on $\Gamma_\eps$ and hence, since $\alpha<\frac 32$, we can apply
\cite[Lemma 4.3 (i)]{CaLuSu} to this vector field and obtain
$$
-\frac{\partial u}{\partial x_N}(x',0)\frac{\partial b}{\partial
y_i}(y')=0 \qquad \text{for any } x'\in W \ \text{and} \ y'\in
\R^{N-1} \, .
$$
Since $b$ is not constant  we deduce that $\frac{\partial
u}{\partial x_N}\equiv 0$ in $\Gamma$ thus proving the claim.

We now prove that $u$ does not depend on the sequence converging
to zero. Letting $\varphi\in H^2_{0,\Gamma}(\Omega)$ we obtain
\begin{align} \label{eq:mod}
& \int_\Omega D^2u_\eps:D^2\varphi \, dx=\int_{\Omega_\eps}
D^2u_\eps:D^2(\mathcal E_0\varphi) \, dx=\int_{\partial\Omega_\eps}
(w_\eps)_\nu (\mathcal E_0\varphi)_\nu \,
dS=\int_{\partial\Omega\setminus\Gamma} (w_\eps)_\nu \varphi_\nu
\, dS \, .
\end{align}
From the weak convergence $(u_\eps)_{|\Omega}\rightharpoonup u$ in
$H^2(\Omega)$, boundedness of $\{(w_\eps)_{|\Omega}\}_{0<\eps \le \eps_0}$ in
$H^2(\Omega)$ as $\eps\to 0$, and compactness of the trace map, we
can pass to the limit in \eqref{eq:mod} as $\eps\to 0$ along a sequence and obtain
\begin{equation} \label{eq:var-Gamma}
\int_\Omega D^2u:D^2 \varphi\, dx=\int_{\partial\Omega\setminus
\Gamma} w_\nu \varphi_\nu \, dS \qquad \text{for any } \varphi\in
H^2_{0,\Gamma}(\Omega) \, .
\end{equation}
We observe that the first function appearing in the integral at the right hand side of \eqref{eq:var-Gamma} is exactly $w$ since, by Lemma \ref{l:unif-equiv-2} (i), $w_\eps\EC w$ implies $(w_\eps)_{|\Omega}\to w$ in $H^2(\Omega)$.

The variational identity \eqref{eq:var-Gamma} proves that $u\in H^2_{0,\Gamma}(\Omega)$ is a weak solution
of the problem
\begin{equation} \label{eq:modificato}
\begin{cases}
\Delta^2 u=0, & \qquad \text{in } \Omega\, , \\
u=0, & \qquad \text{on } \partial \Omega\, , \\
u_\nu=0, & \qquad \text{on } \Gamma \, , \\
\Delta u-K(x)u_\nu=w_\nu, & \qquad \text{on } \partial\Omega\setminus\Gamma \, .
\end{cases}
\end{equation}
Since it is easy to check that this problem admits a
unique solution, we deduce that $u$ depends only on $w$ and not on
the sequence converging to zero. Moreover, we also have that
$u=S_{D^2,\Gamma}w$.

Then we prove that $u_\eps \EC u$. Since $\mathcal E_0u$ vanishes
outside $\Omega$ and $(u_\eps)_{|\Omega}\rightharpoonup u$ in
$H^2(\Omega)$, we have
\begin{align} \label{eq:strong-conv}
& \|u_\eps-\mathcal E_0u\|_{V(\Omega_\eps)}^2=\int_{\Omega_\eps}
|D^2u_\eps|^2 dx-2\int_{\Omega} D^2 u_\eps: D^2 u \,
dx+\int_\Omega |D^2 u|^2 dx\\
\notag & \qquad =\int_{\Omega_\eps} |D^2u_\eps|^2 dx-\int_\Omega
|D^2 u|^2 dx+o(1) \, .
\end{align}
Now we observe that
\begin{equation} \label{eq:D^2 u}
\int_{\Omega_\eps} |D^2u_\eps|^2 dx=\int_{\partial\Omega_\eps}
(w_\eps)_\nu (u_\eps)_\nu \, dS \, .
\end{equation}
We prove that
\begin{equation} \label{eq:convbordo}
\int_{\partial\Omega_\eps} (w_\eps)_\nu (u_\eps)_\nu \, dS \to
\int_{\partial\Omega} w_\nu u_\nu \,
dS=\int_{\partial\Omega\setminus\Gamma} w_\nu u_\nu \, dS \, .
\end{equation}
Indeed we have
\begin{align*}
& \left|\int_{\partial\Omega_\eps} (w_\eps)_\nu (u_\eps)_\nu \,
dS-\int_{\partial\Omega} w_\nu u_\nu \, dS\right|\\
& \qquad \le \left|\int_{\partial\Omega_\eps} (w_\eps-\mathcal E_0
w)_\nu (u_\eps)_\nu \, dS\right|+\left|\int_{\partial\Omega_\eps}
(\mathcal E_0 w)_\nu (u_\eps)_\nu \, dS -\int_{\partial\Omega} w_\nu
u_\nu \, dS\right|\\
& \qquad \le(\delta_1^\eps)^{-1} \|w_\eps-\mathcal
E_0 w\|_{V(\Omega_\eps)}
\|u_\eps\|_{V(\Omega_\eps)}+\left|\int_{\partial\Omega} w_\nu
(u_\eps)_\nu \, dS -\int_{\partial\Omega} w_\nu u_\nu \,
dS\right|=o(1) \quad \text{as } \eps \to 0
\end{align*}
since $(\delta_1^\eps)^{-1}$ is bounded as $\eps\to 0$ in view of
Lemma \ref{l:delta-1}, $\|u_\eps\|_{V(\Omega_\eps)}$ is bounded as
$\eps\to 0$ as already observed, $\|w_\eps-\mathcal
E_0 u\|_{V(\Omega_\eps)}\to 0$ as $\eps\to 0$ since $w_\eps\EC w$ and
$(u_\eps)_\nu \to u_\nu$ in $L^2(\partial\Omega)$ since
$(u_\eps)_{|\Omega}\rightharpoonup u$ in $H^2(\Omega)$ as $\eps\to
0$. This proves the validity of \eqref{eq:convbordo}.

Finally we observe that, by \eqref{eq:var-Gamma}, we also have
\begin{equation} \label{eq:var-Gamma-2}
\int_\Omega |D^2u|^2 dx=\int_{\partial\Omega\setminus\Gamma} w_\nu
u_\nu \, dS \, .
\end{equation}
Combining \eqref{eq:D^2 u}, \eqref{eq:convbordo} and
\eqref{eq:var-Gamma-2} we infer that $\int_{\Omega_\eps} |D^2
u_\eps|^2 dx\to \int_\Omega |D^2 u|^2 dx$ as $\eps\to 0$ which
combined with \eqref{eq:strong-conv} proves that $u_\eps \EC u$.

We just proved that $S_{D^2,\eps}w_\eps \EC S_{D^2,\Gamma} w$ as
$\eps\to 0$ and this, according to Definition
\ref{d:EE-convergence-ArCaLo}, implies $S_{D^2,\eps}\EEC
S_{D^2,\Gamma}$ as $\eps\to 0$.
\end{proof}

\textit{End of the proof of Theorem \ref{t:main-2}. }
Having already proved Lemma~\ref{l:E-conv-bis}, it suffices to prove the validity of the compactness property stated in Definition~\ref{d:CC-convergence-ArCaLo}.
 Let $w_\eps \in
V(\Omega_\eps)$ and $\|w_\eps\|_{V(\Omega_\eps)}=1$ for any
$\eps>0$. As in Lemma \ref{l:E-conv-bis} we put
$u_\eps=S_{D^2,\eps} w_\eps$. By Lemma \ref{l:unif-equiv-2} (i) we
may assume that, along a sequence converging to zero, there exists
$w\in H^2(\Omega)$ such that $(w_\eps)_{|\Omega}\rightharpoonup w$
in $H^2(\Omega)$. Proceeding as in the proof of Lemma
\ref{l:E-conv-bis} we obtain
\begin{equation*}
\|u_\eps\|_{V(\Omega_\eps)}\le (\delta_1^\eps)^{-1}
\|w_\eps\|_{V(\Omega_\eps)}\le C \qquad \text{for any } \eps\in
(0,\eps_0]
\end{equation*}
for some constants $\eps_0$ and $C$ independent of $\eps$. By
Lemma \ref{l:unif-equiv-2} (i) we have that
$\|u_\eps\|_{H^2(\Omega_\eps)}\le C$ for any $\eps\in (0,\eps_0]$
and hence $\{(u_\eps)_{|\Omega}\}_{0<\eps\le \eps_0}$ is bounded
in $H^2(\Omega)$; in particular there exists $\widetilde u\in
H^2(\Omega)$ such that $(u_\eps)_{|\Omega}\rightharpoonup
\widetilde u$ in $H^2(\Omega)$ as $\eps\to 0$ along a sequence.
Proceeding as in the proof of Lemma \ref{l:E-conv-bis} one can
show that $\widetilde u\in H^2_{0,\Gamma}(\Omega)$ and it solves
\eqref{eq:modificato}. In particular, we have that $\widetilde
u=u:=S_{D^2,\Gamma}w$.

We claim that $u_\eps\EC u$ as $\eps\to 0$ along a sequence.
Proceeding exactly as in the proof of Lemma \ref{l:E-conv-bis} one
can verify that \eqref{eq:strong-conv}, \eqref{eq:D^2 u} and
\eqref{eq:var-Gamma-2} still hold true here. It remains to prove
\eqref{eq:convbordo} but here we cannot exploit the
$E$-convergence of $w_\eps$ as we did in the proof of Lemma
\ref{l:E-conv-bis}. Since $\alpha\ge 1$ we have that
$\partial\Omega_\eps$ is Lipschitz uniformly with respect to
$\eps$ converging to zero. If we look at the proof of
\cite[Theorem 5.5, Chapter 2]{Necas} with $k=1$ and $p=2$, we
deduce the following uniform estimate
\begin{equation} \label{eq:unif-v}
\|v\|_{H^{1/2}(\partial\Omega_\eps)}\le C\|v\|_{H^1(\Omega_\eps)}
\qquad \text{for any } v\in H^1(\Omega_\eps)
\end{equation}
where $C$ is a positive constant independent of $v$ and $\eps$
small. Introducing a $\eps$-dependent family of extension linear
operators from $H^1(\Omega_\eps)$ to $H^1(W\times (-1,+\infty))$
whose operatorial norms remain bounded as $\eps\to 0$ (we recall
that $\alpha\ge 1$ in the definition of $g_\eps$), we deduce that
if $\{v_\eps\}_{0<\eps\le \eps_0}$ is a family of functions whose
$H^1(\Omega_\eps)$-norms remain bounded as $\eps\to 0$ then
\begin{equation} \label{eq:residuo}
\int_{\partial\Omega_\eps\setminus (\Gamma_\eps\cup
\partial\Omega)} v_\eps^2 \, dS\to 0 \qquad \text{as } \eps\to 0
\end{equation}
since the surface measure of $\partial\Omega_\eps\setminus
(\Gamma_\eps\cup \partial\Omega)$ converges to zero as $\eps\to
0$.

Applying \eqref{eq:residuo} to $\frac{\partial w_\eps}{\partial
x_i}$ and $\frac{\partial u_\eps}{\partial x_i}$ and recalling
that $\|u_\eps\|_{H^2(\Omega_\eps)}, \|w_\eps\|_{H^2(\Omega_\eps)}$
remain bounded as $\eps\to 0$, we infer that
\begin{equation} \label{eq:residuo-2}
\int_{\partial\Omega_\eps\setminus (\Gamma_\eps\cup
\partial\Omega)} (w_\eps)_\nu (u_\eps)_\nu \, dS\to 0 \qquad
\text{as } \eps\to 0 \, .
\end{equation}
On the other hand, by \eqref{eq:unif-v} applied to $\frac{\partial w_\eps}{\partial
x_i}$ and $\frac{\partial u_\eps}{\partial x_i}$ combined with the compact embedding \\
 $H^{1/2}(\partial\Omega\setminus \Gamma)\subset L^2(\partial\Omega\setminus \Gamma)$, we have that
\begin{equation} \label{eq:partial-Omega}
\int_{\partial\Omega\setminus \Gamma} (w_\eps)_\nu (u_\eps)_\nu \, dS
\to \int_{\partial\Omega\setminus \Gamma} w_\nu u_\nu \, dS
\end{equation}
along a sequence converging to zero.

It remains to consider $\int_{\Gamma_\eps} (w_\eps)_\nu
(u_\eps)_\nu \, dS$. By Lemma \ref{l:delta-1}, the fact that
$\alpha\ge 1$ and $b\in C^{1,1}(\R^{N-1})$, we obtain for $\eps$ small
enough
\begin{align} \label{eq:st-Gamma}
& \left|\int_{\Gamma_\eps} (w_\eps)_\nu (u_\eps) \, dS\right|\le \left(\int_{\Gamma_\eps} (w_\eps)_\nu^2 \, dS\right)^{1/2}
 \left(\int_{\Gamma_\eps} (u_\eps)_\nu^2 \, dS\right)^{1/2} \\
 \notag &  \le (\delta_1^\eps)^{-1/2} \|w_\eps\|_{V(\Omega_\eps)}  \left(\int_{\Gamma_\eps} (u_\eps)_\nu^2 \, dS\right)^{1/2}
 \le C\left(\sum_{i=1}^N \int_{\Gamma_\eps} \left|\frac{\partial u_\eps}{\partial x_i}\right|^2 dS\right)^{1/2}
 \\
\notag  & = C\left(\sum_{i=1}^N \int_{W} \left|\tfrac{\partial
u_\eps}{\partial x_i}(x',g_\eps(x'))\right|^2 \!\!\!
\sqrt{1+|\nabla_{x'}g_\eps(x')|^2} \, dx' \right)^{\!\!\! 1/2}
\!\!\! \le C\left(\sum_{i=1}^N \int_{W} \left|\tfrac{\partial
u_\eps}{\partial x_i}(x',g_\eps(x'))\right|^2  dx' \right)^{\!\!\!
1/2}
\end{align}
where $C$ denotes a constant independent of $\eps$ which may vary
in every estimate.

We claim that $\int_{W} \left|\tfrac{\partial u_\eps}{\partial
x_i}(x',g_\eps(x'))\right|^2  dx'\to 0$ as $\eps\to 0$. Indeed, we
have
\begin{align} \label{eq:st-Gamma-2}
& \int_{W} \left|\tfrac{\partial u_\eps}{\partial
x_i}(x',g_\eps(x'))\right|^2  dx' \le 2 \int_{W}
\left|\tfrac{\partial u_\eps}{\partial
x_i}(x',g_\eps(x'))-\tfrac{\partial u_\eps}{\partial
x_i}(x',0)\right|^2  dx' +2 \int_{W} \left|\tfrac{\partial
u_\eps}{\partial x_i}(x',0)\right|^2  dx'  \, .
\end{align}
The second term at the right hand side of \eqref{eq:st-Gamma-2} converges to zero as one can verify combining the fact that $(u_\eps)_{|\Omega}\rightharpoonup u$ in $H^2(\Omega)$ and $u\in H^2_{0,\Gamma}(\Omega)$, with compactness of the trace map.
With the first term at the right hand side of \eqref{eq:st-Gamma-2} we proceed in this way
\begin{align} \label{eq:st-Gamma-3}
& \int_{W} \left|\tfrac{\partial u_\eps}{\partial x_i}(x',g_\eps(x'))-\tfrac{\partial u_\eps}{\partial x_i}(x',0)\right|^2  dx'
=\int_W \left|\int_0^{g_\eps(x')} \tfrac{\partial^2 u_\eps}{\partial x_N \partial x_i}(x',x_N)\, dx_N\right|^2 dx'
\\
\notag & \le \int_W |g_\eps(x')| \left|\int_0^{g_\eps(x')}
\left|\tfrac{\partial^2 u_\eps}{\partial x_N \partial
x_i}(x',x_N)\right|^2 dx_N\right|dx'
\le \eps^\alpha \|b\|_{L^\infty(\R^{N-1})} \int_{\Omega_\eps \setminus \Omega} \left|\tfrac{\partial^2 u_\eps}{\partial x_N \partial x_i}\right|^2 dx\\
\notag & \le  \eps^\alpha \|b\|_{L^\infty(\R^{N-1})} \|u_\eps\|_{H^2(\Omega_\eps)}^2\to 0 \qquad \text{as } \eps\to 0 \, .
\end{align}
The proof of the claim follows by combining \eqref{eq:st-Gamma-2} and \eqref{eq:st-Gamma-3}. In turn, the claim combined with \eqref{eq:st-Gamma} implies
\begin{equation} \label{eq:tb}
\int_{\Gamma_\eps} (w_\eps)_\nu (u_\eps)_\nu \, dS\to 0 \qquad \text{as } \eps \to 0 \, .
\end{equation}
Now, combining \eqref{eq:residuo-2}, \eqref{eq:partial-Omega} and \eqref{eq:tb}, we obtain
\begin{equation} \label{eq:conv-bordo-finale}
\int_{\partial\Omega_\eps} (w_\eps)_\nu (u_\eps)_\nu \, dS \to \int_{\partial\Omega\setminus \Gamma} w_\nu u_\nu \, dS
\end{equation}
as $\eps\to 0$ along a sequence.
Since $u_\eps$ satisfies \eqref{eq:D^2 u} and $u\in H^2_{0,\Gamma}(\Omega)$ is a weak solution of \eqref{eq:modificato}, from \eqref{eq:conv-bordo-finale}
we obtain
\begin{align} \label{eq:conv-norm-finale}
\int_{\Omega_\eps} |D^2 u_\eps|^2 dx=\int_{\partial\Omega_\eps} (w_\eps)_\nu (u_\eps)_\nu \, dS \to \int_{\partial\Omega\setminus \Gamma} w_\nu u_\nu \, dS=
\int_\Omega |D^2 u|^2 dx
\end{align}
as $\eps\to 0$ along a sequence.

Since \eqref{eq:strong-conv} still holds true for $u_\eps$, by
\eqref{eq:conv-norm-finale}, it follows that $u_\eps \EC u$ as
$\eps\to 0$ along a sequence. We have proved that
$\{S_{D^2,\eps}w_\eps\}$ is precompact in the sense of Definition
\ref{d:precompact-ArCaLo}. The proof of the theorem is now a
consequence of Lemma \ref{l:E-conv-bis} and Definition
\ref{d:CC-convergence-ArCaLo}.

\subsection{Case $\alpha =3/2$: strange term}\label{strangetermsteklov}




Assume that $\Omega$, $\Omega_{\eps }$, $\Gamma$ and
$\Gamma_{\eps}$ are as above.  We set   $\Sigma_{\eps}=
\partial\Omega_{\eps }\setminus \Gamma_{\eps}$ and  $\Sigma =\partial \Omega \setminus \Gamma$. In this
subsection we discuss the following modified Steklov problem:
\begin{equation} \label{eq:Steklov-modificatoDIR}
\begin{cases}
\Delta^2 u=0, & \qquad \text{in } \Omega_{\eps} \, , \\
u=0, & \qquad \text{on } \partial\Omega_{\eps } \, , \\
u_\nu=0, & \qquad \text{on } \Sigma_{\eps} \, , \\
\Delta u-K(x)u_\nu=\delta u_\nu, & \qquad \text{on } \Gamma_{\eps
} \, .
\end{cases}
\end{equation}

The third condition $u_\nu=0$ on $\Sigma_{\eps}$ is used here to
simplify  our arguments as well as to avoid a few annoying
technicalities which seem not particularly interesting for the
specific case under discussion.

The eigenvalues of problem   \eqref{eq:Steklov-modificatoDIR} are
the reciprocals of the eigenvalues of a Navier-to-Neumann     map   $N_{\Omega_{\eps}}$   that
we now define. Namely, given  $f \in L^2(\Gamma_{\eps })$ consider
the solution $u_f=u$ to the following modified Navier problem
\begin{equation} \label{eq:Steklov-modificato-pier}
\begin{cases}
\Delta^2 u=0, & \qquad \text{in } \Omega_{\eps} \, , \\
u=0, & \qquad \text{on } \partial\Omega_{\eps } \, , \\
u_\nu=0,  & \qquad \text{on }  \Sigma_{\eps} \, , \\
\Delta u-K(x)u_\nu=f, & \qquad \text{on } \Gamma_{\eps } \, ,
\end{cases}
\end{equation}
which means that $u\in H^2_{0,\Sigma_{\eps}}(\Omega_{\eps } )\cap
H^1_0(\Omega_{\eps } )$ is the solution to the weak problem
\begin{equation}
\label{crit} \int_{\Omega_{\eps  }}D^2u:D^2\varphi \, dx
=\int_{\Gamma_{\eps } } f\frac{\partial \varphi }{\partial \nu }
\, dS \, ,\ \ \ \forall \  \varphi\in
H^2_{0,\Sigma_{\eps}}(\Omega_{\eps } )\cap H^1_0(\Omega_{\eps } ).
\end{equation}
Here $H^2_{0,\Sigma_{\eps}}(\Omega_{\eps })$ is defined as in \eqref{eq:H^2-0-Gamma}  with $\Omega $ replaced by $\Omega_{\eps}$ and  $\Gamma $ replaced by $\Sigma_{\eps} $. Then we consider the function
$N_{\Omega_{\eps }} $ from $L^2(\Gamma_{\eps } )$ to itself
defined by $N_{\Omega_{\eps } }f=\frac{\partial u_{f}}{\partial
\nu }$ for all $f\in  L^2(\Gamma_{\eps } )$ where $u_f$ is the solution
to \eqref{crit}. Here we analyse the behaviour of the map
$N_{\Omega_{\eps }}$ as $\eps \to 0$ in the case $\alpha =3/2$.

The limiting problem involves a strange factor $\gamma $,
which is called `strange curvature' in  \cite{ArLa2}. We recall that
$\gamma $ is defined by
 \begin{equation}\label{strangecurvature}
 \gamma=\int_{Y\times (-\infty , 0)  }|D^2
V|^2dy,
\end{equation}
 and the function $V$ is $Y$-periodic in the variables $y'$ and  satisfies the following microscopic problem
 \begin{equation}
\left\{
 \begin{array}{ll}
 \Delta^2V=0,\  & \hbox{\rm in } Y\times (-\infty, 0), \\
 V(y', 0)=b(y'),\ & \hbox{\rm on } Y, \\
 \frac{\partial^2 V}{\partial y_N^2}(y',0)=0,\  &\hbox{\rm on } Y.
  \end{array}
\right.
 \end{equation}

 It turns out that the limiting
functional for the maps $N_{\Omega_{\eps}}$ as $\eps\to 0$ is the map $S_{\Omega }$ defined as follows. Given $f
\in L^2(\Gamma )$, consider the solution $\tilde u_f=u$ to the problem
\begin{equation} \label{eq:Steklov-modificatoOmega}
\begin{cases}
\Delta^2 u=0, & \qquad \text{in } \Omega  \, , \\
u=0, & \qquad \text{on } \partial\Omega  \, , \\
u_\nu =0,  & \qquad \text{on }  \Sigma \, , \\
\Delta u-K(x)u_\nu+\gamma u_\nu=f, & \qquad \text{on } \Gamma\, .
\end{cases}
\end{equation}

 Then  $S_{\Omega }$ is the map from  $L^2(\Gamma  )$ to itself defined by
$S_{\Omega }f=\frac{\partial \tilde u_f}{\partial \nu}$ for all $f
\in L^2(\Gamma)$. Since the trace operator is compact we have that
both operators $N_{\Omega_{\eps }} $ and $S_{\Omega } $ are
compact.

We now consider a family of operators  $E_\eps$ from  $L^2(\Gamma )$ to
$L^2(\Gamma_{\eps}  )$ defined by  $E_\eps(f)(x',g_{\eps}(x'))=f(x',0)$
for all $x'\in W$ and all  $f\in L^2(\Gamma)$. It is obvious
that if $\alpha =3/2$ then $g_{\eps }$ converges to zero uniformly
together with its first derivatives hence
\begin{equation}
\| E_\eps f\|_{L^2(\Gamma_{\eps })} \to \|f\|_{L^2(\Gamma)}
\end{equation}
as $\eps \to 0$, which means that condition
\eqref{eq:norm-convergence} is satisfied. With reference to these
operators $E_\eps$, the following theorem holds.

 \begin{theorem}\label{thmcritstek}
 If $\alpha =3/2$ then  $N_{\Omega_{\eps }}\CC S_{\Omega }$ as $\eps \to 0$, hence the spectrum of problem \eqref{eq:Steklov-modificatoDIR}
 converges to the spectrum of \eqref{eq:Steklov-modificatoOmega} in the sense of Theorem~\ref{vaithm}.
 \end{theorem}

\begin{proof} Let $f_{\eps }\in L^2(\Gamma_{\eps})$ be $E$-convergent to $f\in L^2(\Gamma )$ as $\eps \to 0$ in the sense of Definition \ref{d:E-convergence-ArCaLo}. Let $u_{\eps }$
be the  solution to problem \eqref{crit} with $f$ replaced by
$f_{\eps}$. Using $u_{\eps}$
 as a test function in \eqref{crit}, we get that $\|u_{\eps}\|_{H^2(\Omega_{\eps})}$ is uniformly bounded, hence there exists
 $u\in H^2_{0, \Sigma}(\Omega)\cap H^1_0(\Omega)$ such that  possibly passing to a subsequence, $u_{\eps }\rightharpoonup u$ weakly in $H^2(\Omega )$
 and strongly  in $H^1(\Omega)$.  We now prove that  $u$ satisfies  problem \eqref{eq:Steklov-modificatoOmega}.

To do so, following~\cite[\S~8.1]{ArLa2} we define a diffeomorphism $\Phi_{\eps}$ from
$\overline\Omega_{\eps }$ onto $\overline\Omega$.
 Namely,  $ \Phi_{\eps }(x',x_N)=(x',x_N- h_{\eps }(x',x_N))$ for all $(x',x_N)\in
\overline\Omega_{\eps }$ where
\begin{equation}
\label{accatris}
 h_{\eps }(x', x_N)=\left\{\begin{array}{ll}0,& \ \ {\rm if }\ -1\le x_N\le -\eps \vspace{2mm}  \\
g_{\eps }(x')\left( \frac{x_N+\eps }{g_{\eps }(x')+\eps }
\right)^{3},& \ \ {\rm if }\  -\eps<x_N\le g_{\eps }(x')\, .
\end{array} \right.
\end{equation}
It is convenient to denote  by  $ K_{\eps}$ the set where
$\Phi_{\eps}$ coincides with the identity, that is $K_{\eps} =\{
(x', x_N)\in W\times \R:\ -1<x_N\le-\eps \}$.

Let $\varphi \in H^2_{0, \Sigma}(\Omega)\cap H^1_0(\Omega )$ be
fixed.  We set $\varphi _{\eps }=\varphi \circ \Phi_{\eps }$ and
we note that $\varphi_{\eps }\in
H^2_{0,\Sigma_{\eps}}(\Omega_{\eps } )\cap H^1_0(\Omega_{\eps }
)$. Using $\varphi_{\eps }$ as a test function in \eqref{crit} we
get

\begin{equation}
\label{crit1} \int_{K_{\eps} }D^2 u_{\eps} : D^2  \varphi \,
dx+\int_{\Omega_{\eps}\setminus K_{\eps} }D^2 u_{\eps} : D^2
\varphi_{\eps} \, dx= \int_{\Omega_{\eps }}D^2u_{\eps
}:D^2\varphi_{\eps } \, dx =\int_{\Gamma_{\eps } } f_{\eps
}\frac{\partial \varphi_{\eps } }{\partial \nu } \, dS \, .
\end{equation}
It is proved in \cite[Theorem~8.72]{ArLa2} that
\begin{equation} \label{trico2bis}
\lim_{\eps \to 0}\int_{\Omega_{\eps}\setminus K_{\eps}} D^2
u_{\eps} : D^2  \varphi_{\eps} \, dx= \gamma \int_W\frac{\partial
u}{\partial x_N}(x', 0) \frac{\partial \varphi}{\partial
x_N}(x',0)\, dx' \, ,
\end{equation}
where $\gamma$ is as in the statement. Note that a careful inspection of the arguments in \cite[\S~8]{ArLa2} shows that the  definition of
$\gamma$ is not affected by our Dirichlet boundary conditions on
$\Sigma_{\eps}$. Thus
\begin{equation}
\label{crit2bis} \lim_{\eps \to 0} \int_{\Omega_{\eps }}
D^2u_{\eps }:D^2\varphi_{\eps} \, dx= \int_{\Omega}
D^2u:D^2\varphi \, dx+\gamma  \int_W\frac{\partial u}{\partial
x_N}(x',0)\frac{\partial \varphi}{\partial x_N}(x',0)\, dx' \, .
\end{equation}
We now consider the integral in the right-hand
side of \eqref{crit1}. We have

\begin{align}\label{crit3}
& \left|\int_{\Gamma_{\eps } } f_{\eps }\frac{\partial
\varphi_{\eps } }{\partial \nu } \, dS -
\int_{\Gamma}f\frac{\partial \varphi }{\partial \nu } \, dS\right|
\le \int_{\Gamma_{\eps } }   \left| f_{\eps } -E_\eps f\right| \left|
\frac{\partial \varphi_{\eps } }{\partial \nu }\right| dS+\left|
\int_{\Gamma_{\eps } } E_\eps f\frac{\partial \varphi_{\eps } }{\partial
\nu } \, dS - \int_{\Gamma} f\frac{\partial \varphi
}{\partial \nu } \, dS \right| \\
\notag & \le C\|  f_{\eps } -E_\eps f\|_{L^2(\Gamma_{\eps })}+
\left|\int_{\Gamma_{\eps }}E_\eps f \frac{\partial \varphi_{\eps }
}{\partial \nu } \, dS - \int_{\Gamma} f \frac{\partial \varphi
}{\partial \nu } \, dS\right|\\
\notag & =o(1)+\left|\int_{W} f(x',0) \frac{\partial
\varphi_{\eps} }{\partial \nu} (x',g_{\eps }(x'))
\sqrt{1+|\nabla_{x'} g_{\eps}( x')|^2 }\, dx'-\int_{W} f(x',0)
\frac{\partial \varphi }{\partial \nu
 }(x',0) \, dx'\right|=o(1),
\end{align}
where $C>0$ is independent of $\eps$.
It follows by \eqref{crit2bis}-\eqref{crit3} that function $u$
satisfies the problem
\begin{equation}
\label{crit4} \int_{\Omega } D^2u:D^2\varphi \, dx+ \gamma
\int_W\frac{\partial u}{\partial x_N}(x',0) \frac{\partial
\varphi}{\partial x_N}(x',0)\, dx'= \int_{\Gamma} f\frac{\partial
\varphi }{\partial \nu } dS \, .
\end{equation}

We have now to prove that $\frac{\partial u_{\eps
}}{\partial\nu_{\eps }}$ is $E$-convergent to $\frac{\partial u
}{\partial\nu}$. We note that

\begin{eqnarray}\label{crit5}\lefteqn{
\int_{\partial \Omega_{\eps } } \left|      \frac{\partial u_{\eps
} }{\partial \nu } - E_\eps \frac{\partial u}{\partial \nu }
\right|^2 dS = \int_{W } \left|\frac{\partial u_{\eps }
}{\partial \nu } (x',g_{\eps}(x')) -
\frac{\partial u }{\partial \nu }(x',0) \right|^2  \sqrt{1+|\nabla g_{\eps }(x')|^2 }dx' }\nonumber \\
& & \qquad\qquad\qquad\qquad\qquad\qquad\qquad\qquad\qquad =
\int_{W } \left|      \frac{\partial u_{\eps } }{\partial \nu }
(x',g_{\eps}(x'))-\frac{\partial u }{\partial \nu }(x',0)
\right|^2  dx' +o(1)\, .
\end{eqnarray}

By the Fundamental Theorem of Calculus we have that

\begin{equation}
\label{crit6} \|\nabla u_{\eps} (x',g_{\eps}(x'))-\nabla
u_{\eps}(x',0)\|_{L^2(W)}\le c \eps^{\alpha/2 }\| u_{\eps
}\|_{H^2(\Omega_{\eps } \setminus \Omega )}=O(\eps^{\alpha /2})
\end{equation}
as $\eps \to 0$. On the other hand, by  the compactness of the
trace operator we have that possibly passing to a subsequence
\begin{equation}
\label{crit7} \| \nabla u_{\eps}(x',0)-\nabla u (x',0)
\|_{L^2(W)}\to 0
\end{equation}
 $\eps \to 0$.
By combining \eqref{crit5}-\eqref{crit7} we conclude that
$\int_{\partial \Omega_{\eps } } \left|\frac{\partial u_{\eps }
}{\partial \nu } - E\frac{\partial u}{\partial \nu } \right|^2 dS
\to 0 $ as $\eps \to 0$ hence $\frac{\partial u_{\eps
}}{\partial\nu}$ is $E$-convergent to $\frac{\partial u
}{\partial\nu}$.

In order to conclude the proof, it suffices to observe that if
$f_{\eps}$ is a sequence with $\| f_{\eps }\|_{L^2(\Gamma_{\eps
})}=1$ for all $\eps >0$, then the corresponding solutions
$u_{\eps }$ have uniformly bounded $H^2$ norms, hence exploiting
the compactness of the trace operator as above one can easily
prove that $\frac{\partial u_{\eps }}{\partial\nu  }$ has an
$E$-convergent subsequence.
\end{proof}

\subsection{Another way of viewing the trichotomy}

In this subsection we consider the eigenvalue problem \eqref{eq:Steklov-modificatoDIR} and we re-interpret the results of this section adapting the analysis
of Subsection~\ref{degsec} to this specific case. This allows to state  a trichotomy result for the eigenvalues which we believe is quite transparent. Namely, we have the following theorem.

\begin{theorem} Let $\Omega_{\eps}$, $\eps \geq 0$ be as above, with $\Omega_0=\Omega$. Let $\lambda_n(\eps )$, $n\in \N$,  be the eigenvalues of problems  \eqref{eq:Steklov-modificatoDIR} for  $\eps \geq 0$. Then the following statements hold:
\begin{itemize}
\item[(i)] If $\alpha >3/2$ then  $\lambda_n(\eps )\to \lambda_n(0
)$, as $\eps \to 0$. \item[(ii)] If $\alpha =3/2$ then
$\lambda_n(\eps )\to  \lambda_n(0 )+\gamma $ as $\eps \to 0$,
where $\gamma$ is as in \eqref{strangecurvature}. \item[(iii)] If
$1\le \alpha < 3/2$ then $\lambda_n(\eps)\to \infty $, as $\eps
\to 0$.
\end{itemize}
\end{theorem}

\begin{proof} The proof of (i) is  exactly the same proof of Corollary~\ref{t:main-2ante} with obvious minor modifications. Statement (ii) is a straightforward application of
Theorem~\ref{thmcritstek}. We now prove statement (iii). We
proceed as in Subsection~\ref{degsec} replacing the energy spaces
$H^2(\Omega_{\eps })\cap H^1_0(\Omega_{\eps })$ and $H^2_{0,
\Gamma}(\Omega )$ by  $H^2_{0,\Sigma_{\eps}}(\Omega_{\eps } )\cap
H^1_0(\Omega_{\eps } )$ and $H^2_0(\Omega )$ respectively. Accordingly, the
operators $S_{D^2, \eps}$ and $S_{D^2,\Gamma}$ used in
Subsection~\ref{degsec} have to be re-defined  in an obvious way
(formally, it is the same way) taking into account the new energy
spaces. Let's call $\tilde S_{D^2, \eps}$ and $\tilde
S_{D^2,\Gamma}$ the new operators replacing $S_{D^2, \eps}$ and
$S_{D^2,\Gamma}$, respectively. Having a closer look at the
definition of $\tilde S_{D^2,\Gamma}$, one realises that actually
$\tilde S_{D^2,\Gamma}$ is the operator identically equal to zero:
indeed, the map $J_{\Gamma }$ used in the definition of
$S_{D^2,\Gamma}$ vanishes on $H^2_0(\Omega )$. Thus, by the same
argument used in the proof of Theorem~\ref{t:main-2}  we have that
$\tilde S_{D^2,\eps}\CC 0$ as $\eps \to 0$. Assume now by
contradiction that there exists a sequence $\eps_k\to 0$ such that
for some $n\in \N$ we have $\sup_{k\in \N}\lambda _n(\eps
_k)<\infty $. Using the the Poincar\'{e} inequality, one can see
that $\inf_{k\in \N}\lambda _n(\eps _k)>0$, see also
Lemma~\ref{l:delta-1}. Now, we set $\mu_{n}(\eps_k )=1/\lambda
_n(\eps _k)$ and we observe that the sequence $\mu_{n}(\eps_k )$,
$k\in \N$ is bounded away from zero and infinity. Thus, possibly
passing to a subsequence, there exists $\mu >0$ such that
$\mu_{n}(\eps_k )\to \mu$ as $k\to \infty$. Since $\mu_{n}(\eps_k
)$ is an eigenvalue of $\tilde S_{D^2,\eps_{k}}$ and $\tilde
S_{D^2,\eps_k}\CC 0$, by Vainikko~\cite[Theorem~6.1]{Vainikko} it
follows that $\mu$ is an eigenvalue of the zero operator, a
contradiction.
\end{proof}

\section{Navier-type problems and Navier-to-Neumann maps}
\label{navtoneusection}

The main aim of this section is to study the stability of  the inverse  of a Dirichlet-to-Neumann-type map associated with problem \eqref{eq:Steklov}
under weak assumptions on the convergence of the varying domains.  As we have mentioned in the introduction, this map
is called here Navier-to-Neumann map. To do so, we begin with discussing  the stability of the Navier problem under domain perturbations of general type.

\subsection{Navier problems}
\label{naviersub}

Given a sufficiently smooth bounded domain $\Omega $  in
${\mathbb{R}}^N$,  the classical Navier problem for the biharmonic
operator reads
\begin{equation}
\label{classicalnav} \left\{
\begin{array}{ll}
\Delta^2u=0,& \ {\rm in}\ \Omega, \\
u=0,& \  {\rm on }\ \partial \Omega, \\
\Delta u=f, & \  {\rm on }\ \partial \Omega .
\end{array}
\right.
\end{equation}
In a classical setting, it would be natural to require  at least
that $\Omega$  is of class $C^{0,1}$ and that $f\in H^{1/2}(
\partial \Omega )$.  Actually, since the trace operator is
surjective, one could directly assume that $f\in H^1(\Omega )$.

As it is well-known, problem \eqref{classicalnav} could be recast
in the form of a system by setting $v:=\Delta u$ and solving
\begin{equation}\label{poidir}
\left\{
\begin{array}{ll}
\Delta v=0,&\ {\rm in }\ \Omega\\
v=f,&\ {\rm on}\ \partial \Omega\\
\end{array}
\right. \ \ {\rm and }\ \ \left\{
\begin{array}{ll}
\Delta u=v,&\ {\rm in }\ \Omega\\
u=0,&\ {\rm on}\ \partial \Omega.\\
\end{array}
\right.
\end{equation}

Accordingly,  in order to prove that problem \eqref{classicalnav}
is stable under domain perturbations, one may think of exploiting
the properties of  both Poisson and Dirichlet problems in
\eqref{poidir} which are stable under sufficiently regular
boundary perturbations. However, since we aim at considering a
rather general class of domain perturbations as well as at
comparing problem \eqref{classicalnav}  with problem
\eqref{classicalhes}, we prefer to adopt another point of view
which has  its own interest, and which is based on the energy
quadratic form naturally associated with \eqref{classicalnav}.
This point of view allows to relax the boundary regularity
assumptions on $\Omega$ even more.

Namely, we interpret  problem  \eqref{classicalnav} in a weak
sense by formulating it  as follows
\begin{equation}\label{weaknav}
\int_{\Omega }\Delta u \Delta \varphi \, dx = \int_{\Omega }
(f\Delta \varphi +\nabla f\nabla \varphi) \, dx,\ \ {\rm for\
all}\ \varphi \in   H(\Delta, \Omega )\cap H^1_0(\Omega ),
\end{equation}
where the unknown $u$ has to be considered in $ H(\Delta , \Omega
)\cap H^1_0(\Omega )$. Note that our weak formulation is motivated
by the fact that if $\Omega $ is  regular enough, for example as
in Lemma~\ref{complete} (iii),  then $ H(\Delta , \Omega )\cap
H^1_0(\Omega )=H^2(\Omega )\cap H^1_0(\Omega )$ hence  equality
\eqref{weaknav} could be re-written as
\begin{equation}\label{weaknavclassical}
\int_{\Omega }\Delta u \Delta \varphi \, dx = \int_{\partial\Omega
}f \varphi_\nu \, dS \, ,
\end{equation}
which is a more familiar   way of writing the variational
formulation for \eqref{classicalnav}. In fact  the right-hand side
of \eqref{weaknav} is the standard way of defining the normal
derivative for a function in  $ H(\Delta , \Omega )$ as an element
of $H^{-1/2}(\partial \Omega )$, see \eqref{conormal} below, see
also e.g., \cite{Helf}.

The following lemma shows that problem \eqref{weaknav} is
well-posed even without any boundary regularity assumption on
$\Omega$.

\begin{lemma}
Let $\Omega $ be a domain  in $\R^N$ such that the Poincar\'{e}
inequality holds and let $f\in H^1(\Omega)$. Then problem
\eqref{weaknav} is well-posed in the sense of Hadamard, that is
there exists a unique solution $u$ in $H(\Delta, \Omega )\cap
H^1_0(\Omega )$ which depends continuously on $f$.
\end{lemma}

\begin{proof}We simply observe that by Lemma~\ref{complete}, the map $\varphi \mapsto \int_{\Omega } (f\Delta \varphi +\nabla f\nabla \varphi)\, dx$
is an element of the dual of $ H(\Delta , \Omega )\cap H^1_0(\Omega )$ which depends continuously on $f\in H^1(\Omega)$. Then the proof easily follows by the Riesz Theorem applied to the Hilbert space
 $ H(\Delta , \Omega )\cap H^1_0(\Omega )$.
\end{proof}

We now consider a family of bounded domains $\Omega_{\eps }$,
$\eps >0$   and the corresponding boundary value problems
\begin{equation}
\label{classicalnaveps} \left\{
\begin{array}{ll}
\Delta^2u_{\eps }=0,& \ {\rm in}\ \Omega_{\eps }, \\
u_{\eps }=0,& \  {\rm on }\ \partial \Omega_{\eps }, \\
\Delta u_{\eps }=f_{\eps }, & \  {\rm on }\ \partial \Omega_{\eps
} ,
\end{array}
\right.
\end{equation}
in the unknown $u_{\eps }\in  H(\Delta , \Omega_{\eps} )\cap
H^1_0(\Omega_{\eps} )$, where $f_{\eps }\in H^1(\Omega_{\eps })$,
with the understanding that
 problem \eqref{classicalnaveps} has to be interpreted in the weak sense as above.

We shall assume that the domains $\Omega_{\eps}$ compactly
converge to $\Omega$ and that $\Omega $ is regular (stable) in the
sense of the following definition which goes back to Keldysh (cf.
\cite[\S~4.9]{BucBut}).

\begin{definition}
We say that the domains $\Omega_{\eps}$, $\eps >0$, compactly
converge  to a domain   $\Omega$ as $\eps \to 0$ if for any
compact set $K\subset \Omega \cup ({\R^N\setminus \overline \Omega})$
there exists $\eps_{K}>0$ such that $K\subset \Omega_{\eps} \cup
({\R^N\setminus \overline \Omega_{\eps }})$ for all $\eps \in
]0,\eps_K[$. Moreover, we say that $\Omega $ is stable  if any
function in $\ H^1(\R^N)$ vanishing almost everywhere on
$\R^N\setminus \overline \Omega $ belongs to $H^1_0(\Omega )$.
\end{definition}

It is well-known that if $\Omega$ is of class $C^{0,1}$ then
$\Omega $ is stable. We refer to  \cite{BucBut} for more
information.  Then we have the following stability result. Here we denote  by $v_0$  the extension-by-zero of a function $v$ outside its natural domain of definition.

\begin{theorem}
\label{stabnav} Let $D$ be a fixed bounded set and let
$\Omega_{\eps}$, $\eps >0$, and $\Omega$ be domains  contained in
$D$. Assume that the sets  $\Omega_{\eps }$ compactly converge to
$\Omega $ as $\eps \to 0$ and that $\Omega $ is stable. If the
data $f_{\eps}$ converge to $f$ in the sense that
\begin{equation}
\label{stabnav1} \| f_{\eps}-f_0\|_{L^2(\Omega_{\eps } )}, \ \ \|
\nabla f_{\eps}-(\nabla f)_0\|_{L^2(\Omega_{\eps }) } \to 0,\ \
{\rm as }\ \eps \to 0
\end{equation}
 then the solutions $u_{\eps }$ of problems \eqref{classicalnaveps}  converge to the solution $u$ of problem  \eqref{classicalnav} in the sense that
\begin{equation}\label{stabnav2}
\| u_{\eps}-u_0\|_{L^2(\Omega_{\eps } )}, \ \ \| \nabla
u_{\eps}-(\nabla u)_0\|_{L^2(\Omega_{\eps } )}, \ \ \| \Delta
u_{\eps }- (\Delta u)_0\|_{L^2(\Omega_{\eps } )} \to 0 ,\ \ {\rm
as }\ \eps \to 0 .
\end{equation}

\end{theorem}

\begin{proof} By assumption, it follows that
\begin{equation}\label{stabnav3}
\| f_{\eps }\|_{L^2(\Omega_{\eps })}, \ \| \nabla f_{\eps
}\|_{L^2(\Omega_{\eps })} \le C\, ,
\end{equation}
for all $\eps >0$ sufficiently small.  Using $u_{\eps}$ as a test
function in (\ref{weaknav}) we have
\begin{equation}\label{stabnav35}
\int_{\Omega_{\eps} }|\Delta u_{\eps }|^2 dx = \int_{\Omega_{\eps}
} (f_{\eps}\Delta u_{\eps }+\nabla f_{\eps }\nabla u_{\eps })\,
dx,
\end{equation}
which, combined with (\ref{poin}) and (\ref{stabnav3}), allows to
conclude that
\begin{equation}\label{stabnav4}
\| u_{\eps}\|_{L^2(\Omega_{\eps })}, \| \nabla
u_{\eps}\|_{L^2(\Omega_{\eps })}, \| \Delta
u_{\eps}\|_{L^2(\Omega_{\eps })}\le C,
\end{equation}
for all $\eps >0$ sufficiently small. Since $D$ is bounded, by the Rellich-Kondrakov Theorem, $H^1_0(D)$ is compactly embedded in $L^2(D)$. This combined with the fact that $(u_\eps)_0\in H^1_0(D)$ allows to conclude that there exists $u\in H^1_0(D)$ such that
possibly passing to a subsequence
\begin{equation}\label{stabnav5}
(u_{\eps})_0\to u, \ \ {\rm in}\ \ L^2(D),\ \ {\rm and}\ \
(u_{\eps})_0\  \rightharpoonup  u, \ \ {\rm in}\ \ H^1(D)
\end{equation}
as $\eps \to 0$.  By the compact convergence  of the domains
$\Omega_{\eps} $ to the set  $\Omega $ it follows that $u=0$ almost everywhere  on
$D\setminus \overline \Omega$. Since $\Omega $ is stable, this implies
that $u\in H^1_0(\Omega )$ and that $u=0$  almost everywhere on $D\setminus \Omega$, see \cite[p.~117]{BucBut}. By \eqref{stabnav4} it follows that
possibly passing to a subsequence there exists $v\in L^2(D )$ such
that $(\Delta u_{\eps })_0 \rightharpoonup v$ in $L^2(D )$. Let
$\psi\in C^{\infty }_c(\Omega )$ be  fixed.  By the compact
convergence  of the domains  $\Omega_{\eps} $ to the set $\Omega
$ it follows that for all $\eps
>0$ sufficiently small, $\psi \in C^{\infty }_c(\Omega_{\eps} )$ hence
$\int_{\Omega_{\eps }} u_{\eps }\Delta \psi \, dx
=\int_{\Omega_{\eps }} \Delta u_{\eps } \psi \, dx$ and
$\int_{\Omega } u_{\eps }\Delta \psi \, dx=\int_{\Omega } \Delta
u_{\eps } \psi \, dx$. Passing to the limit in the last equality
as $\eps \to 0$, it follows that  $\int_{\Omega } u \Delta \psi
dx=\int_{\Omega } v \psi dx$ which, by the arbitrary choice of
$\psi$,  implies that $u\in H(\Delta, \Omega )\cap H^1_0(\Omega )$
and $\Delta u=v  $. In particular
\begin{equation}\label{stabnav6}
((\Delta u_{\eps})_0)_{|_{\Omega}}\rightharpoonup  \Delta u, \ \
{\rm in}\ \ L^2(\Omega ),
\end{equation}
as $\eps \to 0$.

Let $\varphi \in H(\Delta, \Omega )\cap H^1_0(\Omega )$ be fixed.
For every $\eps >0$ we consider the function $\varphi_{\eps}\in
 H(\Delta , \Omega_{\eps })\cap H^1_0(\Omega_{\eps} )  $ uniquely
defined as the solution to the problem
\begin{equation}
\label{stabnav7} \left\{
\begin{array}{ll}
\Delta \varphi_{\eps }=(\Delta \varphi )_0,& \ \ {\rm in }\ \Omega_{\eps},\\
\varphi_\eps =0,& \ \ {\rm on }\ \partial \Omega_{\eps } .
\end{array}
\right.
\end{equation}
This means that $\varphi_{\eps}$ satisfies the weak equation
\begin{equation}
\label{stabnav8} \int_{\Omega _{\eps }}\nabla \varphi_{\eps
}\nabla \eta \, dx=\int_{\Omega_{\eps }}(\Delta \varphi )_0\eta \,
dx,\ \ {\rm for \ all }\ \eta\in H^{1 }_0(\Omega_{\eps }).
\end{equation}

We claim that
\begin{equation}\label{stabnav9}
((\nabla \varphi_{\eps })_0)_{|_{\Omega }} \rightharpoonup \nabla
\varphi , \ \ {\rm in }\ \  L^2(\Omega ),
\end{equation}
 as $\eps \to 0$. Indeed, using $\varphi_{\eps }$ as a test function in (\ref{stabnav8}) and the Poincar\'{e} inequality  we obtain
 \begin{equation}\label{stabnav10}
 \|\nabla \varphi_{\eps } \|_{L^2(\Omega_{\eps})}^2\le  \| \Delta\varphi \|_{L^2(\Omega)} \| \varphi_{\eps } \|_{L^2(\Omega_{\eps})}  \\
 \le
 c_N|\Omega_{\eps }|^{\frac{1}{N}}\| \Delta\varphi \|_{L^2(\Omega)} \| \nabla \varphi_{\eps } \|_{L^2(\Omega_{\eps})}
 \end{equation}
 by which we deduce that $\|\nabla \varphi_{\eps}\|_{L^2(\Omega_{\eps})}$ is uniformly  bounded for   $\eps >0 $ sufficiently small.
 Thus, possibly passing to a subsequence,  it turns out that $(\varphi_{\eps})_0 $ is strongly convergent in $L^2(D )$ and weakly in $H^1(D)$ to a function  $\psi $.
 As above, the compact convergence of $\Omega_{\eps} $ implies that $\psi =0$ on ${\mathbb{R}}^N\setminus\overline \Omega $. Thus, since $\Omega$ is stable, $\psi \in H^1_0(\Omega)$. Fixing now $\eta\in C^{\infty}_c(\Omega)$ and observing that  $\eta\in C^{\infty}_c(\Omega_{\eps })$ for all $\eps >0$ sufficiently small, we can test $\eta$ in \eqref{stabnav8} and pass to the limit to conclude that $\psi =\varphi$ in $\Omega$. Thus   (\ref{stabnav9}) holds.

Before proceeding with using function $\varphi_{\eps}$, we observe
that by assumptions  \eqref{stabnav1}  we have that
\begin{equation}\label{strip}
\int_{\Omega_{\eps}\setminus \Omega }|f_{\eps}|^2 dx, \ \
\int_{\Omega_{\eps}\setminus \Omega }|\nabla f_{\eps}|^2 dx \to 0
\end{equation}
as $\eps \to 0$, and by the compact convergence
of domains we have
\begin{equation}\label{measurein}
|\Omega\setminus \Omega_{\eps}|\to 0 \, , \qquad \text{as } \eps
\to 0 \, .
\end{equation}

Testing $\varphi_{\eps}$  in  the weak formulation of problem
\eqref{classicalnaveps} we  have
\begin{eqnarray}\label{stabnav11}\lefteqn{
\int_{\Omega}(\Delta u_{\eps })_0\Delta\varphi \,
dx=\int_{\Omega\cap\Omega_{\eps }}\Delta u_{\eps }\Delta\varphi \,
dx=
\int_{\Omega\cap\Omega_{\eps }}\Delta u_{\eps }(\Delta\varphi )_0 \, dx  }\nonumber \\
& & \quad =\int_{\Omega_{\eps }}\Delta u_{\eps
}\Delta\varphi_{\eps }dx=\int_{\Omega_{\eps}} (f_{\eps
}\Delta\varphi_{\eps}
+\nabla f_{\eps}\nabla \varphi_{\eps}) \, dx\nonumber \\
& & \quad = \int_{\Omega_{\eps}} (f_{\eps }(\Delta\varphi )_0
+\nabla f_{\eps}\nabla \varphi_{\eps}) \, dx
=\int_{\Omega\cap \Omega_{\eps}} (f_{\eps }\Delta\varphi +\nabla f_{\eps}\nabla \varphi_{\eps}) \, dx\nonumber \\
& &  \quad +\int_{\Omega_{\eps }\setminus \Omega} (\nabla
f_{\eps}\nabla \varphi_{\eps}) \, dx \, .
\end{eqnarray}

Now we note that by the triangle inequality
\begin{eqnarray}\label{stabnav12}\lefteqn{\biggl|
\int_{\Omega_{\eps }\cap\Omega   }\nabla f_{\eps}\nabla
\varphi_{\eps}\, dx-
\int_{\Omega_{\eps }\cap \Omega   }\nabla f \nabla \varphi \, dx \biggr|   } \nonumber \\
& & \le  C \| \nabla f_{\eps}-\nabla f \|_{L^2(\Omega_{\eps
}\cap\Omega) } +\biggl|\int_{\Omega_{\eps }\cap\Omega   } \nabla f
(\nabla \varphi_{\eps }-\nabla \varphi) \, dx\biggr| \, .
\end{eqnarray}

By combining \eqref{stabnav1}, \eqref{stabnav9},    \eqref{strip}
and \eqref{stabnav12}, it follows that
\begin{equation}\label{stabnav14}
\int_{\Omega_{\eps }\cap\Omega   }\nabla f_{\eps}\nabla
\varphi_{\eps}\, dx=  \int_{\Omega_{\eps }\cap \Omega }\nabla f
\nabla \varphi \, dx +o(1),\ \  \int_{\Omega\cap
\Omega_{\eps}}f_{\eps }\Delta\varphi  \, dx = \int_{\Omega\cap
\Omega_{\eps}}f\Delta\varphi \, dx +o(1),
\end{equation}
and  $\int_{\Omega_{\eps }\setminus \Omega   }\nabla
f_{\eps}\nabla \varphi_{\eps}\, dx\to 0$, as $\eps \to 0$. Thus,
 \eqref{stabnav11} can be re-written as follows

\begin{equation}\label{stabnav11bis}
\int_{\Omega}(\Delta u_{\eps })_0\Delta\varphi dx=\int_{\Omega\cap
\Omega_{\eps}} (f \Delta\varphi +\nabla f \nabla \varphi) \, dx
+o(1) \, .
\end{equation}

By passing to the limit in \eqref{stabnav11bis} and using
\eqref{stabnav6}, \eqref{measurein} we conclude that $u$ is the
solution to problem \eqref{classicalnav} as required in the
statement.

Now, the first limit in \eqref{stabnav2} follows by the first limit in \eqref{stabnav5} and by observing that $u$ vanishes almost everywhere outside $\Omega $.

We now  prove that the third limit in \eqref{stabnav2} holds. To
do so, we consider the identity
\begin{eqnarray} \label{stabnav15}
 \| \Delta u_{\eps}- (\Delta u)_0\|^2_{L^2(\Omega_{\eps})}&=& \| \Delta u_{\eps }\|^2_{L^2(\Omega_{\eps})}-2(\Delta u_{\eps },\Delta u)_{L^2(\Omega_{\eps} \cap \Omega )}+ \| \Delta u \|^2_{L^2(\Omega_{\eps }\cap \Omega )}\nonumber\\
& =  &
 \| \Delta u_{\eps }\|^2_{L^2(\Omega_{\eps})}-2(\Delta u_{\eps },\Delta u)_{L^2( \Omega )}+ \| \Delta u \|^2_{L^2(\Omega )} +o(1),
\end{eqnarray}
where the second equality  in \eqref{stabnav15} is deduced by
\eqref{measurein}.

We note that arguing as above we can prove that
\begin{equation}\label{stabnav1505}
\int_{\Omega_{\eps} } (f_{\eps}\Delta u_{\eps }+\nabla f_{\eps
}\nabla u_{\eps })\, dx\to \int_{\Omega} (f \Delta u +\nabla f
\nabla u) \,  dx,
\end{equation}
as $\eps \to 0$. By passing to the limit in \eqref{stabnav35} as
$\eps \to 0$, and by \eqref{stabnav1505}, we deduce that
\begin{equation}\label{stabnav16}
 \| \Delta u_{\eps }\|_{L^2(\Omega_{\eps})}\to \| \Delta u\|_{L^2(\Omega )},
\end{equation}
as $\eps \to 0$.  Finally, using \eqref{stabnav16} and passing to
the limit in \eqref{stabnav15}  allows to conclude that the third
limit in \eqref{stabnav2} holds.

It remains to prove that the second limit in  \eqref{stabnav2}
holds.  To do so, it is convenient to set $F_{\eps}:=-\Delta
u_{\eps}$ and  $F:=-\Delta u $. Thus, $u_{\eps}$ and $u$ are
solutions to two Poisson problems for the Dirichlet Laplacian in
$\Omega_{\eps}$ and $\Omega$   with data $F_{\eps }$ and $F$
respectively, and $\|F_{\eps}-F_0\|_{L^2(\Omega_{\eps })}\to 0$ as
$\eps \to 0$. Proceeding as above, using the identity
\begin{equation} \label{stabnav17}
\| \nabla u_{\eps}- (\nabla u)_0\|^2_{L^2(\Omega_{\eps})}= \|
\nabla  u_{\eps }\|^2_{L^2(\Omega_{\eps})}-2(\nabla u_{\eps },
\nabla u)_{L^2(\Omega )}+ \| \nabla u \|^2_{L^2(\Omega )}+o(1)
\end{equation}
and exploiting the weak formulations of the Poisson problems for
the Dirichlet Laplacian and the weak convergence of the gradients
of $\nabla u_{\eps}$, we conclude that the second limit in
\eqref{stabnav2}  holds.
\end{proof}

\subsection{Stability of the Navier-to-Neumann map}

It is well-known that if $\Omega $ is a bounded domain  of class
$C^{0,1}$, given a function $u\in  H(\Delta , \Omega )\cap H^1_0(\Omega)$, it is possible
to define the normal derivative $u_{ \nu}$
of $u$ as an element of $H^{-1/2}(\partial\Omega )$ by
setting
\begin{equation}\label{conormal}
_{H^{-1/2}(\partial\Omega )}\left\langle u_{ \nu},
\varphi\right\rangle_{H^{1/2}(\partial\Omega )} =\int_{\Omega }(\Delta u \varphi +\nabla u
\nabla \varphi) \, dx
\end{equation}
for all $\varphi \in H^{1/2}(\partial \Omega )$, where it
is meant that $\varphi $ is extended to the whole of $\Omega$ as
an element of $H^1(\Omega)$. Recall that the trace operator acts
from $H^1(\Omega)$ onto $H^{1/2}(\partial\Omega )$ and
note that definition \eqref{conormal} does not depend on the
specific extension of $\varphi$ since the above integral vanishes
for functions $\varphi$ belonging to $H^1_0(\Omega)$.

Since we aim at studying a domain perturbation problem, we find it
convenient to consider $u_{\nu}$ as an
element of the dual of $H^1(\Omega)$  defined  by formula
\eqref{conormal} for all $\varphi \in H^1(\Omega)$.

We call {\it Navier-to-Neumann map}  the function ${\mathcal
N}_{\Omega}$ from $H^1(\Omega )$ to $(H^1(\Omega))'$ defined by
${\mathcal N}_{\Omega }(f)=u_{\nu}$  where
$u$ is the solution to \eqref{classicalnav}.

We note that if $\Omega $ is sufficiently smooth, say $\Omega$ is a domain of class $C^{1,1}$, $d$ is an
eigenvalue of \eqref{eq:Steklov} with eigenfunction $u\in
V(\Omega)$ if and only
\begin{equation}
\label{veryweak}
\mathcal N_\Omega f=\mu J_0 f
\end{equation}
where  $\mu=\frac 1d$ and  $f$ is any function in $H^1(\Omega)$ such that
$f_{|\partial\Omega}=du_\nu\in H^{1/2}(\partial\Omega)$ and
$J_0:H^1(\Omega)\to (H^1(\Omega))'$ is defined by
$$
_{(H^1(\Omega))'}\langle J_0 f,v\rangle_{H^1(\Omega)}:=\int_{\partial\Omega} fv \, dS \qquad
\text{for any } f,v\in H^1(\Omega) \, .
$$

\begin{remark}\label{veryweakrem}
If one wishes to consider  equation \eqref{veryweak} as an eigenvalue problem, it would be clearly appropriate to formulate it on $H^1(\Omega)/H^1_0(\Omega)$ by considering  ${\mathcal
N}_{\Omega}$  and $J_0$ as functions from  $H^1(\Omega)/H^1_0(\Omega)$  to its dual. Indeed,  any function $f\in H^1_0(\Omega)$ satisfies  \eqref{veryweak} for any $\mu\in {\mathbb{R}}$. However, as we have mentioned in the introduction, using the quotient space would not bypass the main obstructions  to the proof of the spectral stability of the classical DBS problem under weak conditions on the convergence of the domains. For this reason, we prefer
 to work directly on the space $H^1(\Omega)$ to avoid further technical complications.
\end{remark}

In Theorems  \ref{l:EE*}, \ref{t:dual} below we establish   convergence
results for the family of maps ${\mathcal N}_{\Omega_{\eps}}$
associated with a family of domains  $\Omega_{\eps}$ compactly
convergent to $\Omega$. For the sake of simplicity, in order to avoid imposing extra assumptions on $\Omega_\eps$ or further technicalities in the
proofs, we assume that $\Omega\subseteq \Omega_\eps \subset D$ for any $\eps>0$, where $D$ is a fixed bounded domain.

Since ${\mathcal N}_{\Omega_{\eps }} $ takes values in the dual of
$H^1(\Omega_{\eps })$, in order to establish the vicinity of
${\mathcal N}_{\Omega_{\eps }}$ to ${\mathcal N}_{\Omega }$ we
introduce a notion of $E$-convergence for families of functionals
on $H^1$.

\begin{definition}\label{connectingdual} Let $\Omega, \Omega_{\eps} $, with $\eps>0$,  be bounded domains in $\R^N$ with $\Omega $ of class $C^{0,1}$ and
$\Omega \subseteq \Omega_{\eps}$ for all $\eps>0$. Let
 $\mathcal E :H^1(\Omega)\to H^1(\R^N)$ be  a fixed  linear
continuous extension operator.
\begin{itemize}
\item[(i)]
Let $\{f_{\eps}\}_{\eps>0}$ be a family of functions
$f_\eps\in H^1(\Omega_\eps)$. We say that   $f_\eps \EC f$ with $f\in H^1(\Omega)$
 if
$\|f_\eps-\mathcal E f\|_{H^1(\Omega_\eps)}\to 0$ as $\eps\to 0$.
\item[(ii)]
Let $\{\Lambda\}_{\eps>0}$ be a family of functionals
$\Lambda_\eps\in (H^1(\Omega_\eps))'$. We say that $\Lambda_\eps
\ECDUAL \Lambda$ with $\Lambda\in (H^1(\Omega))'$ if
\begin{equation} \label{eq:dual-conv}
\|\Lambda_\eps-E_\eps^*\Lambda\|_{(H^1(\Omega))'}\to 0 \qquad
\text{as } \eps\to 0
\end{equation}
where $E_\eps^*:(H^1(\Omega))'\to (H^1(\Omega_\eps))'$ is the
family of linear continuous operators defined by
$$
_{(H^1(\Omega_\eps))'}\langle
E_\eps^*\Lambda,v\rangle_{H^1(\Omega_\eps)}:=\
_{(H^1(\Omega))'}\langle \Lambda,v_{|\Omega}\rangle_{H^1(\Omega)}
\qquad \text{for any } v\in H^1(\Omega_\eps)
$$
for any $\Lambda\in (H^1(\Omega))'$.
\item[(iii)]    Let
$\{B_\eps\}_{\eps>0}$ be a family of linear continuous operators
such that $B_\eps:H^1(\Omega_\eps) \to (H^1(\Omega_\eps))'$. We
say that $B_\eps \EECDUAL B$ as $\eps\to 0$, with $B\in \mathcal
L(H^1(\Omega);(H^1(\Omega))')$, if $B_\eps f_\eps \ECDUAL Bf$
whenever $f_\eps \EC f$.

\end{itemize}
\end{definition}

\begin{remark}
The family of operators $\{E_\eps^*\}_{\eps>0}$ represents a
``connecting system'' in the sense of \cite[Section 1]{Vainikko}.
In order to see this, one can first show that
$\|E_\eps^*\Lambda\|_{(H^1(\Omega_\eps))'}\le
\|\Lambda\|_{(H^1(\Omega))'}$ and then, by using the extension
operator $\mathcal E $ defined above and the fact that $|\Omega_{\eps}\setminus \Omega |\to 0$ as $\eps \to 0$, one shows that for any
$\sigma>0$ there exists $\overline \eps>0$ such that for any
$\eps\in (0,\overline\eps)$ one has
$\|\Lambda\|_{(H^1(\Omega))'}\le
(1+\sigma)\|E_\eps^*\Lambda\|_{(H^1(\Omega_\eps))'}+\sigma$; the
conclusion follows by letting $\eps\to 0$ and by exploiting the
arbitrariness of $\sigma>0$.
\end{remark}

As a consequence of Theorem \ref{stabnav} we have

\begin{lemma} \label{l:stabnav}
Let $\Omega_{\eps}$, $\Omega$ and $D$ be as in Theorem
\ref{stabnav}. Assume that    $\Omega $ is of class $C^{0,1}$,
$\Omega \subseteq \Omega_{\eps}$ for all $\eps>0$ and       $\Omega_{\eps }$ compactly converges
to $\Omega $ as $\eps \to 0$.
Let $f_\eps\in H^1(\Omega_\eps)$ for any $\eps >0$ and $f\in H^1(\Omega)$.
Let
$u_{\eps}$, $u$ be the solutions of problems \eqref{classicalnaveps},  \eqref{classicalnav} respectively. If
 $f_\eps \EC f$ as $\eps\to 0$ as above then \eqref{stabnav2} holds.
\end{lemma}

\begin{proof} Under the assumptions of the lemma we have that $|\Omega_{\eps}\setminus  \Omega|=|\Omega_{\eps}\setminus \overline \Omega|\to 0$ as $\eps \to 0$ hence
\eqref{stabnav1}  holds true  and Theorem~\ref{stabnav} applies.
\end{proof}

Thanks to Lemma \ref{l:stabnav} we can prove the following

\begin{theorem} \label{l:EE*} Let $\Omega_{\eps}$, $\Omega$ and $D$ be as in Theorem
\ref{stabnav}. Assume that $\Omega $ is of class $C^{0,1}$,
$\Omega \subseteq \Omega_{\eps}$ for all $\eps>0$ and $\Omega_{\eps }$ compactly converges
to $\Omega $ as $\eps \to 0$. Let $\mathcal N_{\Omega_\eps}$ and $\mathcal N_\Omega$ be the
Navier-to-Neumann maps corresponding to $\Omega_\eps$
and $\Omega$ respectively. Then $\mathcal N_{\Omega_\eps} \EECDUAL \mathcal
N_\Omega$ as $\eps\to 0$.
\end{theorem}

\begin{proof} We have to prove that
\begin{equation} \label{eq:dual-conv-2}
\|\mathcal N_{\Omega_\eps}f_\eps-E_\eps^*\mathcal N_\Omega
f\|_{(H^1(\Omega_\eps))'}\to 0 \qquad \text{as } \eps\to 0
\end{equation}
whenever $f_\eps \EC f$.  To do this, we define $u_\eps$ and $u$
respectively as the solutions of \eqref{classicalnaveps} and
\eqref{classicalnav} and we observe that for any $v\in H^1(\Omega_{\eps})$ we have
\begin{align*}
& _{(H^1(\Omega_\eps))'} \langle \mathcal N_{\Omega_\eps}
f_\eps-E_\eps^* \mathcal N_\Omega f,v\rangle_{H^1(\Omega_\eps)}= \
_{(H^1(\Omega_\eps))'} \langle (u_\eps)_\nu,v\rangle_{H^1(\Omega_\eps)}-\ _{(H^1(\Omega))'} \langle u_\nu,v_{|\Omega}\rangle_{H^1(\Omega)} \\
& \qquad =\int_{\Omega_\eps} (\Delta u_\eps \, v+\nabla u_\eps \nabla v)\, dx-\int_{\Omega} (\Delta u \, v+\nabla u \nabla v)\, dx\\
& \qquad =\int_\Omega [(\Delta u_\eps-\Delta u)v+(\nabla
u_\eps-\nabla u)\nabla v]\,dx+\int_{\Omega_\eps\setminus \Omega}
(\Delta u_\eps \, v+\nabla u_\eps \nabla v)\, dx \, .
\end{align*}
This implies
\begin{align} \label{eq:dual-conv-3}
& \left|_{(H^1(\Omega_\eps))'} \langle \mathcal N_{\Omega_\eps}
f_\eps-E_\eps^* \mathcal N_\Omega
f,v\rangle_{H^1(\Omega_\eps)}\right|
\\
\notag & \qquad \le \left[\|\Delta u_\eps-\Delta
u\|_{L^2(\Omega)}+\|\nabla u_\eps-\nabla u\|_{L^2(\Omega)}
+\|\Delta u_\eps\|_{L^2(\Omega_\eps\setminus \Omega)}+\|\nabla
u_\eps\|_{L^2(\Omega_\eps\setminus\Omega)}\right]\,
\|v\|_{H^1(\Omega_\eps)} \, .
\end{align}
The proof of \eqref{eq:dual-conv-2} then follows from Lemma
\ref{l:stabnav}. This completes the proof of the lemma.
\end{proof}

Finally we prove the compact convergence $\mathcal
N_{\Omega_\eps}\CCDUAL \mathcal N_\Omega$, i.e. $\mathcal N_{\Omega_\eps} \EECDUAL \mathcal N_\Omega$ and $\{\mathcal
N_{\Omega_\eps}f_\eps\}_{\eps>0}$ is precompact in the sense of
Definition \ref{d:precompact-ArCaLo} with $\mathcal
H_\eps=(H^1(\Omega_\eps))'$ and $\mathcal H_0=(H^1(\Omega))'$
whenever $\|f_\eps\|_{H^1(\Omega_\eps)}=1$. To do so, we shall require that the following condition (A)  is satisfied. \\

{\it (A): For any family $\{f_{\eps} \}_{0<\eps\le \eps_0}$ of functions $f_{\eps}\in H^1(\Omega_{\eps})$ such that $\| f_{\eps}\|_{H^1(\Omega_{\eps})} $  is uniformly bounded in $\eps$, we have that $\| f_{\eps}\|_{L^2(\Omega_{\eps} \setminus \Omega)  }\to 0$ as $\eps\to 0$.} \\

This condition  has been extensively used by J.M. Arrieta and his co-authors in the analysis of the spectral stability of second order elliptic operators subject to Neumann boundary conditions on dumbbell domains, see  \cite{ArCaLo} for references . We note that in our case, this condition will be used also to control the behaviour of the data $f_{\eps}$ of the problems.

\begin{remark}
 If $\Omega_{\eps}, \Omega$ are of class $C^0({\mathcal{A}})$ for all $\eps$, where ${\mathcal{A}}$ is a fixed atlas, then the compact convergence of $\Omega_{\eps}$ to $\Omega$ is equivalent to the uniform convergence of the functions describing the boundaries of $\Omega_{\eps}$ and that using the  boundedness of functions in Sobolev spaces along lines as in Lemma~\ref{l:1} we have that  the compact convergence implies the validity of condition (A).
\end{remark}

\begin{theorem} \label{t:dual}   Let $\Omega_{\eps}$, $\Omega$ and $D$ be as in Theorem
\ref{stabnav}. Assume that    $\Omega $ is of class $C^{0,1}$,
$\Omega \subseteq \Omega_{\eps}$ for all $\eps>0$,       $\Omega_{\eps }$ compactly converges
to $\Omega $ as $\eps \to 0$, and condition (A) is satisfied.  Let
$\mathcal N_{\Omega_\eps}$ and $\mathcal N_\Omega$ be the
Navier-to-Neumann maps corresponding to $\Omega_\eps$
and $\Omega$ respectively. Then $\mathcal N_{\Omega_\eps} \CCDUAL \mathcal
N_\Omega$ as $\eps\to 0$.
\end{theorem}

\begin{proof} Let $f_{\eps}\in H^1(\Omega_{\eps})$ be such that $\|f_\eps\|_{H^1(\Omega_\eps)}=1$ for all $\eps >0$. Let $u_{\eps}\in H(\Delta, \Omega_{\eps} )\cap H^1_0(\Omega_{\eps})$ be the solutions to problem \eqref{classicalnaveps}.    Proceeding as in the proof of Theorem \ref{stabnav}
one can show that possibly passing to a subsequence  \eqref{stabnav4}, \eqref{stabnav5}, \eqref{stabnav6} hold true for some $ u \in H^1(D)$ such that $u_{|\Omega}\in H(\Delta, \Omega )\cap H^1_0(\Omega)$, $u $ vanishing outside $\Omega$.
 Writing
 \begin{equation}\label{t:dual0}
 \int_{\Omega_\eps}
|\nabla u_\eps|^2 dx=\int_{\Omega_\eps} -u_\eps \Delta u_\eps \,
dx =\int_{\Omega } -u_\eps \Delta u_\eps \,
dx + \int_{\Omega_\eps\setminus \Omega } -u_\eps \Delta u_\eps \,
dx ,
 \end{equation}
 and  passing to the limit as $\eps \to 0$, we have by \eqref{stabnav4}, \eqref{stabnav5}, \eqref{stabnav6} and condition $(A)$
 \begin{equation} \label{t:dual1}
 \lim_{\eps\to 0} \int_{\Omega_\eps}
|\nabla u_\eps|^2 dx= \lim_{\eps\to0} \int_{\Omega } -u_\eps \Delta u_\eps\, dx =  \int_{\Omega } -u \Delta u \, dx= \int_{\Omega}|\nabla u|^2dx.
 \end{equation}
 Moreover,
 \begin{equation}\label{t:dual2}
 \int_{\Omega}|\nabla u|^2dx\le \liminf_{\eps\to 0 } \int_{\Omega}|\nabla u_{\eps}|^2dx\le \limsup_{\eps\to 0}\int_{\Omega}|\nabla u_{\eps}|^2dx\le
  \lim_{\eps\to 0}\int_{\Omega_{\eps}}|\nabla u_{\eps}|^2dx= \int_{\Omega}|\nabla u|^2dx
 \end{equation}
 hence
\begin{equation}\label{t:dual3}
\int_{\Omega}|\nabla u_{\eps}|^2dx\to \int_{\Omega}|\nabla u|^2dx, \ \ \int_{\Omega_{\eps} \setminus \Omega }|\nabla u_{\eps}|^2dx\to 0, \ \ (u_\eps)_{|\Omega}\to u \quad \text{in } H^1(\Omega),
\end{equation}
as $\eps\to 0$.

 By our assumptions, in particular by condition $(A)$, we have that possibly passing to a subsequence there exists $f\in H^1(\Omega )$ such that
\begin{align} \label{eq:facts-3}
 f_\eps \to f \quad \text{in } L^2(\Omega ) \, , \qquad
 f_\eps \rightharpoonup f \quad \text{in } H^1(\Omega )
\, , \qquad \|f_\eps\|_{L^2(\Omega_\eps\setminus\Omega)}\to 0
\end{align}
as $\eps\to 0$.
We claim that function $u$ is a solution to problem \eqref{classicalnav} with this  datum $f$. To prove this, one can use the same argument
as in the proof of Theorem~\ref{stabnav}. Here we cannot use the strong convergence of the gradients of $f_{\eps}$, not even the second limit in \eqref{strip}. However, by looking closely at the proof of Theorem~\ref{stabnav}, one can see that here it suffices to use the fact that the family of functions $\varphi_{\eps}$ used in that proof satisfies the conditions $\varphi_{\eps}\to \varphi $ in $H^1(\Omega)$ and $\| \nabla \varphi_{\eps}\|_{L^2(\Omega_{\eps}\setminus \Omega )}\to 0$
as $\eps\to 0$, which can be proved by proceeding exactly in the same way as in \eqref{t:dual0}-\eqref{t:dual3}.

By  \eqref{stabnav4}, \eqref{stabnav6}, \eqref{t:dual3}, \eqref{eq:facts-3}, we conclude that
\begin{align*}
& \int_{\Omega_\eps} |\Delta u_\eps|^2 dx=\int_{\Omega} (f_\eps
\Delta u_\eps+\nabla f_\eps\nabla u_\eps) \,
dx+\int_{\Omega_\eps\setminus\Omega} (f_\eps \Delta u_\eps+\nabla
f_\eps\nabla u_\eps) \, dx \\
& \qquad =\int_{\Omega} (f \Delta  u+\nabla f\nabla  u) \, dx+o(1)=\int_\Omega |\Delta u|^2
dx+o(1).
\end{align*}
 In particular,  proceeding as we have done above for $\nabla u_{\eps}$,  we have that
\begin{equation} \label{eq:facts-4}
(\Delta u_\eps)_{|\Omega} \to \Delta u \quad \text{in }
L^2(\Omega) \, , \qquad \int_{\Omega_\eps\setminus\Omega} |\Delta
u_\eps|^2 dx \to 0
\end{equation}
as $\eps\to 0$. Inserting \eqref{t:dual3} and  \eqref{eq:facts-4} into
\eqref{eq:dual-conv-3} we obtain $\mathcal N_{\Omega_\eps} f_\eps
\ECDUAL \mathcal N_\Omega f$ as $\eps\to 0$ along a sequence. We
have proved that $\{\mathcal N_{\Omega_\eps} f_\eps\}$ is
precompact in the sense of Definition \ref{d:precompact-ArCaLo}.
The proof of the theorem now follows combining this with  Theorem~\ref{l:EE*}.
\end{proof}

Unfortunately, as we have mentioned in the introduction the compact convergence result obtained in Theorem
\ref{t:dual} is not sufficient to prove a stability result for
the spectrum of the Steklov problem \eqref{eq:Steklov}.  Thus  we suggest the following open
problem
\begin{open} \label{o:1} Let $\{\Omega_\eps\}_{0<\eps\le \eps_0}$ and $\Omega$ be such that
condition \eqref{eq:assumptions} is not satisfied. For simplicity
one may think to a family of domains as in Section
\ref{s:optimality} and choose $1\le \alpha\le \frac32$ is such a
way that \eqref{eq:assumptions} is not satisfied. Study the
stability of the spectrum of \eqref{eq:Steklov}.
\end{open}

We recall that, in a situation like the one proposed in Open
Problem \ref{o:1}, we proved that the spectrum of the modified
Steklov problem \eqref{eq:Steklov-modificato} does not behave
continuously as $\eps\to 0$. The question proposed in Open Problem
\ref{o:1} aims at clarifying what happens to the spectrum of \eqref{eq:Steklov}  under the very same
assumptions.

\subsection{Instability of a modified Navier problem }\label{instabsec}

Given a sufficiently smooth bounded domain  $\Omega $   in
${\mathbb{R}}^N$  we consider the boundary value problem
\begin{equation}
\label{classicalhes} \left\{
\begin{array}{ll}
\Delta^2u=0,& \ {\rm in}\ \Omega, \\
u=0,& \  {\rm on }\ \partial \Omega, \\
\Delta u-K(x)u_\nu=f, & \  {\rm on }\ \partial \Omega ,
\end{array}
\right.
\end{equation}
in the unknown $u$. The natural way of writing the weak
formulation  of \eqref{classicalhes} is
\begin{equation}\label{weakhes0}
\int_{\Omega }D^2 u : D^2  \varphi \, dx = \int_{\partial \Omega
}f \varphi_\nu \, dS,\ \ {\rm for\ all}\ \varphi \in
H^2(\Omega)\cap H^1_0(\Omega ),
\end{equation}
where the unknown $u$ has to be considered in $H^2(\Omega)\cap
H^1_0(\Omega )$. This formulation allows to consider the datum $f$
in $L^2(\partial\Omega)$. However, in order to emphasize the peculiar
behaviour of this problem with respect to problem \eqref{weaknav},
we prefer to write \eqref{weakhes0} in a form similar to
\eqref{weaknav}. Thus we assume that $f$ is the trace of function
$f\in H^1(\Omega )$ and rewrite  \eqref{weakhes0} as
\begin{equation}\label{weakhes}
\int_{\Omega }D^2 u : D^2  \varphi \, dx = \int_{\Omega } (f\Delta
\varphi +\nabla f\nabla \varphi)\, dx,\ \ {\rm for\ all}\ \varphi
\in H^2(\Omega)\cap H^1_0(\Omega ) \, .
\end{equation}
As in the previous section, we now consider a family of domains
$\Omega_{\eps }$, $\eps >0$   and the corresponding boundary value
problems
\begin{equation}
\label{classicalheseps} \left\{
\begin{array}{ll}
\Delta^2u_{\eps }=0,& \ {\rm in}\ \Omega_{\eps }, \\
u_{\eps }=0,& \  {\rm on }\ \partial \Omega_{\eps }, \\
\Delta u_\eps-K_\eps(x)(u_\eps)_\nu=f_{\eps }, & \ {\rm on }\
\partial \Omega_{\eps } ,
\end{array}
\right.
\end{equation}
in the unknown $u_{\eps }\in H^2(\Omega_{\eps })\cap H^1_0(\Omega
)$, where $f_{\eps }\in H^1(\Omega_{\eps })$, with the
understanding that
 problem \eqref{classicalheseps} has to be interpreted  as in \eqref{weakhes}.\\

We assume that  $\Omega$ and $\Omega_\eps$ are as in Section
\ref{s:optimality}. As in \cite{ArLa2}, the limiting behaviour of
problem \eqref{classicalheseps} depends on the value of $\alpha$.
In particular, we have stability for $\alpha>3/2$ and degeneration
for $\alpha<3/2$. For $\alpha =3/2$, the limiting problem involves
the strange factor $\gamma$ defined by \eqref{strangecurvature}.

\begin{remark} Theorem \ref{trico} below and Theorem \ref{stabnav} show that problem \eqref{classicalnav} is much more stable than problem \eqref{classicalhes} under domain perturbations. Note that the main step in the proof of Theorem~\ref{stabnav}, is the definition of  the test function $\varphi_{\eps}$ in \eqref{stabnav7}.
Unless the convergence of the domains $\Omega_{\eps}$ to $\Omega$ is stronger, this function cannot be used in the analysis of problem \eqref{classicalhes}, where more information on the Hessian of $\varphi_{\eps}$ would be required.
\end{remark}

\begin{theorem}\label{trico}
Assume that  $\Omega $ and $\Omega_{\eps}$, $\eps >0$, are  as in Section
\ref{s:optimality}, and that the datum  $f_{\eps}$ converges to $f$ in the
sense of \eqref{stabnav1}.  Let  $u_{\eps}$  be the solution of
problem \eqref{classicalheseps}. Then the following statements
hold:
\begin{itemize}
\item[(i)] If $\alpha >3/2$ then  $u_{\eps}$ converges to the
solution of problem \eqref{classicalhes} in the sense that for
every multi-index $\beta\in {\mathbb{N}}_0^N$ with $0\le |\beta
|\le 2$ we have $ \| D^{\beta }u_{\eps}
-(D^{\beta}u)_0\|_{L^2(\Omega_\eps)}\to 0, $ as $\eps \to 0$.
\item[(ii)] If $\alpha =3/2$ then  $u_{\eps}$ converges to the
solution $ u$ of the problem
\begin{equation}
\label{classicalhesstrange} \left\{
\begin{array}{ll}
\Delta^2 u=0,& \ {\rm in}\ \Omega, \vspace{1mm}\\
 u=0\ {\rm and}\ \Delta u-K(x)u_\nu=f, & \  {\rm on }\ \partial \Omega \setminus \Gamma, \vspace{1mm}\\
 u=0\ {\rm and}\ \Delta u-K(x)u_\nu+\gamma u_\nu=f , & \  {\rm on }\ \Gamma ,
\end{array}
\right.
\end{equation}
in the sense that for every multi-index $\beta\in
{\mathbb{N}}_0^N$ with $0\le |\beta |\le 1$ we have $ \| D^{\beta
}u_{\eps} -(D^{\beta}u)_0\|_{L^2(\Omega_\eps)}\to 0, $ and $
u_{\eps}\rightharpoonup u$ weakly in $H^2(\Omega)$ as $\eps \to
0$.

\item[(iii)] If $0<\alpha <3/2$ then $u_{\eps}$ converges to the
solution $ u$ of the problem
\begin{equation}
\label{classicalhesdir} \left\{
\begin{array}{ll}
\Delta^2 u=0,& \ {\rm in}\ \Omega, \vspace{1 mm}\\
 u=0\ {\rm and}\ \Delta u-K(x)u_\nu=f, & \  {\rm on }\ \partial \Omega \setminus \Gamma, \vspace{1 mm}\\
 u=0\ {\rm and}\ u_\nu=0, & \  {\rm on }\ \Gamma ,
\end{array}
\right.
\end{equation}
in the sense that for every multi-index $\beta\in {\mathbb{N}}_0^N$
with $0\le |\beta |\le 2$ we have $ \| D^{\beta }u_{\eps}
-(D^{\beta}u)_0\|_{L^2(\Omega_\eps)}\to 0, $ as $\eps \to 0$.
\end{itemize}
\end{theorem}

\begin{proof}  In each of the three cases, the solution $u_{\eps }\in H^2(\Omega_\eps)\cap H^1_0(\Omega_\eps)$ of problem \eqref{classicalheseps} satisfies
  the weak equation
\begin{equation}\label{weakheseps}
\int_{\Omega_{\eps}}D^2 u_{\eps} : D^2  \varphi \, dx=
\int_{\Omega_{\eps}} (f_{\eps }\Delta \varphi +\nabla f_{\eps
}\nabla \varphi) \, dx,
\end{equation}
 for all  $ \varphi \in H^2(\Omega_{\eps })\cap H^1_0(\Omega_{\eps})$.
Thus, by using $u_{\eps}$ as test function in \eqref{weakheseps}
and arguing as in the proof of Theorem~\ref{stabnav}, we get that
$\|u_{\eps}\|_{H^2(\Omega_{\eps})}$ is uniformly bounded, hence
there exists $u\in H^2(\Omega)\cap H^1_0(\Omega)$ such that
possibly passing to a subsequence, $u_{\eps }\rightharpoonup u$ in
$H^2(\Omega )$ and $u_\eps\to u$ in $H^1(\Omega)$. We have now to
identify the limiting problem satisfied by $u$ and this depends on
the value of $\alpha$.

We begin with statement (i). For every $\eps>0$ sufficiently
small, say $\eps\in ]0,\eps_0[$, we consider a diffeomorphism
$\Phi_{\eps }$ from $\overline \Omega_{\eps}$ onto $\overline \Omega$
defined exactly as the diffeomorphism ${\Phi_{\eps , j}}$
introduced in   Subsection \ref{subsecoperatorsE}. To do so, for
fixed $\tilde\alpha\in ]3/2,\alpha [$, we set
$ \kappa_{\eps }=\eps ^{2\tilde \alpha/3}$ and we note that condition
\eqref{eq:assumptions} is satisfied for $\eps_0 $ sufficiently
small. Then, $\Phi_{\eps}$ is defined by $\Phi_\eps (x',x_N)=(x',
x_N-h_\eps(x',x_N))$ for all $(x', x_N)\in \overline\Omega_{\eps}$
where $h_{\eps}$ is defined as $h_{\eps,j}$ by formula
\eqref{eq:def-h} with $g_{\eps,j}$ replaced by $g_{\eps}$, $g_j$
replaced by the zero function, and $a_{N,j}$ replace by $-1$.  It
is also convenient to recall that $\Phi_{\eps}(x',x_N)=(x',x_N)$
for all $(x',x_N)\in K_{\eps}$ where $K_{\eps}=\{(x',x_N)\in
W\times \R:\ -1<x_N<\tilde g_{\eps}(x',x_N)\}$ and $\tilde
g_{\eps}=-k_{\eps}=-\hat k \kappa_{\eps}$ with $\hat k>6$ as in
Subsection \ref{subsecoperatorsE}.

Let $\varphi\in  H^2(\Omega)\cap H^1_0(\Omega)$. We set
$\varphi_{\eps}=\varphi\circ \Phi_{\eps}$ and we note that
$\varphi_{\eps}\in  H^2(\Omega_{\eps })\cap H^1_0(\Omega_{\eps}
)$, hence it can be used as test function in  \eqref{weakheseps}.
Thus, since  $\Phi_{\eps}$ coincides with  the identity on
$K_{\eps }$ we have

\begin{eqnarray}
\label{trico1} \lefteqn{
\int_{K_{\eps} }D^2 u_{\eps} : D^2  \varphi \, dx+\int_{\Omega_{\eps}\setminus K_{\eps} }D^2 u_{\eps} : D^2  \varphi_{\eps} \, dx}\nonumber \\
& & \qquad =\int_{\Omega_{\eps} }D^2 u_{\eps} : D^2 \varphi_\eps
\, dx
= \int_{\Omega_{\eps}} (f_{\eps }\Delta \varphi_\eps +\nabla f_{\eps }\nabla \varphi_\eps) \, dx \nonumber\\
& & \qquad =\int_{K_{\eps}} (f_{\eps }\Delta \varphi +\nabla
f_{\eps }\nabla \varphi) \, dx+ \int_{\Omega_\eps \setminus K_{\eps }}
(f_{\eps }\Delta \varphi_{\eps }+\nabla f_{\eps }\nabla
\varphi_{\eps}) \, dx\, .
\end{eqnarray}

By Lemma~\ref{l:1} we have that  \eqref{eq:conv-norm-D^2} holds hence
$\int_{\Omega_{\eps}\setminus K_{\eps} }|D^2 \varphi_{\eps}|^2 dx\to 0$
 and $\int_{\Omega_{\eps}\setminus K_{\eps} }D^2 u_{\eps} : D^2  \varphi_{\eps} dx\to 0$ as $\eps\to 0$. By a similar argument, one can also prove
 that $ \int_{\Omega_\eps \setminus K_{\eps }} (f_{\eps }\Delta \varphi_{\eps } +\nabla f_{\eps }\nabla \varphi_{\eps}) \, dx\to 0$ as $\eps \to 0$
 (in which case one should use a limiting relation for the norms of the gradients analogous to \eqref{eq:conv-norm-D^2}, see the proof of Theorem \ref{t:dual} for more details).

Thus, passing to the limit in \eqref{trico1} we obtain
 \begin{equation}
\int_{\Omega  }D^2 u  : D^2  \varphi \, dx = \int_{\Omega}
(f\Delta \varphi +\nabla f \nabla \varphi) \, dx \, .
\end{equation}

Using the same argument as in the proof of Theorem~\ref{stabnav}
with identity \eqref{stabnav15} replaced by
\begin{equation} \label{stabnav15bis}
\| D^2 u_{\eps}- (D^2u)_0\|^2_{L^2(\Omega_{\eps})}= \| D^2 u_{\eps
}\|^2_{L^2(\Omega_{\eps})}
-2 \int_\Omega D^2 u_{\eps }:D^2 u \, dx
+\| D^2 u \|^2_{L^2(\Omega )}\, ,
\end{equation}
we deduce  that $\| D^2u_{\eps } -(D^2u)_0\|_{L^2(\Omega_\eps)} \to 0$
as $\eps \to 0$. In order to conclude, it suffices to use the
strong convergence in $H^1(\Omega)$ stated above and to note that
for any multi-index $|\beta |\le 1 $ we have
$\|D^{\beta}u_{\eps}\|_{L^2(\Omega_{\eps} \setminus\Omega)}\to 0$
as $\eps \to 0$ which follows from the fact that for $|\beta |\le
1 $  and  for almost all $x'\in W$ we have that
$D^{\beta}u_{\eps}(x',x_N) $ is bounded in $x_N$, see
\eqref{eq:o(1)-2} for more details.

We now prove statement (ii).  In this case, we need a different
diffeomorphism from $\overline\Omega_{\eps }$ onto $\overline\Omega$, which
for convenience we denote with the same symbol $\Phi_{\eps}$.
Namely,  $ \Phi_{\eps }(x',x_N)=(x',x_N-h_{\eps}(x',x_N))$ for all
$(x',x_N)\in \overline\Omega_{\eps}$ where  $h_{\eps}$ is defined as in \eqref{accatris}.

It is convenient to denote again by $K_{\eps}$ the set where
$\Phi_{\eps}$ coincides with the identity, that is $K_{\eps} =\{
(x',x_N)\in W\times \R:\ -1<x_N\le-\eps \}$. Let $\varphi\in
H^2(\Omega)\cap H^1_0(\Omega)$. We set $
\varphi_{\eps}=\varphi\circ\Phi_{\eps}$ and we note that
$\varphi_{\eps}\in  H^2(\Omega_{\eps })\cap H^1_0(\Omega_{\eps}
)$, hence it can be used as test function in  \eqref{weakheseps}.
Clearly, equality \eqref{trico1} holds.
Concerning the second term in the right-hand side of
\eqref{trico1}, we note that
\begin{equation}
\label{trico3} \biggl|\int_{\Omega_\eps \setminus K_{\eps }}f_{\eps
}\Delta \varphi_{\eps }\, dx \biggr| \le \| f_{\eps }
\|_{L^2(\Omega_\eps\setminus K_{\eps })}  \| \Delta \varphi_{\eps
}\|_{L^2(\Omega_\eps\setminus K_{\eps })}=o(1)  \| \Delta
\varphi_{\eps }\|_{L^2(\Omega_\eps\setminus K_{\eps })}
\end{equation}
and
\begin{equation}
\label{trico3bis} \bigg|\int_{\Omega_\eps \setminus K_{\eps }}\nabla
f_{\eps }\nabla \varphi_{\eps } dx\biggr| \le \| \nabla f_{\eps }
\|_{L^2(\Omega_\eps \setminus K_{\eps })}  \|\nabla \varphi_{\eps
}\|_{L^2(\Omega_\eps \setminus K_{\eps })} = o(1) \| \nabla
\varphi_{\eps }\|_{L^2(\Omega_\eps \setminus K_{\eps })}
\end{equation}
where we have used \eqref{strip}. On the other hand, one can see
that
\begin{equation}\label{trico4}
  \| D^2 \varphi_{\eps }\|_{L^2(\Omega_\eps\setminus K_{\eps })} =O(1), \ \ {\rm and }\ \
  \| \nabla \varphi_{\eps }\|_{L^2(\Omega_\eps\setminus K_{\eps })}  = o(1)\, .
  \end{equation}

Indeed,  by the chain rule it follows that

\begin{align} \label{eq:compute-biseps}
\frac{\partial^2 \varphi_{\eps }}{\partial x_n \partial
x_m}(x)&=\sum_{k=1}^N
 \left[ \sum_{l=1}^N \frac{\partial^2 \varphi }{\partial x_l
\partial x_k}(\Phi_{\eps }(x))
\frac{\partial[(   \Phi_{\eps }   (x))_l]}{\partial x_m} \right]
 \frac{\partial[(\Phi_{\eps }(x))_k]}{\partial x_n} \\[7pt]
\notag  & +\frac{\partial \varphi }{\partial x_k}(\Phi_{\eps }(x))
 \frac{\partial^2[(\Phi_{\eps }(x))_k]}
 {\partial x_n\partial x_m}  \, .
\end{align}
Now the $L^2$-norm of the first term in the right-hand side of
\eqref{eq:compute-biseps} is clearly $O(\eps ^{2\alpha -2})$ hence
it is infinitesimal for $\alpha=3/2$ as $\eps \to 0$. As for the
second term, we note that $\frac{\partial^2[(\Phi_{\eps }(x))_k]}
 {\partial x_n\partial x_m}=O(\eps^{\alpha -2})$ while  $\| \frac{\partial \varphi _{\eps }}{\partial
x_k}(\Phi_{\eps }(x))\|_{L^2(\Omega\setminus K_{\eps })}=O(\eps
^{1/2})$ which simply follows by the fact that for almost all
$x'\in W$,  $\frac{\partial \varphi _{\eps }}{\partial
x_k}(x',x_N)$ is bounded in the variable $x_N$. Thus, for $\alpha
=3/2$ we have that the $L^2$-norm of the second term in the
right-hand side of \eqref{eq:compute-biseps} is $O(1)$ as $\eps
\to 0$, hence the first equality in \eqref{trico4} is proved. The
second equality in \eqref{trico4} can be proved in the same way.
The proof of statement (ii) follows by passing to the limit in
\eqref{trico1} as $\eps\to 0$ and using \eqref{trico2bis},
\eqref{trico4}.

We now prove statement (iii). It follows from the proof of
\cite[Theorem~7.4]{ArLa1} that since $\alpha <3/2$ the  function
$u$ belongs to $H^2_{0,\Gamma }(\Omega)\cap H^1_0(\Omega)$ where
$H^2_{0,\Gamma}(\Omega)=\{u\in H^2(\Omega): \  u=u_{\nu } =0 \ {\rm
on}\  \Gamma\}$, see the proof of Lemma~\ref{l:E-conv-bis} for more details. Let  $\varphi \in H^2_{0,\Gamma }(\Omega)$. Since $\varphi
=\varphi_\nu=0$ on $\Gamma$, the extension-by-zero $\varphi_0$ of
$\varphi $ belongs to $H^2(\Omega_{\eps })\cap H^1_0(\Omega_{\eps
} )$, hence it can uses ad test function in \eqref{weakheseps} to
obtain
\begin{equation}\label{trico2-bis}
\int_{\Omega }D^2 u_{\eps} : D^2  \varphi \, dx =\int_{\Omega_{\eps} }D^2 u_{\eps} : D^2  \varphi_0 \, dx
=\int_{\Omega_{\eps }} (f_{\eps }\Delta \varphi_0 +\nabla f_{\eps }\nabla \varphi_0)\, dx
=\int_{\Omega} (f_{\eps} \Delta \varphi +\nabla f_{\eps }\nabla \varphi) \, dx \, . \\
\end{equation}

Passing to the limit in \eqref{trico2-bis} as $\eps\to 0$ we
conclude that
 $u$ is the solution to problem \eqref{classicalhesdir}.
 The proof of the convergence of $u_{\eps}$ to $u$ in  the $H^2$ norms follows the same lines of the analogous proof in statement (i).

\end{proof}




\bigskip

{\bf Acknowledgments} The authors are grateful to the anonymous referee for the very careful reading of the paper and for the many useful comments and suggestions which helped them to improve the presentation of the paper, and for bringing to their attention item \cite{Post} in the references list.
The authors are also very thankful to Professors Jos\'{e} M. Arrieta and Giles Auchmuty for useful discussions and references. The authors are members of the Gruppo Nazionale per l'Analisi Matematica, la Probabilit\`{a} e le loro Applicazioni (GNAMPA) of the Istituto Nazionale di Alta Matematica (INdAM).  The first author acknowledges partial financial support from the PRIN project 2012 ``Equazioni alle derivate parziali di tipo ellittico e parabolico: aspetti
geometrici, disuguaglianze collegate, e applicazioni''.
The first author acknowledges partial financial support from the INDAM -
GNAMPA project 2017 ``Stabilit\`{a} e
analisi spettrale per problemi alle derivate parziali''.
The second author acknowledges partial financial support from the INDAM -
GNAMPA project 2017 ``Equazioni alle derivate parziali non
lineari e disuguaglianze funzionali: aspetti geometrici ed
analitici''. This research was partially supported by the research project CPDA120171/12
 ``Singular perturbation problems for
differential operators'' Progetto di
Ateneo of the University of Padova and by the research project ``Metodi analitici, numerici e di simulazione per lo studio di equazioni
differenziali a derivate parziali e applicazioni'' Progetto di Ateneo 2016 of the University of Piemonte Orientale ``Amedeo Avogadro''.

\end{document}